\documentclass[acmsmall]{acmart}
\AtBeginDocument{%
  \providecommand\BibTeX{{%
    \normalfont B\kern-0.5em{\scshape i\kern-0.25em b}\kern-0.8em\TeX}}}

\setcopyright{acmcopyright}
\copyrightyear{2018}
\acmYear{2018}
\acmDOI{10.1145/1122445.1122456}

\acmJournal{TOCL}
\acmVolume{37}
\acmNumber{4}
\acmArticle{111}
\acmMonth{8}

\usepackage{tikz}

\def\righttail{\mathrel{%
  \makebox[.2pt][l]{$\righttailhelper$}\righttailhelper\mkern-4mu-}}
\def\righttailhelper{\scalebox{.5}[1]{$\succ$}}

\newcommand{\Tfrac}[2]{%
  \mathchoice
    {\ooalign{%
      $\genfrac{}{}{1.2pt}0{#1}{#2}$\cr%
      $\color{white}\genfrac{}{}{.4pt}0{\phantom{#1}}{\phantom{#2}}$}}%
    {\ooalign{%
      $\genfrac{}{}{1.2pt}1{#1}{#2}$\cr%
      $\color{white}\genfrac{}{}{.4pt}1{\phantom{#1}}{\phantom{#2}}$}}%
    {\ooalign{%
      $\genfrac{}{}{1.2pt}2{#1}{#2}$\cr%
      $\color{white}\genfrac{}{}{.4pt}2{\phantom{#1}}{\phantom{#2}}$}}%
    {\ooalign{%
      $\genfrac{}{}{1.2pt}3{#1}{#2}$\cr%
      $\color{white}\genfrac{}{}{.4pt}3{\phantom{#1}}{\phantom{#2}}$}}%
}

\newtheorem{theorem}{Theorem}[section]
 \newtheorem{lemma}[theorem]{Lemma}
 \newtheorem{prop}[theorem]{Proposition}
 
 \newtheorem{cor}[theorem]{Corollary}

 \theoremstyle{definition}
 \newtheorem{definition}[theorem]{Definition}
 \newtheorem{example}[theorem]{Example}

 \newtheorem{remark}[theorem]{Remark}
 \newtheorem{notation}[theorem]{Notation}
 
 \theoremstyle{remark}

\newcommand{\commment}[1]{}

\newcommand{\amp}{\mathop{\&}}

\newcommand{\marginnote}[1]{\marginpar{\raggedright\tiny{#1}}}

\newcommand{\bba}{\mathbb{A}}
\newcommand{\bbs}{\mathbb{S}}
\newcommand{\bbA}{\mathbb{A}}
\newcommand{\bbas}{\mathbb{A}^{\delta}}
\newcommand{\bbss}{\mathbb{S}^{\delta}}

\newcommand{\kbbas}{K(A^{\delta})}
\newcommand{\obbas}{O(A^{\delta})}

\newcommand{\jir}{J^{\infty}(A^{\delta})}
\newcommand{\mir}{M^{\infty}(A^{\delta})}

\newcommand{\jira}{J^{\infty}(A)}
\newcommand{\mira}{M^{\infty}(A)}
\newcommand{\jty}{J^{\infty}}
\newcommand{\mty}{M^{\infty}}

\newcommand{\nomi}{\mathbf{i}}
\newcommand{\nomj}{\mathbf{j}}

\newcommand{\cnomm}{\mathbf{m}}
\newcommand{\cnomn}{\mathbf{n}}

\renewcommand{\phi}{\varphi}

\renewcommand{\vec}[1]{\overline{#1}}

\newcommand{\oz}{\overline{z}}

\renewcommand{\epsilon}{\varepsilon}

\tikzset{Ske/.style={draw,rectangle,rounded corners=3pt}}
\tikzset{PIA/.style={draw,rectangle,rounded corners=3pt,dashed}}

\begin{document}

%%
%% The "title" command has an optional parameter,
%% allowing the author to define a "short title" to be used in page headers.
\title{Slanted  Canonicity of Analytic Inductive Inequalities}

%%
%% The "author" command and its associated commands are used to define
%% the authors and their affiliations.
%% Of note is the shared affiliation of the first two authors, and the
%% "authornote" and "authornotemark" commands
%% used to denote shared contribution to the research.
\author{Laurent De Rudder}
\email{l.derudder@uliege.be}
\affiliation{%
  \institution{University of Li\`ege}
  \country{Belgium}
}

\author{Alessandra Palmigiano}
\affiliation{%
  \institution{Vrije Universiteit Amsterdam}
  \country{Netherlands}}
\email{a.palmigiano@vu.nl}

%%
%% By default, the full list of authors will be used in the page
%% headers. Often, this list is too long, and will overlap
%% other information printed in the page headers. This command allows
%% the author to define a more concise list
%% of authors' names for this purpose.
\renewcommand{\shortauthors}{De Rudder and Palmigiano}

%%
%% The abstract is a short summary of the work to be presented in the
%% article.
\begin{abstract}
We prove an algebraic canonicity theorem for normal LE-logics of arbitrary signature, in a generalized setting in which the non-lattice connectives are interpreted as operations mapping tuples of elements of the given lattice to closed or open elements of its canonical extension. Interestingly, the syntactic shape of LE-inequalities which guarantees their canonicity in this generalized setting turns out to coincide with the syntactic shape of {\em analytic inductive inequalities}, which guarantees LE-inequalities to be equivalently captured by analytic structural rules of a proper display calculus. We show that this canonicity result connects and strengthens a number of recent canonicity results in two different areas: subordination algebras, and  transfer results via G\"odel-McKinsey-Tarski translations.

\end{abstract}

%%
%% The code below is generated by the tool at http://dl.acm.org/ccs.cfm.
%% Please copy and paste the code instead of the example below.
%%


%%
%% Keywords. The author(s) should pick words that accurately describe
%% the work being presented. Separate the keywords with commas.

%%
%% This command processes the author and affiliation and title
%% information and builds the first part of the formatted document.
\maketitle

\section{Introduction}
The present paper addresses the connection between canonicity problems in two seemingly unrelated areas, namely subordination algebras and transfer results for nonclassical modal logics via G\"odel-McKinsey-Tarski translations or variations thereof (GMT-type translations). 
Subordination algebras were introduced in  \cite{BezSouVe} as a generalization of de Vries' compingent algebras \cite{deVries} 
and are equivalent presentations of pre-contact algebras \cite{Precontact}, proximity algebras  \cite{duntsch2007region}, and  quasi-modals algebras \cite{Quasimodal, Celani}. 
Canonicity for subordination algebras has been studied in \cite{DeHA} using topological techniques,  in the  context of a Sahlqvist-type result obtained in the setting of classical modal logic for a proper subclass of Sahlqvist formulas, referred to as {\em s-Sahlqvist} formulas. The syntactic shape of s-Sahlqvist formulas guarantees key algebraic/topological properties to their algebraic interpretation, which compensate for the fact that the semantic modal operations on subordination algebras are not defined on its original algebra, but might map elements of it to closed or open elements of its canonical extension.

As to the problem of obtaining Sahlqvist-type results for certain non-classical logics by reduction to classical Sahlqvist theory by means of GMT-type translations, in \cite{GhNaVe05}, the correspondence-via-translation problem has been completely solved for Sahlqvist inequalities in the signature of Distributive Modal Logic, but the corresponding canonicity-via-translation problem, reported to be much harder, was not addressed there, and the canonicity result was obtained following the methodology introduced by  J\'onsson \cite{Jonsson-sahlqvist}. In \cite{CoPaZh19}, results on both correspondence-via-translation and canonicity-via-translation for inductive inequalities in arbitrary signatures of normal distributive lattice expansions (aka normal DLE-logics) are presented, but the canonicity via translation is restricted  to arbitrary normal expansions of bi-Heyting algebras. The source of the additional difficulties was identified in the fact that the algebraic interpretations of the S4-modal operators used to define the GMT-type translations are not defined on each original algebra but might map elements of it to closed or open elements of its canonical extension.

The two independent problems described above have hence a common root in their involving operations on canonical extensions of distributive lattice expansions that do not in general restrict to clopen elements but map clopens to open or to closed elements. These maps, which we refer to as {\em slanted} maps (cf.~Definition \ref{def:c-slanted o-slanted}), have been  considered in \cite[Section 2.3]{GeJo04} in the context of a characterization of canonical extensions of maps as continuous extensions w.r.t.~certain topologies, but the canonicity theory of term inequalities involving these maps was not developed there; interestingly, examples of maps endowed with these weaker topological properties are the adjoints/residuals of the $\sigma$- or $\pi$-extensions of normal modal expansions, and their key role in achieving canonicity results, and specifically in extending J\'onsson's methodology for canonicity from Sahlqvist to inductive inequalities, was emphasised in \cite{PaSoZh15}.

In the present paper, we develop the {\em generalized Sahlqvist-type canonicity} theory for normal LE-logics of arbitrary signature, in a setting in which the algebraic interpretations of the connectives of the expanded signature map elements of the given algebra to closed or open elements of its canonical extension. Interestingly, the class of formulas/inequalities for which this result holds is the class of {\em analytic inductive} LE-inequalities, introduced in \cite{GMPTZ} in the context of the theory of analytic calculi in structural proof theory, to characterize the logics which can be presented by means of proper display calculi \cite{wansing2013displaying}.

Perhaps surprisingly, far from being hard, this generalized canonicity result is obtained as a very smooth refinement of extant generalized Sahlqvist-type canonicity results for LE-logics (cf.~\cite{CoPa-nondist,CoPa-constructive}), established within unified correspondence theory \cite{UnifCorresp}. One of the main contributions of unified correspondence theory is the introduction of an algebraic and algorithmic approach to the proof of canonicity (and correspondence) results that unifies and uniformly generalizes   the methodologies developed by J\'onsson \cite{Jonsson-sahlqvist}, Ghilardi-Meloni \cite{GhMe97}, Sambin-Vaccaro \cite{Sambin:Vacarro:89}, and Conradie-Goranko-Vakarelov \cite{CoGoVaSEQMA}. The fact that the algebraic and algorithmic approach extends so smoothly to the present setting (a step-by-step comparison with the algorithmic canonicity results in standard algebras of \cite{CoPa-nondist,CoPa-constructive} is discussed in Section \ref{appendix}) is further evidence of its robustness. 

The generalized canonicity result obtained in the present paper is then applied to the two canonicity problems mentioned above. Namely, a strengthening of the canonicity result for subordination algebras of \cite{DeHA} is obtained as a direct application, simply by recognizing that the s-Sahlqvist formulas exactly coincide with the analytic 1-Sahlqvist formulas in the classical normal modal/tense logic signature. Moreover, the canonicity-via-translation result of \cite{CoPaZh19} is extended to normal DLE-logics in arbitrary signatures for a subclass of analytic inductive inequalities referred to as {\em transferable} (cf.~Definition \ref{def: transferable}); the syntactic shape of the formulas in this subclass guarantees that the suitable parametric translation of each formula in this class is analytic inductive, so that the generalized canonicity result obtained in the present paper applies to them.
\paragraph{Structure of the paper} In Section \ref{Prelim:Section}, we collect preliminary notions, facts and notation on LE-logics, their standard\footnote{We warn the reader of a possible clash in terminology with the fuzzy logic literature, where  the expression ``standard algebras'' has a technical meaning. Throughout the present paper, ``standard algebras'' refers to  the well-known universal-algebraic definition of algebras, and the adjective ``standard'' is used to emphasise the distinction between algebras as they are usually defined and the slanted algebras introduced in the present article.} algebraic semantics, canonical extensions of normal LEs,  (analytic) Sahlqvist and inductive LE-inequalities, and the algorithm ALBA on analytic inductive LE-inequalities. In Section \ref{subsec:slanted lattice:expansions}, we introduce slanted LE-algebras and their canonical extensions, define how these structures can serve as a semantic environment for normal LE-logics, and introduce the notion of {\em slanted canonicity} (or {\em s-canonicity}, cf.~Definition \ref{def:slanted canonicity}). In Section \ref{section:canonicity}, we prove the main result of the present paper, namely that analytic inductive LE-inequalities are s-canonical. In Section \ref{sec:Sahlqvist canonicity via translation}, we apply the main result of the previous section to extend the transfer result of canonicity to the class of transferable analytic inductive DLE-inequalities. In Section \ref{sec:subordination}, we apply the main result to the setting of subordination algebras to strengthen the canonicity result of \cite{DeHA}. In Section \ref{sec:conclusions}, we discuss further directions stemming from the present results.  In Section \ref{appendix}, we collect the technical lemmas intervening in the proof of our main result.
\section{Preliminaries}\label{Prelim:Section}
In the present section we  recall the definition of normal LE-logics (for Lattice Expansion, see Definition \ref{def:DLE}) and various notions and facts about their algebraic semantics and algorithmic correspondence and canonicity theory. The material presented here re-elaborates \cite[Sections 1, 3, 4]{CoPa-nondist}, \cite[Section 3]{GMPTZ} and \cite[Section 5]{syntactic-completeness}.

\subsection{Basic normal LE-logics}\label{Subsec:PrelimDef}
%Our base language is an unspecified but fixed language $\mathcal{L}_\mathrm{LE}$, to be interpreted over  lattice expansions of compatible similarity type. %As mentioned in the introduction, this setting uniformly accounts for many well known logical systems as the language, related algebras, canonical extension and deductive systems are defined parametrically in the families $\mathcal{F}$ and  $\mathcal{G}$ (see below).
	
Throughout the present paper, we will make heavy use of the following auxiliary definition: an {\em order-type}\footnote{Throughout the paper, order-types will be typically associated with arrays of variables $\vec p: = (p_1,\ldots, p_n)$. When the order of the variables in $\vec p$ is not specified, we will sometimes abuse notation and write $\varepsilon(p) = 1$ or $\varepsilon(p) = \partial$.} over $n\in \mathbb{N}$ is an $n$-tuple $\epsilon\in \{1, \partial\}^n$. For every order-type $\epsilon$, we denote its {\em opposite} order-type by $\epsilon^\partial$, that is, $\epsilon^\partial_i = 1$ iff $\epsilon_i=\partial$ for every $1 \leq i \leq n$. For any lattice $\bba$, we let $\bba^1: = \bba$ and $\bba^\partial$ be the dual lattice, that is, the lattice associated with the converse partial order of $\bba$. For any order-type $\varepsilon$ over $n$, we let $\bba^\varepsilon: = \Pi_{i = 1}^n \bba^{\varepsilon_i}$.
	The language $\mathcal{L}_\mathrm{LE}(\mathcal{F}, \mathcal{G})$, from now on abbreviated as $\mathcal{L}_\mathrm{LE}$, takes as parameters: 1) a denumerable set $\mathsf{PROP}$ of proposition letters, elements of which are denoted $p,q,r$, possibly with indexes; 2) disjoint\footnote{We assume that these families are disjoint for the sake of generality. As done e.g.~in \cite{gehrke-priestley05},  elements in $\mathcal{F} \cap \mathcal{G}$ can be duplicated, and the two copies treated separately, on the basis of different order-theoretic properties (cf.~Remark \ref{rmk:motivation f-g}).} sets of connectives $\mathcal{F}$ and $\mathcal{G}$. %\footnote{The reader might wonder why connectives are partitioned like this. In short, it is dictated by an important distinction in the order-theoretic properties of their interpretations: the interpretations of $\mathcal{F}$-connectives preserve (reverse) all finite joins (meets) coordinate-wise as dictated by the their order-types, and dually for the $\mathcal{G}$-connectives. This distinction turns out to be crucial for the general definition of inductive and Sahlqvist inequalities (see Definition \ref{Inducive:Ineq:Def}) and for the the rules of non-distributive ALBA (see Section \ref{Sec:ReductionElimination}).} 
	 Each $f\in \mathcal{F}$ and $g\in \mathcal{G}$ has arity $n_f\in \mathbb{N}$ (resp.\ $n_g\in \mathbb{N}$) and is associated with some order-type $\varepsilon_f$ over $n_f$ (resp.\ $\varepsilon_g$ over $n_g$).\footnote{Unary $f$ (resp.\ $g$) connectives will be typically denoted  $\Diamond$ (resp.\ $\Box$) if their order-type is 1, and $\lhd$ (resp.\ $\rhd$) if their order-type is $\partial$.} The terms (formulas) of $\mathcal{L}_\mathrm{LE}$ are defined recursively as follows:
	\[
	\phi ::= p \mid \bot \mid \top \mid \phi \wedge \phi \mid \phi \vee \phi \mid f(\overline{\phi}) \mid g(\overline{\phi})
	\]
	where $p \in \mathsf{PROP}$, $f \in \mathcal{F}$, $g \in \mathcal{G}$. Note that, to simplify notations, for $\circ \in \mathcal{F} \cup \mathcal{G}$, we will sometimes write $\circ(\overline{\phi},\overline{\psi})$ where $\overline{\phi}$ is used in the coordinates whose order-type is 1 of $\circ$ and $\overline{\psi}$ in the ones whose order-type is $\partial$. Terms in $\mathcal{L}_\mathrm{LE}$ are denoted either by $s,t$, or by lowercase Greek letters such as $\varphi, \psi, \gamma$ etc. We let $\mathcal{L}^{\leq}_\mathrm{LE}$ denote the set of $\mathcal{L}_\mathrm{LE}$-{\em inequalities}, i.e.~expressions of the form $\phi \leq \psi$ where $\phi, \psi$ are $\mathcal{L}_\mathrm{LE}$-terms, and $\mathcal{L}_\mathrm{LE}^{\mathit{quasi}}$ denote the set of  $\mathcal{L}_\mathrm{LE}$-{\em quasi-inequalities}, i.e.~expressions of the form $(\phi_1 \leq \psi_1 \amp \cdots \amp \phi_n \leq \psi_n) \Rightarrow \phi \leq \psi$ where $\phi_1, \ldots, \phi_n, \psi_1, \ldots \psi_n, \phi, \psi \in \mathcal{L}_\mathrm{LE}$.

\begin{remark} 
\label{rmk:motivation f-g}
The purpose of grouping LE-connectives in the families $\mathcal{F}$ and $\mathcal{G}$ is to identify -- and refer to -- the two types of order-theoretic behaviour which will be relevant for the development of this theory and are specified in  Definition \ref{def:DLE}. The order-theoretic properties defining membership in these families are dual to each other, and are such that, roughly speaking, connectives in $\mathcal{F}$  (resp.~$\mathcal{G}$) can be thought of generalized {\em operators} (resp.~{\em dual operators}), of which diamond (resp.~box) in modal logic and fusion in substructural logic (resp.~intuitionistic implication) are prime examples. We refer to \cite{CoPa-nondist} for an extensive illustration of how this classification can be instantiated in several well known LE-signatures.
The order-duality of this classification is of course very convenient for presentation purposes, since it allows to recover one half of the relevant proofs from the other half, simply by invoking order-duality. However, there is more to it: as discussed more in detail before Definition \ref{def:canext LE standard},
 the order-theoretic properties underlying the definition of the family $\mathcal{F}$ (resp.~$\mathcal{G}$) are strongly linked with the order-theoretic properties of the $\sigma$-extensions (resp.~ $\pi$-extensions) of the algebraic interpretations of connectives in $\mathcal{F}$ (resp.~$\mathcal{G}$), and especially with their properties of {\em adjunction/residuation}, which in turn are key and have been exploited in other papers (cf.~e.g.~\cite{GMPTZ}) also from a {\em proof-theoretic} perspective for guaranteeing certain key results about proper display calculi (e.g.~conservativity, cut elimination) to hold uniformly. Specific to the results we are presently after, these order-theoretic properties also guarantee that  these $\sigma$- and $\pi$-extensions will have sufficient {\em topological} properties for our main canonicity result to go through.
\end{remark}

\begin{definition}
		\label{def:DLE:logic:general}
		For any language $\mathcal{L}_\mathrm{LE} = \mathcal{L}_\mathrm{LE}(\mathcal{F}, \mathcal{G})$, an $\mathcal{L}_\mathrm{LE}$-{\em logic} is a set of sequents $\phi\vdash\psi$, with $\phi,\psi\in\mathcal{L}_\mathrm{LE}$, which contains the following axioms:
		\begin{itemize}
			\item Sequents for lattice operations:
			\begin{align*}
				&p\vdash p, && \bot\vdash p, && p\vdash \top, & &  &\\
				&p\vdash p\vee q, && q\vdash p\vee q, && p\wedge q\vdash p, && p\wedge q\vdash q, &
			\end{align*}
			\item Sequents for each connective $f \in \mathcal{F}$ and $g \in \mathcal{G}$ with $n_f, n_g \geq 1$:
			\begin{align*}
				& f(p_1,\ldots, \bot,\ldots,p_{n_f}) \vdash \bot,~\mathrm{for}~ \varepsilon_f(i) = 1,\\
				& f(p_1,\ldots, \top,\ldots,p_{n_f}) \vdash \bot,~\mathrm{for}~ \varepsilon_f(i) = \partial,\\
				&\top\vdash g(p_1,\ldots, \top,\ldots,p_{n_g}),~\mathrm{for}~ \varepsilon_g(i) = 1,\\
				&\top\vdash g(p_1,\ldots, \bot,\ldots,p_{n_g}),~\mathrm{for}~ \varepsilon_g(i) = \partial,\\
				&f(p_1,\ldots, p\vee q,\ldots,p_{n_f}) \vdash f(p_1,\ldots, p,\ldots,p_{n_f})\vee f(p_1,\ldots, q,\ldots,p_{n_f}),~\mathrm{for}~ \varepsilon_f(i) = 1,\\
				&f(p_1,\ldots, p\wedge q,\ldots,p_{n_f}) \vdash f(p_1,\ldots, p,\ldots,p_{n_f})\vee f(p_1,\ldots, q,\ldots,p_{n_f}),~\mathrm{for}~ \varepsilon_f(i) = \partial,\\
				& g(p_1,\ldots, p,\ldots,p_{n_g})\wedge g(p_1,\ldots, q,\ldots,p_{n_g})\vdash g(p_1,\ldots, p\wedge q,\ldots,p_{n_g}),~\mathrm{for}~ \varepsilon_g(i) = 1,\\
				& g(p_1,\ldots, p,\ldots,p_{n_g})\wedge g(p_1,\ldots, q,\ldots,p_{n_g})\vdash g(p_1,\ldots, p\vee q,\ldots,p_{n_g}),~\mathrm{for}~ \varepsilon_g(i) = \partial,
			\end{align*}
		\end{itemize}
		and is closed under the following inference rules:
		\[ \frac{\varphi \vdash \chi \ \ \ \chi \vdash \psi}{\varphi \vdash \psi}  \> \> \> \> \> \frac{\varphi \vdash \psi}{\phi(\chi/p) \vdash \psi(\chi/p)}  \> \> \> \> \> \frac{\chi \vdash \phi \ \ \ \chi \vdash \psi }{\chi \vdash \phi \wedge \psi}  \> \> \> \> \> \frac{\phi \vdash \chi \ \ \ \psi \vdash \chi}{\phi \vee \psi \vdash \chi}\]
where $\phi(\chi/p)$ denotes uniform substitution of $\chi$ for $p$ in $\phi$,  and for each connective $f \in \mathcal{F}$ and $g \in \mathcal{G}$,
		
		\[
			 \frac{\phi\vdash\psi}{f(\phi_1,\ldots,\phi,\ldots,\phi_n)\vdash f(\phi_1,\ldots,\psi,\ldots,\phi_n)}{~(\varepsilon_f(i) = 1)}
\]
\[
			 \frac{\phi\vdash\psi}{f(\phi_1,\ldots,\psi,\ldots,\phi_n)\vdash f(\phi_1,\ldots,\phi,\ldots,\phi_n)}{~(\varepsilon_f(i) = \partial)}
\]
		\[
			 \frac{\phi\vdash\psi}{g(\phi_1,\ldots,\phi,\ldots,\phi_n)\vdash g(\phi_1,\ldots,\psi,\ldots,\phi_n)}{~(\varepsilon_g(i) = 1)}
		\]
	\[
			 \frac{\phi\vdash\psi}{g(\phi_1,\ldots,\psi,\ldots,\phi_n)\vdash g(\phi_1,\ldots,\phi,\ldots,\phi_n)}{~(\varepsilon_g(i) = \partial)}.
	\]
	
		The minimal $\mathcal{L}_{\mathrm{LE}}(\mathcal{F}, \mathcal{G})$-logic is denoted by $\mathbf{L}_\mathrm{LE}(\mathcal{F}, \mathcal{G})$, or simply by $\mathbf{L}_\mathrm{LE}$ when $\mathcal{F}$ and $\mathcal{G}$ are clear from the context. 
	\end{definition}
	
The standard algebraic semantics of LE-logics is given as follows:
	\begin{definition}
		\label{def:DLE}
		For any LE-signature $\mathcal{L}_{\mathrm{LE}} = \mathcal{L}_{\mathrm{LE}}(\mathcal{F}, \mathcal{G})$, an $\mathcal{L}_{\mathrm{LE}} $-{\em  algebra}  is a tuple $\bba = (A, \mathcal{F}^\bbA, \mathcal{G}^\bbA)$ such that $A$ is a bounded  lattice, $\mathcal{F}^\bbA = \{f^\bbA\mid f\in \mathcal{F}\}$ and $\mathcal{G}^\bbA = \{g^\bbA\mid g\in \mathcal{G}\}$, such that every $f^\bbA\in\mathcal{F}^\bbA$ (resp.\ $g^\bbA\in\mathcal{G}^\bbA$) is an $n_f$-ary (resp.\ $n_g$-ary) operation on $\bbA$. A lattice expansion\footnote{\label{footnote:DLE vs DLO} Normal LEs are sometimes referred to as {\em lattices with operators} (LOs). This terminology derives from the setting of Boolean algebras with operators, in which operators are understood as operations which preserve finite (hence also empty) joins in each coordinate. Thanks to the Boolean negation, operators are typically taken as primitive connectives, and all the other modal operations are reduced to these. However, this terminology is somewhat ambiguous in the lattice setting, in which primitive operations are typically maps which are operators if seen as $\bbA^\epsilon\to \bbA^\eta$ for some order-type $\epsilon$ on $n$ and some order-type $\eta\in \{1, \partial\}$. Rather than speaking of lattices with $(\varepsilon, \eta)$-operators, we then speak of normal LEs. This terminology is also used in other papers, e.g.~\cite{CoPaZh19}. For the sake of internal consistency, we stick with the name ``Boolean Algebra Expansion'' in Section \ref{sec:subordination}.} is {\em normal} if every $f^\bbA\in\mathcal{F}^\bbA$ (resp.\ $g^\bbA\in\mathcal{G}^\bbA$) preserves finite (hence also empty) joins (resp.\ meets) in each coordinate with $\epsilon_f(i)=1$ (resp.\ $\epsilon_g(i)=1$) and reverses finite (hence also empty) meets (resp.\ joins) in each coordinate with $\epsilon_f(i)=\partial$ (resp.\ $\epsilon_g(i)=\partial$). % Let $\mathbb{LE}$ be the class of LEs. Sometimes we will refer to certain LEs as $\mathcal{L}_\mathrm{LE}$-algebras when we wish to emphasize that these algebras have a compatible signature with the logical language we have fixed.
	\end{definition}
	In what follows, we will generically refer to algebras in the definition above as LEs when it is not important to emphasize the specific signature, and as $\mathcal{L}_{\mathrm{LE}} $-algebras when it is. 
Standard LEs as defined above are not the main focus of the present paper, which is rather the {\em non-standard} algebraic semantics of normal LE-logics which we discuss in Section \ref{subsec:slanted lattice:expansions}.

\subsection{Perfect LEs and standard canonical extensions}\label{Perfect_Le} %\marginnote{sicuramente questa sottosez si deve modificare}
%The way the algebraic and the relational semantics of any classical modal logic are linked to one another is very well known: every Boolean algebra with operators (BAO) can be associated with its ultrafilter frame, and with every Kripke frame is associated its complex algebra. To close this triangle, the J\'onsson-Tarski expansion of Stone representation theorem states that every BAO $\bba$ canonically embeds in the complex algebra of its ultrafilter frame. This complex algebra, which is called the {\em perfect}, or {\em canonical extension} of $\bba$, has several additional properties, both intrinsic to it (for instance, it is a {\em powerset} algebra, and not just an algebra of sets) and also relative to its embedded subalgebra. These properties can be expressed purely algebraically, hence independently of the ultrafilter frame construction, and characterize the canonical extension up to an isomorphism fixing the embedded algebra. Analogously well behaved constructions can be performed also for normal (distributive) lattice expansions ((D)LEs), of which we will not give a full account here. Interestingly, whereas the counterparts of the ultrafilter frames look rather different from their Boolean versions, the canonical extension of an LE is defined exactly as the one of a BAO, and the definition is based on the following:
\begin{definition}
Let $A$ be a (bounded) sublattice of a complete lattice $A'$.
\begin{enumerate}
\item  $A$ is {\em dense} in $A'$ if every element of $A'$ can be expressed both as a join of meets and as
a meet of joins of elements from $A$.
\item $A$ is {\em compact} in $A'$ if, for all $S, T \subseteq A$, if $\bigwedge S\leq \bigvee T$ then $\bigwedge S'\leq \bigvee T'$ for some finite $S'\subseteq S$ and $T'\subseteq T$.
\item The {\em canonical extension} of a lattice $A$ is a complete lattice $A^\delta$ containing $A$
as a dense and compact sublattice.
\end{enumerate}
\end{definition}
For any lattice $\bba$, its canonical extension, besides being unique up to an isomorphism fixing $\bba$, always exists (cf.\ \cite[Propositions 2.6 and 2.7]{GH01}\footnote{In \cite{GH01}, the proof of the existence of the canonical extension is constructive, and is based on the complete lattice of Galois-stable sets of the polarity $(A, X, I)$, where $A$ and $X$ respectively are the sets of filters and ideals of the given lattice, and $I$ the relation of having non-empty intersection, as in the lattice representation of Hartonas and Dunn \cite{hartonas1997stone}.}).
%Early on, we mentioned the   intrinsic special properties of canonical extensions. Just like the canonical extension of a Boolean algebra (BA) can be shown to be a {\em perfect} BA (i.e.\ a BA  isomorphic to the powerset algebra of some set), the canonical extension of any lattice is a perfect lattice \cite[Corollary 2.10]{DGP}:
\begin{definition}
\label{def: perfect lattice}
A  complete lattice $A$ is {\em perfect} if $A$  is both completely join-generated by the set $\jira$ of the completely
join-irreducible elements of $A$, and completely meet-generated by the set $\mira$ of
the completely meet-irreducible elements of $A$.
%\item  $\bba$ is isomorphic to the complete lattice of the Galois-stable sets of some RS-frame.
%\end{enumerate}
\end{definition}

Denseness implies that $\jir$ is contained in the meet closure $\kbbas$ of $A$ in $A^\delta$ and that $\mir$ is contained in the join closure $\obbas$ of $A$ in $A^\delta$ \cite{DUGePa2005}.\label{Page:JIr:Clsd:MIr:Opn} The elements of $\kbbas$ are referred to as {\em closed} elements, and elements of $\obbas$ as {\em open} elements.
The canonical extension of an LE $\bba$ will be defined as a suitable expansion of the canonical extension of the underlying lattice of $\bba$.
Before turning to this definition, recall that
taking the canonical extension of a lattice  commutes with
taking order-duals and products, namely:
${(A^\partial)}^\delta = {(A^\delta)}^\partial$ and ${(A_1\times A_2)}^\delta = A_1^\delta\times A_2^\delta$ (cf.\ \cite[Theorem 2.8]{DUGePa2005}).
Hence, ${(A^n)}^\delta$ can be  identified with ${(A^\delta)}^n$ and
${(A^\varepsilon)}^\delta$ with ${(A^\delta)}^\varepsilon$ for any order-type $\varepsilon$. 

Thanks to these identifications,
in order to extend operations of any arity which are monotone or antitone in each coordinate from a lattice $\bba$ to its canonical extension, treating the case
of {\em monotone} and {\em unary} operations suffices:
\begin{definition}
For every unary, order-preserving operation $f : \bba \to \bba$, the $\sigma$-{\em extension} of $f$ is defined firstly by declaring, for every $k\in \kbbas$,
$$f^\sigma(k):= \bigwedge\{ f(a)\mid a\in \bba\mbox{ and } k\leq a\},$$ and then, for every $u\in A^\delta$,
$$f^\sigma(u):= \bigvee\{ f^\sigma(k)\mid k\in \kbbas\mbox{ and } k\leq u\}.$$
The $\pi$-{\em extension} of $f$ is defined firstly by declaring, for every $o\in \obbas$,
$$f^\pi(o):= \bigvee\{ f(a)\mid a\in \bba\mbox{ and } a\leq o\},$$ and then, for every $u\in A^\delta$,
$$f^\pi(u):= \bigwedge\{ f^\pi(o)\mid o\in \obbas\mbox{ and } u\leq o\}.$$
\end{definition}
Key to the use of canonical extensions in logic (e.g.~for proving semantic completeness via canonicity, as well as for the proof-theoretic results mentioned in Remark \ref{rmk:motivation f-g}) 
are certain basic desiderata that canonical extensions must satisfy. These desiderata start from the condition that the canonical extension of a given $\mathcal{L}_{\mathrm{LE}}$-algebra must also be an $\mathcal{L}_{\mathrm{LE}}$-algebra, and that, moreover, the canonical extension of an $\mathcal{L}_{\mathrm{LE}}$-algebra must be a perfect (resp.~complete, in the constructive setting) $\mathcal{L}_{\mathrm{LE}}$-algebra (cf.~Definition \ref{def:perfect LE} below). The first desideratum is met immediately whenever all connectives in the given signature $\mathcal{L}_{\mathrm{LE}}$ are \emph{smooth}, i.e.~their $\sigma$- and $\pi$-extensions coincide. While \emph{unary} normal connectives are smooth, (see e.g.\ \cite[Lemma 4.4]{GH01}), connectives with arity greater than 1 are typically non-smooth (see e.g.\ \cite[Example 4.6]{Vosmaer}). Hence, whenever $\mathcal{L}_{\mathrm{LE}}$ includes non-smooth connectives, to satisfy the first desideratum one needs to decide whether to define the canonical extensions of $\mathcal{L}_{\mathrm{LE}}$-algebras  by taking, for each connective in $\mathcal{L}_{\mathrm{LE}}$,  either its $\sigma$- or its $\pi$-extension. Our choice is guided by the second desideratum: 
it is easy to see that the $\sigma$- and $\pi$-extensions of $\varepsilon$-monotone maps are $\varepsilon$-monotone; moreover,  the $\sigma$-extension of a map which coordinate-wise preserves {\em finite} joins or reverses {\em finite} meets will coordinate-wise preserve {\em arbitrary} joins
or reverse {\em arbitrary} meets, and dually, the $\pi$-extension of a map which coordinate-wise preserves {\em finite} meets or reverses {\em finite} joins  will coordinate-wise preserve {\em arbitrary} meets
or reverse {\em arbitrary} joins  (see \cite[Lemma 4.6]{GH01}). Therefore, defining the canonical extension of an $\mathcal{L}_{\mathrm{LE}}$-algebra by choosing the $\sigma$-extensions of operations in $\mathcal{F}^\bbA$ and the $\pi$-extensions of operations in $\mathcal{G}^\bbA$ will guarantee both the first and  the second desideratum.\footnote{For some discussion on this and the importance of choosing the appropriate extension we refer the reader to \cite[Section 7]{GhNaVe05}.} These considerations motivate the following
\begin{definition}
\label{def:canext LE standard}
%For every LE $\bba = (\bbb, \mathcal{F}, \mathcal{G})$, the {\em canonical extension} of $\bba$
The canonical extension of an
$\mathcal{L}_\mathrm{LE}$-algebra $\bbA = (A, \mathcal{F}^\bbA, \mathcal{G}^\bbA)$ is the   $\mathcal{L}_\mathrm{LE}$-algebra
$\bbA^\delta: = (A^\delta, \mathcal{F}^{\bbA^\delta}, \mathcal{G}^{\bbA^\delta})$ such that $f^{\bbA^\delta}$ and $g^{\bbA^\delta}$ are defined as the
$\sigma$-extension of $f^{\bbA}$ and as the $\pi$-extension of $g^{\bbA}$ respectively, for all $f\in \mathcal{F}$ and $g\in \mathcal{G}$.
\end{definition}
As mentioned in the previous discussion, defined as indicated above, the canonical extension of an $\mathcal{L}_\mathrm{LE}$-algebra $\bba$ can be shown to be a {\em perfect} $\mathcal{L}_\mathrm{LE}$-algebra:
\begin{definition}
\label{def:perfect LE}
An LE $\bbA = (A, \mathcal{F}^\bbA, \mathcal{G}^\bbA)$ is perfect if $A$ is a perfect lattice (cf.\ Definition \ref{def: perfect lattice}), and moreover the following infinitary distribution laws are satisfied for each $f\in \mathcal{F}$, $g\in \mathcal{G}$, $1\leq i\leq n_f$ and $1\leq j\leq n_g$: for every $S\subseteq A$,
\begin{center}
\begin{tabular}{c c }
$f(x_1,\ldots, \bigvee S, \ldots, x_{n_f}) =\bigvee \{ f(x_1,\ldots, x, \ldots, x_{n_f}) \mid x\in S \}$  & if $\varepsilon_f(i) = 1$\\

$f(x_1,\ldots, \bigwedge S, \ldots, x_{n_f}) =\bigvee \{ f(x_1,\ldots, x, \ldots, x_{n_f}) \mid x\in S \}$  & if $\varepsilon_f(i) = \partial$\\

$g(x_1,\ldots, \bigwedge S, \ldots, x_{n_g}) =\bigwedge \{ g(x_1,\ldots, x, \ldots, x_{n_g}) \mid x\in S \}$  & if $\varepsilon_g(i) = 1$\\

$g(x_1,\ldots, \bigvee S, \ldots, x_{n_g}) =\bigwedge \{ g(x_1,\ldots, x, \ldots, x_{n_g}) \mid x\in S \}$  & if $\varepsilon_g(i) = \partial$.\\

\end{tabular}
\end{center}

\end{definition}
Before finishing the present subsection, let us spell out and further simplify the definitions of the extended operations.
First of all, we recall that taking the order-dual interchanges closed and open elements:
$K(({A^\delta)}^\partial) \cong O(A^\delta)$ and $O({(A^\delta)}^\partial) \cong\kbbas$;  similarly, $K({(A^n)}^\delta) \cong\kbbas^n$, and $O({(A^n)}^\delta) \cong\obbas^n$. Hence,  $K({(A^\delta)}^\epsilon) \cong\prod_i K(A^\delta)^{\epsilon(i)}$ and $O({(A^\delta)}^\epsilon) \cong\prod_i O(A^\delta)^{\epsilon(i)}$ for every LE $\bba$ and every order-type $\epsilon$ on any $n\in \mathbb{N}$, where
\begin{center}
\begin{tabular}{cc}
$K(A^\delta)^{\epsilon(i)}: =\begin{cases}
K(A^\delta) & \mbox{if } \epsilon(i) = 1\\
O(A^\delta) & \mbox{if } \epsilon(i) = \partial\\
\end{cases}
$ &
$O(A^\delta)^{\epsilon(i)}: =\begin{cases}
O(A^\delta) & \mbox{if } \epsilon(i) = 1\\
K(A^\delta) & \mbox{if } \epsilon(i) = \partial.\\
\end{cases}
$\\
\end{tabular}
\end{center}
Denoting by $\leq^\epsilon$ the product order on $(A^\delta)^\epsilon$, we have for every $f\in \mathcal{F}$, $g\in \mathcal{G}$,  $\overline{k} \in K({(A^\delta)}^{\epsilon_f})$, $\overline{o} \in O({(A^\delta)}^{\epsilon_f})$ $\overline{u}\in (A^\delta)^{n_f}$ and $\overline{v}\in (A^\delta)^{n_g}$,
\begin{center}
\begin{tabular}{l l}
$f^\sigma (\overline{k}):= \bigwedge\{ f( \overline{a})\mid \overline{a}\in A^{\epsilon_f}\mbox{ and } \overline{k}\leq^{\epsilon_f} \overline{a}\}$ & $f^\sigma (\overline{u}):= \bigvee\{ f^\sigma( \overline{k})\mid \overline{k}\in K({(A^\delta)}^{\epsilon_f})\mbox{ and } \overline{k}\leq^{\epsilon_f} \overline{u}\}$ \\
$g^\pi (\overline{o}):= \bigvee\{ g( \overline{a})\mid \overline{a}\in A^{\epsilon_g}\mbox{ and } \overline{a}\leq^{\epsilon_g} \overline{o}\}$ & $g^\pi (\overline{v}):= \bigwedge\{ g^\pi( \overline{o})\mid \overline{o}\in O({(A^\delta)}^{\epsilon_g})\mbox{ and } \overline{v}\leq^{\epsilon_g} \overline{o}\}$. \\
\end{tabular}
%
%\begin{tabular}{cc}
%$\Diamond^\sigma u:= \bigvee\{ \Diamond k\mid k\in \kbbas\mbox{ and } k\leq u\}$ & $\Box^\pi u:= \bigwedge\{ \Box o\mid o\in \obbas\mbox{ and } u\leq o\}$\\
%${\lhd}^\sigma u:= \bigvee\{ {\lhd}o\mid o\in \obbas\mbox{ and } u\leq o\}$ & ${\rhd}^\pi u:= \bigwedge\{ {\rhd}k\mid k\in \kbbas\mbox{ and } k\leq u\}$\\
%$u\circ^\sigma v:= \bigvee\{ k\circ k'\mid k, k'\in \kbbas, k\leq u \mbox{ and } k'\leq v\}$ & $u\star^\pi v:= \bigwedge\{o\star o'\mid o, o'\in \obbas, u\leq o \mbox{ and } v\leq o'\}$.\\
%\end{tabular}
\end{center}

The algebraic completeness of  $\mathbf{L}_\mathrm{LE}$ %and $\mathbf{L}_\mathrm{LE}^*$ 
and the canonical embedding of LEs into their canonical extensions immediately yield completeness of $\mathbf{L}_\mathrm{LE}$ %and $\mathbf{L}_\mathrm{LE}^*$ 
w.r.t.\ the appropriate class of perfect LEs.
		%\end{comment}

\subsection{Inductive and Sahlqvist (analytic) LE-inequalities}\label{Inductive:Fmls:Section}
In this section we  recall the definitions of inductive and Sahlqvist LE-inequalities introduced in \cite{CoPa-nondist} and their corresponding `analytic' restrictions introduced in \cite{GMPTZ} in the distributive setting and then generalized to the setting of  LEs of arbitrary signatures in \cite{greco2018algebraic}.  Each inequality in any of these classes 
is canonical  and elementary (cf.~\cite[Theorems 8.8 and 8.9]{CoPa-nondist}).  %We will not give a direct proof that all inductive inequalities are elementary and canonical, but this will follow from the facts that they are all reducible by the {\sf ALBA}-algorithm and that all inequalities so reducible are elementary and canonical.

%\subsection{Inductive inequalities}\label{ssec:InductiveInequalities}
		%A DLE-inequality %(resp.\ DLE$^*$-inequality)
		%is an expression of the form $s\leq t$ where $s,t\in\mathcal{L}_\mathrm{DLE}$ %(resp.\ $s,t\in\mathcal{L}_\mathrm{DLE}^*$),
		%which is essentially a sequent in algebraic form.
		
	%	In the present subsection, we define the {\em inductive} $\mathcal{L}_\mathrm{LE}$-inequalities on which the algorithm ALBA defined in Section \ref{Spec:Alg:Section} will be shown to succeed. %equivalently transform into one (or the conjunction of more) pure quasi-inequalities in an expanded language. For more details,

		\begin{definition}[\textbf{Signed Generation Tree}]
			\label{def: signed gen tree}
			The \emph{positive} (resp.\ \emph{negative}) {\em generation tree} of any $\mathcal{L}_\mathrm{LE}$-term $s$ is defined by labelling the root node of the generation tree of $s$ with the sign $+$ (resp.\ $-$), and then propagating the labelling on each remaining node as follows:
			\begin{itemize}
				%\item The root node $+s$ (resp.\ $-s$) is the root node of the positive (resp.\ negative) generation tree of $s$ signed with + (resp.\ $-$).
				\item For any node labelled with $ \lor$ or $\land$, assign the same sign to its children nodes.
				%\item If a node is labelled with $\lhd$, $\rhd$, assign the opposite sign to its child node.
				\item For any node labelled with $h\in \mathcal{F}\cup \mathcal{G}$ of arity $n_h\geq 1$, and for any $1\leq i\leq n_h$, assign the same (resp.\ the opposite) sign to its $i$th child node if $\varepsilon_h(i) = 1$ (resp.\ if $\varepsilon_h(i) = \partial$).
			\end{itemize}
			Nodes in signed generation trees are \emph{positive} (resp.\ \emph{negative}) if are signed $+$ (resp.\ $-$).
		\end{definition}
		
		Signed generation trees will be mostly used in the context of term inequalities $s\leq t$. In this context we will typically consider the positive generation tree $+s$ for the left-hand side and the negative one $-t$ for the right-hand side. We will also say that a term-inequality $s\leq t$ is \emph{uniform} in a given variable $p$ if all occurrences of $p$ in both $+s$ and $-t$ have the same sign, and that $s\leq t$ is $\epsilon$-\emph{uniform} in a (sub)array $\vec{p}$ of its variables if $s\leq t$ is uniform in every $p$ in $\vec{p}$, occurring with the sign indicated by $\epsilon$. With a routine proof by induction, one can show that if a term-inequality $s\leq t$ is $1$-uniform (resp.~$\partial$-uniform) in a given variable $p$, then the term-function associated with $s$ (see Definitions \ref{def:DLE} and \ref{def:c-slanted o-slanted})  is order-preserving (resp.~order-reversing) in $p$, and the term-function associated with $t$ is order-reversing (resp.~order-preserving) in $p$. Therefore, the validity of  $s(p) \leq t(p)$ is equivalent to the validity of $s(\top) \leq t(\top)$ (resp.~$s(\bot) \leq t(\bot)$)\footnote{As noticed in \cite[Remark 6.3]{CoPa-Dist}, these equivalences are in fact instances of the Ackermann Lemmas %\ref{Ackermann:Dscrptv:Left:Lemma} and \ref{Ackermann:Dscrptv:Right:Lemma}, 
		where the $p$-variant  assignment $v'$  is such that $v'(p)=\top$  (if $s\leq t$ is $1$-uniform) or $v'(p) = \bot$ (if $s\leq t$ is $\partial$-uniform).}. This observation is easily generalised to an arbitrary subarray of variables of $s \leq t$.
		
		For any term $s(p_1,\ldots p_n)$, any order-type $\epsilon$ over $n$, and any $1 \leq i \leq n$, an \emph{$\epsilon$-critical node} in a signed generation tree of $s$ is a leaf node $+p_i$ if $\epsilon (i) = 1$ and $-p_i$ if $\epsilon (i) = \partial$. An $\epsilon$-{\em critical branch} in the tree is a branch from an $\epsilon$-critical node. Variable occurrences corresponding to $\epsilon$-critical nodes are to be solved for (cf.~Section \ref{Spec:Alg:Section}).

		For every term $s(p_1,\ldots p_n)$ and every order-type $\epsilon$, we say that $+s$ (resp.\ $-s$) {\em agrees with} $\epsilon$, and write $\epsilon(+s)$ (resp.\ $\epsilon(-s)$), if every leaf in the signed generation tree of $+s$ (resp.\ $-s$) is $\epsilon$-critical.
		%In other words, $\epsilon(+s)$ (resp.\ $\epsilon(-s)$) means that all variable occurrences corresponding to leaves of $+s$ (resp.\ $-s$) are to be solved for according to $\epsilon$. 
		We will also write $+s'\prec \ast s$ (resp.\ $-s'\prec \ast s$) to indicate that the subterm $s'$ inherits the positive (resp.\ negative) sign from the signed generation tree $\ast s$. Finally, we will write $\epsilon(\gamma) \prec \ast s$ (resp.\ $\epsilon^\partial(\gamma) \prec \ast s$) to indicate that the signed subtree $\gamma$, with the sign inherited from $\ast s$, agrees with $\epsilon$ (resp.\ with $\epsilon^\partial$).

		We will write $\phi(!x)$ (resp.~$\phi(!\overline{x})$) to indicate that the variable $x$ (resp.~each variable $x$ in $\overline{x}$) occurs exactly once in $\phi$. Accordingly, we will write $\phi[\gamma / !x]$  (resp.~$\phi[\overline{\gamma}/!\overline{x}]$) to indicate the formula obtained  by substituting $\gamma$ (resp.~each term $\gamma$ in $\overline{\gamma}$) for the unique occurrence of (its corresponding variable) $x$ in $\phi$.

		\begin{definition}
			\label{def:good:branch}
			Nodes in signed generation trees will be called \emph{$\Delta$-adjoints}, \emph{syntactically left residual (SLR)}, \emph{syntactically right residual (SRR)}, and \emph{syntactically right adjoint (SRA)}, according to the specification given in Table \ref{Join:and:Meet:Friendly:Table}.
			A branch in a signed generation tree $\ast s$, with $\ast \in \{+, - \}$, is called a \emph{good branch} if it is the concatenation of two paths $P_1$ and $P_2$, one of which may possibly be of length $0$, such that $P_1$ is a path from the leaf consisting (apart from variable nodes) only of PIA-nodes, and $P_2$ consists (apart from variable nodes) only of Skeleton-nodes. %\footnote{These classes are grouped together into the super-classes \emph{Skeleton} and \emph{PIA} as indicated in the table. This organization is motivated and discussed in \cite{CFPS} and \cite{CoGhPa13}.}
A branch is \emph{excellent} if it is good and in $P_1$ there are only SRA-nodes. A good branch is \emph{Skeleton} if the length of $P_1$ is $0$ (hence Skeleton branches are excellent), and  is {\em SLR}, or {\em definite}, if  $P_2$ only contains SLR nodes.
			\begin{table}[h]
				\begin{center}
                \bgroup
                \def\arraystretch{1.2}%  1 is the default, change whatever you need
					\begin{tabular}{| c | c |}
						\hline
						Skeleton &PIA\\
						\hline
						$\Delta$-adjoints & Syntactically Right Adjoint (SRA) \\
						\begin{tabular}{ c c c c c c}
							$+$ &$\vee$ &\\
							$-$ &$\wedge$ \\
							\hline
						\end{tabular}
						&
						\begin{tabular}{c c c c }
							$+$ &$\wedge$ &$g$ & with $n_g = 1$ \\
							$-$ &$\vee$ &$f$ & with $n_f = 1$ \\
							\hline
						\end{tabular}
						\\
						Syntactically Left Residual (SLR) &Syntactically Right Residual (SRR)\\
						\begin{tabular}{c c c c }
							$+$ &  &$f$ & with $n_f \geq 1$\\
							$-$ &  &$g$ & with $n_g \geq 1$ \\
						\end{tabular}
						&\begin{tabular}{c c c c}
							$+$ & &$g$ & with $n_g \geq 2$\\
							$-$ &  &$f$ & with $n_f \geq 2$\\
						\end{tabular}
						\\
						\hline
					\end{tabular}
                \egroup
				\end{center}
				\caption{Skeleton and PIA nodes for $\mathrm{LE}$.}\label{Join:and:Meet:Friendly:Table}
				\vspace{-1em}
			\end{table}
		\end{definition}
We refer to \cite[Remark 3.3]{CoPa-nondist} %\cite{vanbenthem2005} 
and \cite[Section 3]{GhNaVe05} for a discussion about the notational conventions and terminology.	%\marginnote{il remark e' commentato; ci sono referenze da riportare alla luce}

\begin{example}\label{ex: example Sahlqvist bi} The language $\mathcal{L}_\mathrm{LE}$ of bi-intuitionistic modal logic is obtained by instantiating $\mathcal{F}=\lbrace \Diamond, \righttail \rbrace $ and $\mathcal{G}= \lbrace \square, \rightarrow \rbrace$ with $n_\Diamond = n_\square = 1$, $n_{ \footnotesize{\righttail} } = n_\rightarrow = 2$ and $\epsilon_\Diamond = \epsilon_\square =1$, $\epsilon_{\footnotesize{\righttail}} = \epsilon_\rightarrow = (\partial, 1)$. 
In this language, the signed generation trees associated with the  inequality
\[ \Diamond \square p \vee  q \leq \Diamond \square q \wedge (\Diamond  r \righttail p) \] are represented in the following diagram, where   PIA nodes occur inside dashed rectangles and Skeleton nodes inside continuous ones.

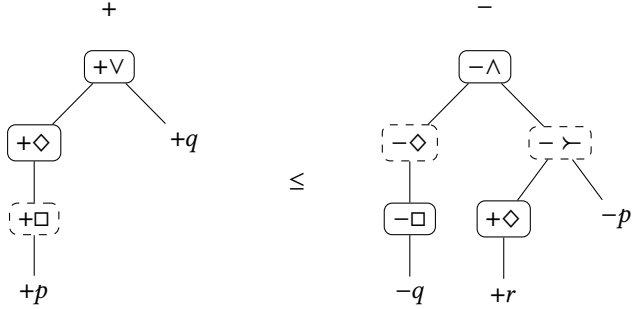
\begin{center}
\begin{tikzpicture}
\node(plus) at (0,0.75) {$+$};
\node(minus) at (5,0.75) {$-$};
\node[Ske](vee1) at (0,0) {$+ \vee$};
\node[Ske](Diamond1) at (-1,-1) {$+ \Diamond$};
\node[PIA](Box1) at (-1,-2){$+\square$};
\node(p1) at (-1,-3){$+p$};
\node(q1) at (1,-1) {$+q$};
\node[Ske](wedge1) at (5,0) {$-\wedge$};
\node(leq) at (2.5, -1.5) {$\leq$};
\node[PIA](Diamond2) at (4,-1) {$-\Diamond$};
\node[Ske](Box2) at (4,-2) {$-\square$};
\node(q2) at (4,-3) {$-q$};
\node[Ske](Diamond3) at (5.25,-2) {$+ \Diamond$};
\node[PIA](righttail) at (6,-1) {$- \righttail$};
\node(r) at (5.25,-3) {$+r$};
\node(p2) at (6.75,-2) {$-p$};
\draw (vee1) to (Diamond1);
\draw (vee1) to (q1);
\draw(Diamond1) to (Box1);
\draw (Box1) to (p1);
\draw (wedge1) to (Diamond2);
\draw (wedge1) to (righttail);
\draw (Diamond2) to (Box2);
\draw (Box2) to (q2);
\draw (Diamond3) to (r);
\draw (righttail) to (Diamond3);
\draw (righttail) to (p2);
\end{tikzpicture}

\end{center}
The inequality above is not uniform in $p$ and $q$ (each of the variables has a positive occurrence and a negative one), but is uniform in $r$ (in this particular case  because $r$ occurs only once).
If we consider the order-type $\epsilon$ on $(p,q,r)$ given by $\epsilon=(1, 1, \partial)$, the critical nodes in the generations trees are $+p$ and $+q$. There is no $\epsilon$-critical occurrence of $r$. %The inequality is not uniform in $p$ and $q$ (each of the variables has a positive occurrence and a negative one), but is uniform in $r$ (in this particular case mainly because $r$ occurs only once). We also saw that t
The term $+s:=+(\Diamond \square p \vee q)$ agrees with $\epsilon$, while the term $-t=-(\Diamond \square q \wedge  (\Diamond r \righttail p))$ agrees with  $\epsilon^\partial$. Now, for the subterms $t':=\Diamond r$ and $t'':= \Diamond r \righttail p$ of $t$, we have $+t' \prec - t$ and $-t'' \prec -t$. Moreover,  $\epsilon^\partial(t') \prec -t$ and $\epsilon^\partial(t'') \prec -t$.

The branches which end in $+p$, $+q$,  and $-p$ are good since, traversing the corresponding branches starting from the root, we first encounter Skeleton  nodes and then only PIA nodes. The branches ending in $+p$ and $+q$ are also excellent because they do not contain SRR nodes (the only SRR node occurring in this example is $- \righttail$). The branch which ends in $+q$ is in particular Skeleton since it contains no occurrences of PIA nodes (the length of $P_1$ is $0$). Finally, the branches which end in $-q$ and $+r$ are not good, since  (again starting from the root) a Skeleton node ($-\square$, $+\Diamond$) occurs in the scope of a PIA node ($-\Diamond$, $-\righttail$).
\end{example}

\begin{definition}[Inductive inequalities]\label{Inducive:Ineq:Def}
For any order-type $\epsilon$ and any irreflexive and transitive relation (i.e.\ strict partial order) $\Omega$ on $p_1,\ldots p_n$, the signed generation tree $*s$ $(* \in \{-, + \})$ of a term $s(p_1,\ldots p_n)$ is \emph{$(\Omega, \epsilon)$-inductive} if
			\begin{enumerate}
				\item for all $1 \leq i \leq n$, every $\epsilon$-critical branch with leaf $p_i$ is good (cf.\ Definition \ref{def:good:branch});
				\item every $m$-ary SRR-node occurring in the critical branch is of the form \[ \circledast(\gamma_1,\dots,\gamma_{j-1},\beta,\gamma_{j+1}\ldots,\gamma_m),\] where for any $h\in\{1,\ldots,m\}\setminus j$: %$\gamma \in \{\gamma_1,\ldots,\gamma_{j-1},\gamma_{j+1},\ldots,\alpha_m\}$
\begin{enumerate}
\item $\epsilon^\partial(\gamma_h) \prec \ast s$ (cf.\ discussion before Definition \ref{def:good:branch}), and
%\item $\epsilon^\partial(\ast \gamma)$, and
%
\item $p_k <_{\Omega} p_i$ for every $p_k$ occurring in $\gamma_h$ and for every $1\leq k\leq n$.
\end{enumerate}
	\end{enumerate}
			
			We will refer to $<_{\Omega}$ as the \emph{dependency order} on the variables. An inequality $s \leq t$ is \emph{$(\Omega, \epsilon)$-inductive} if the signed generation trees $+s$ and $-t$ are $(\Omega, \epsilon)$-inductive. An inequality $s \leq t$ is \emph{inductive} if it is $(\Omega, \epsilon)$-inductive for some $\Omega$ and $\epsilon$.
		\end{definition}
		
		In what follows, we refer to formulas $\phi$ such that only PIA nodes occur in $+\phi$ (resp.\ $-\phi$) as {\em positive} (resp.\ {\em negative}) {\em PIA-formulas}, and to formulas $\xi$ such that only Skeleton nodes occur in $+\xi$ (resp.\ $-\xi$) as {\em positive} (resp.\ {\em negative}) {\em Skeleton-formulas}\label{page: positive negative PIA}. PIA formulas $\ast \phi$ in which no nodes $+\wedge$ and $-\vee$ occur are referred to as {\em definite}. Skeleton formulas $\ast \xi$ in which no nodes $-\wedge$ and $+\vee$ occur are referred to as {\em definite}.

\begin{definition}\label{Sahlqvist:Ineq:Def}
For an order-type $\epsilon$, the signed generation tree $\ast s$, $\ast \in \{-, + \}$, of a term $s(p_1,\ldots p_n)$ is \emph{$\epsilon$-Sahlqvist} if every $\epsilon$-critical branch is excellent. An inequality $s \leq t$ is \emph{$\epsilon$-Sahlqvist} if the trees $+s$ and $-t$ are both $\epsilon$-Sahlqvist.  An inequality $s \leq t$ is \emph{Sahlqvist} if it is $\epsilon$-Sahlqvist for some $\epsilon$.

\end{definition}

\begin{definition}[Analytic inductive and  analytic Sahlqvist inequalities]
	\label{def:type5}
	For every order-type $\epsilon$ and every irreflexive and transitive relation $\Omega$ on the variables $p_1,\ldots p_n$,
			the signed generation tree $\ast s$ ($\ast\in \{+, -\}$) of a term $s(p_1,\ldots p_n)$ is \emph{analytic $(\Omega, \epsilon)$-inductive} (resp.~\emph{analytic $\epsilon$-Sahlqvist}) if
			
			\begin{enumerate}%\marginnote{We still need to sort out the occurences of top and bottom}
				\item $\ast s$ is $(\Omega, \epsilon)$-inductive (resp.~$\epsilon$-Sahlqvist); %(cf.\ Definition \ref{Inducive:Ineq:Def});
				\item every branch of $\ast s$ is good (cf.\ Definition \ref{def:good:branch}).
			\end{enumerate}	
	
	An inequality $s \leq t$ is \emph{analytic $(\Omega, \epsilon)$-inductive} (resp.~\emph{analytic $\epsilon$-Sahlqvist})  if $+s$ and $-t$ are both analytic  $(\Omega, \epsilon)$-inductive (resp.~analytic $\epsilon$-Sahlqvist). An inequality $s \leq t$ is \emph{analytic inductive} (resp.~\emph{analytic Sahlqvist}) if is analytic $(\Omega, \epsilon)$-inductive (resp.~analytic $\epsilon$-Sahlqvist)  for some $\Omega$ and $\epsilon$ (resp.~for some $\epsilon$).
\end{definition}	

\begin{example}
In the light of the previous definitions and the discussion in  Example \ref{ex: example Sahlqvist bi}, the modal bi-intuitionistic inequality $\Diamond \square p \vee  q \leq \Diamond \square q \wedge (\Diamond  r \righttail p) $ is $\epsilon$-Sahlqvist (and hence inductive) for $\epsilon(p,q,r)=(1, 1, \partial)$. However, it is not analytic since the negative generation tree of its right-hand side contains branches which are not good. 
\end{example}

\begin{notation}\label{notation: analytic inductive}
Following \cite{syntactic-completeness},  we will sometimes represent $(\Omega, \epsilon)$-analytic inductive inequalities as follows: \[(\varphi\leq \psi)[\overline{\alpha}/!\overline{x}, \overline{\beta}/!\overline{y},\overline{\gamma}/!\overline{z}, \overline{\delta}/!\overline{w}],\] where $(\varphi\leq \psi)[!\overline{x}, !\overline{y},!\overline{z}, !\overline{w}]$ is  the Skeleton of the given inequality,  $\overline{\alpha}$ (resp.~$\overline{\beta}$) denotes the positive (resp.~negative) maximal PIA-subformulas, i.e.~each $\alpha$ in $\overline{\alpha}$ and $\beta$ in $\overline{\beta}$ contains at least one $\varepsilon$-critical occurrence of some propositional variable, and moreover:
\begin{enumerate}
\item for each $\alpha\in \overline{\alpha}$, either 
$+\alpha\prec +\varphi$ or $+\alpha\prec -\psi$;
\item for each $\beta\in \overline{\beta}$, either 
  $-\beta\prec +\varphi$ or $-\beta\prec -\psi$,
  \end{enumerate}
and $\overline{\gamma}$ (resp.~$\overline{\delta}$) denotes the positive (resp.~negative) maximal $\varepsilon^{\partial}$-subformulas,  i.e.:
\begin{enumerate}

  \item for each $\gamma\in \overline{\gamma}$, either   $+\gamma\prec +\varphi$ or $+\gamma\prec -\psi$;
  \item for each $\delta\in \overline{\delta}$, either  $-\delta\prec +\varphi$ or $-\delta\prec -\psi$.
    \end{enumerate}
    For the sake of a more compact notation, in what follows we sometimes write $(\varphi\leq \psi)[\overline{\alpha}, \overline{\beta},\overline{\gamma}, \overline{\delta}]$ in place of \[ (\varphi\leq \psi)[\overline{\alpha}/!\overline{x}, \overline{\beta}/!\overline{y},\overline{\gamma}/!\overline{z}, \overline{\delta}/!\overline{w}]. \]
\end{notation}

\begin{remark}[The distributive setting]
When interpreting LE-languages on perfect distributive lattice expansions (DLEs), the logical disjunction is interpreted by means of the coordinatewise completely $\wedge$-preserving join operation of the lattice, and the logical conjunction with the coordinatewise completely $\vee$-preserving meet operation of the lattice. Hence we are justified in listing $+\wedge$ and $-\vee$ among the SLRs, and $+\vee$ and $-\wedge$ among the SRRs, as is done in Table \ref{Join:and:Meet:Friendly:Table:DLE}.

Consequently, we obtain enlarged classes of Sahlqvist and inductive inequalities
%, which we will call the \emph{distributive Sahlqvist} and \emph{distributive inductive} inequalities respectively,
by simply applying Definitions \ref{def:good:branch}, \ref{Sahlqvist:Ineq:Def} and \ref{Inducive:Ineq:Def} with respect to Table \ref{Join:and:Meet:Friendly:Table:DLE}.%\footnote{In the analogous table given in \cite{UnifiedCor}, the nodes $+\wedge$ and $-\vee$ are listed among the $\Delta$-adjoints. These definitions both yield the same syntactic classes, but cater for different algorithms. In particular, the algorithm illustrated in \cite{UnifiedCor} is based on \cite{CoPa-dist} and uses the splitting rules (the soundness of which is based on $\Delta$-adjunction) as part of the approximation phase. In contrast, in the present paper, approximation is taken care of by different rules which pivot exclusively on join- and meet-preservation or reversal.}

\begin{table}[h]
				\begin{center}
					\begin{tabular}{| c | c |}
						\hline
						Skeleton &PIA\\
						\hline
						$\Delta$-adjoints & SRA \\
						\begin{tabular}{ c c c c c c}
							$\phantom{\wedge}$ &$+$ &$\vee$ &$\phantom{\lhd}$ & &\\
							$\phantom{\vee}$ &$-$ &$\wedge$ \\
							\hline
						\end{tabular}
						&
						\begin{tabular}{c c c c }
							$+$ &$\wedge$ &$g$ & with $n_g = 1$ \\
							$-$ &$\vee$ &$f$ & with $n_f = 1$ \\
							\hline
						\end{tabular}
						\\
						SLR &SRR\\
						\begin{tabular}{c c c c }
							$+$ & $\wedge$  &$f$ & with $n_f \geq 1$\\
							$-$ & $\vee$ &$g$ & with $n_g \geq 1$ \\
						\end{tabular}
						&\begin{tabular}{c c c c}
							$+$ & $\vee$ &$g$ & with $n_g \geq 2$\\
							$-$ & $\wedge$  &$f$ & with $n_f \geq 2$\\
						\end{tabular}
						\\
						\hline
					\end{tabular}
				\end{center}
				\caption{Skeleton and PIA nodes for $\mathcal{L}_\mathrm{DLE}$.}\label{Join:and:Meet:Friendly:Table:DLE}
				\vspace{-1em}
			\end{table}

\end{remark}

	\subsection{Basic LE-language expanded with residuals}
	
	We now introduce an expansion of the language $\mathcal{L}_\mathrm{LE}(\mathcal{F}, \mathcal{G})$ with connectives  which are to be interpreted as the residuals in each coordinate (cf.~Definition \ref{def_residuated_lattice}) of the connectives in $\mathcal{F}$ and $\mathcal{G}$. This is the first of two expansion steps (the second of which being described in Section \ref{Subsec:Expanded:Land})	which lead to the language $\mathcal{L}^+_\mathrm{LE}(\mathcal{F}, \mathcal{G})$ (see Section \ref{Subsec:Expanded:Land}) in which the ALBA-reductions take place.  %expressive enough to encode the  correspondence-theoretic and canonicity arguments needed for the Sahlqvist theorem for $\mathcal{L}_\mathrm{LE}(\mathcal{F}, \mathcal{G})$. %This is a strategy which originated in \cite{Conradie:et:al:SQEMAI} and whereby the traditional translation to second-order logic is avoided. This contrasts with the approach of in e.g.\ \cite{GNV}, who give separate proofs for correspondence and canonicity, the first by a reduction to the classical case via a modified G\"{o}del-MicKinsey-Tarski-style translation, and the second using an approach similar to that of \cite{Jo94, GhMe97} for which the basic algebraic language suffices. We refer the interested reader to \cite{CPZ:Trans} for a more general study of translation-based correspondence and canonicity proofs and to \cite{CP:constructive} and \cite{CGPSZ14} for discussion of the relationships between the various approaches to canonicity including J\'onsson's canonicity, Ghilardi-Meloni's constructive canonicity, and Sambin-Vaccaro's canonicity-via-correspondence.
	
	\label{ssec:expanded tense language}
	Formally, any given language $\mathcal{L}_\mathrm{LE} = \mathcal{L}_\mathrm{LE}(\mathcal{F}, \mathcal{G})$ can be associated with the language $\mathcal{L}_\mathrm{LE}^* = \mathcal{L}_\mathrm{LE}(\mathcal{F}^*, \mathcal{G}^*)$, where $\mathcal{F}^*\supseteq \mathcal{F}$ and $\mathcal{G}^*\supseteq \mathcal{G}$ are obtained by expanding $\mathcal{L}_\mathrm{LE}$ with the following connectives:
	\begin{enumerate}
		\item the $n_f$-ary connective $f^\sharp_i$ for $1\leq i\leq n_f$, the intended interpretation of which is the right residual of $f\in\mathcal{F}$ in its $i$th coordinate if $\varepsilon_f(i) = 1$ (resp.\ its Galois-adjoint if $\varepsilon_f(i) = \partial$);
		\item the $n_g$-ary connective $g^\flat_i$ for $1\leq i\leq n_g$, the intended interpretation of which is the left residual of $g\in\mathcal{G}$ in its $i$th coordinate if $\varepsilon_g(i) = 1$ (resp.\ its Galois-adjoint if $\varepsilon_g(i) = \partial$).
		% $ g^\flat_j$ for each and $g\in \mathcal{G}$, where and $0\leq j\leq n_g$ ($f^\sharp_i$ is the right residual of $f$ in the $i$-th coordinate, and $g^\flat_j$ is the left residual of $g$ in the $j$-th coordinate).
	
	\end{enumerate}
	We stipulate that
	$f^\sharp_i\in\mathcal{G}^*$ if $\varepsilon_f(i) = 1$, and $f^\sharp_i\in\mathcal{F}^*$ if $\varepsilon_f(i) = \partial$. Dually, $g^\flat_i\in\mathcal{F}^*$ if $\varepsilon_g(i) = 1$, and $g^\flat_i\in\mathcal{G}^*$ if $\varepsilon_g(i) = \partial$. The order-type assigned to the additional connectives is predicated on the order-type of their intended interpretations. That is, for any $f\in \mathcal{F}$ and $g\in\mathcal{G}$,
	%each $g^\flat_j\in\mathcal{F}$, for each coordinate $i$ in $f$ or $g$,
%	\begin{enumerate}
%		\item if $\epsilon_f(i) = 1$, then $\epsilon_{f_i^\sharp}(i) = 1$ and $\epsilon_{f_i^\sharp}(j) = (\epsilon_f(j))^\partial$ for any $j\neq i$.
%		\item if $\epsilon_f(i) = \partial$, then $\epsilon_{f_i^\sharp}(i) = \partial$ and $\epsilon_{f_i^\sharp}(j) = \epsilon_f(j)$ for any $j\neq i$.
%		\item if $\epsilon_g(i) = 1$, then $\epsilon_{g_i^\flat}(i) = 1$ and $\epsilon_{g_i^\flat}(j) = (\epsilon_g(j))^\partial$ for any $j\neq i$.
%		\item if $\epsilon_g(i) = \partial$, then $\epsilon_{g_i^\flat}(i) = \partial$ and $\epsilon_{g_i^\flat}(j) = \epsilon_g(j)$ for any $j\neq i$.
%	\end{enumerate}

\begin{enumerate}
\item if $\epsilon_f(i) = 1$, then $\epsilon_{f_i^\sharp}(i) = 1$ and
$\epsilon_{f_i^\sharp}(j) = \epsilon_f^\partial(j)$ for any $j\neq i$.
\item if $\epsilon_f(i) = \partial$, then $\epsilon_{f_i^\sharp}(i) =
\partial$ and $\epsilon_{f_i^\sharp}(j) = \epsilon_f(j)$ for any $j\neq i$.
\item if $\epsilon_g(i) = 1$, then $\epsilon_{g_i^\flat}(i) = 1$ and
$\epsilon_{g_i^\flat}(j) = \epsilon_g^\partial(j)$ for any $j\neq i$.
\item if $\epsilon_g(i) = \partial$, then $\epsilon_{g_i^\flat}(i) =
\partial$ and $\epsilon_{g_i^\flat}(j) = \epsilon_g(j)$ for any $j\neq i$.
\end{enumerate}
	
	For instance, if $f$ and $g$ are binary connectives such that $\varepsilon_f = (1, \partial)$ and $\varepsilon_g = (\partial, 1)$, then $\varepsilon_{f^\sharp_1} = (1, 1)$, $\varepsilon_{f^\sharp_2} = (1, \partial)$, $\varepsilon_{g^\flat_1} = (\partial, 1)$ and $\varepsilon_{g^\flat_2} = (1, 1)$.
	
	\begin{remark}
		We warn the reader that the notation introduced above depends on which connective is taken as primitive, and needs to be carefully adapted to well known cases. For instance, consider the  `fusion' connective $\circ$ (which, when denoted  as $f$, is such that $\varepsilon_f = (1, 1)$). Its residuals
		$f_1^\sharp$ and $f_2^\sharp$ are commonly denoted $/$ and
		$\backslash$ respectively. However, if $\backslash$ is taken as the primitive connective $g$, then $g_2^\flat$ is $\circ = f$, and $g_1^\flat(x_1, x_2): = x_2/x_1 = f_1^\sharp (x_2, x_1)$. This example shows that, when identifying $g_1^\flat$ and $f_1^\sharp$, the conventional order of the coordinates is not preserved, and depends on which connective is taken as primitive.
	\end{remark}

	\begin{definition}\label{def_residuated_lattice}
		For any language $\mathcal{L}_\mathrm{LE}(\mathcal{F}, \mathcal{G})$, the {\em basic} $\mathcal{L}_\mathrm{LE}$-{\em logic with residuals} is defined by specializing Definition \ref{def:DLE:logic:general} to the language $\mathcal{L}_\mathrm{LE}^* = \mathcal{L}_\mathrm{LE}(\mathcal{F}^*, \mathcal{G}^*)$ %a set of sequents $\phi\vdash\psi$ with $\phi,\psi\in\mathcal{L}_\mathrm{DLE}^*$, which contains the axioms of the DLE-logic $\mathbf{L}_\mathrm{DLE}$, and is closed under rules for DLE-logics plus
		and closing under the following residuation rules for each $f\in \mathcal{F}$ and $g\in \mathcal{G}$ with $n_f, n_g \geq 1$:

			\[(\epsilon_f(i) = 1) \ \Tfrac{(\varphi_1,\ldots,\phi,\ldots, \varphi_{n_f}) \vdash \psi}{\phi\vdash f^\sharp_i(\varphi_1,\ldots,\psi,\ldots,\varphi_{n_f})} \> \> \>  \> \Tfrac{\phi \vdash g(\varphi_1,\ldots,\psi,\ldots,\varphi_{n_g})}{\phi \vdash g(\varphi_1,\ldots,\psi,\ldots,\varphi_{n_g})} \ (\epsilon_g(i) = 1)\]
			\[ (\epsilon_f(i) = \partial) \ \Tfrac{f(\varphi_1,\ldots,\phi,\ldots, \varphi_{n_f}) \vdash \psi}{f^\sharp_i(\varphi_1,\ldots,\psi,\ldots,\varphi_{n_f})\vdash \phi}  \> \> \>  \>  \Tfrac{\phi \vdash g(\varphi_1,\ldots,\psi,\ldots,\varphi_{n_g})}{\psi\vdash g^\flat_i(\varphi_1,\ldots, \phi,\ldots, \varphi_{n_g})} \ (\epsilon_g(i) = \partial) \]

		The double line in each rule above indicates that the rule should be read both top-to-bottom and bottom-to-top.
		Let $\mathbf{L}_\mathrm{LE}^*$ be the minimal basic $\mathcal{L}_\mathrm{LE}$-logic with residuals. %\footnote{\label{ftn: definition basic logic for expanded language} Hence, for any language $\mathcal{L}_\mathrm{DLE}$, there are in principle two logics associated with the expanded language $\mathcal{L}_\mathrm{DLE}^*$, namely the {\em minimal} $\mathcal{L}_\mathrm{DLE}^*$-logic, which we denote by $\mathbf{L}_\mathrm{DLE}^{\underline{*}}$, and which is obtained by instantiating Definition \ref{def:DLE:logic:general} to the language $\mathcal{L}_\mathrm{DLE}^*$, and the bi-intuitionistic `tense' logic $\mathbf{L}_\mathrm{DLE}^*$, defined above. The logic $\mathbf{L}_\mathrm{DLE}^*$ is the natural logic on the language $\mathcal{L}_\mathrm{DLE}^*$, however it is useful to introduce a specific notation for $\mathbf{L}_\mathrm{DLE}^{\underline{*}}$, given that all the results holding for the minimal logic associated with an arbitrary DLE-language can be instantiated to the expanded language $\mathcal{L}_\mathrm{DLE}^*$ and will then apply to $\mathbf{L}_\mathrm{DLE}^{\underline{*}}$.}
For any language $\mathcal{L}_{\mathrm{LE}}$, by an {\em $\mathcal{L}_{\mathrm{LE}}$-logic with residuals} we understand any axiomatic extension of the basic $\mathcal{L}_{\mathrm{LE}}$-logic with residuals in $\mathcal{L}^*_{\mathrm{LE}}$.
	\end{definition}
	
	The algebraic semantics of $\mathbf{L}_\mathrm{LE}^*$ is given by the class of $\mathcal{L}_\mathrm{LE}$-algebras with residuals, defined as tuples $\bba = (A, \mathcal{F}^*, \mathcal{G}^*)$ such that $A$ is a lattice, and moreover,
		\begin{enumerate}
			
			\item for every $f\in \mathcal{F}$ s.t.\ $n_f\geq 1$, all $a_1,\ldots,a_{n_f}\in A$ and $b\in A$, and each $1\leq i\leq n_f$,
			\begin{itemize}
				\item
				if $\epsilon_f(i) = 1$, then $f(a_1,\ldots,a_i,\ldots a_{n_f})\leq b$ iff $a_i\leq f^\sharp_i(a_1,\ldots,b,\ldots,a_{n_f})$;
				\item
				if $\epsilon_f(i) = \partial$, then $f(a_1,\ldots,a_i,\ldots a_{n_f})\leq b$ iff $a_i\leq^\partial f^\sharp_i(a_1,\ldots,b,\ldots,a_{n_f})$.
				
			\end{itemize} We say that $f^\sharp_i$ is the \emph{right residual} of $f$ in its $i$\textsuperscript{th} coordinate.
			\item for every $g\in \mathcal{G}$ s.t.\ $n_g\geq 1$, any $a_1,\ldots,a_{n_g}\in A$ and $b\in A$, and each $1\leq i\leq n_g$,
			\begin{itemize}
				\item if $\epsilon_g(i) = 1$, then $b\leq g(a_1,\ldots,a_i,\ldots a_{n_g})$ iff $g^\flat_i(a_1,\ldots,b,\ldots,a_{n_g})\leq a_i$.
				\item
				if $\epsilon_g(i) = \partial$, then $b\leq g(a_1,\ldots,a_i,\ldots a_{n_g})$ iff $g^\flat_i(a_1,\ldots,b,\ldots,a_{n_g})\leq^\partial a_i$.
			\end{itemize}
			We say that $g_i^\flat$ is the \emph{left residual} of $g$ in its $i$\textsuperscript{th} coordinate. 
		\end{enumerate}
		It is also routine to prove using the Lindenbaum-Tarski construction that $\mathbf{L}_\mathrm{LE}^*$ (as well as any of its axiomatic extensions) is sound and complete w.r.t.\ the class of   $\mathcal{L}_\mathrm{LE}$-algebras with residuals (w.r.t.\ the suitably defined equational subclass, respectively). % and that every consistent DLE$^*$-logic is characterized by its algebras.
		
		\begin{definition} 
		\label{def: RA and LA}
			%
			%\marginnote{controllare questa def che ho editato}
			For every definite positive PIA $\mathcal{L}_{\mathrm{LE}}$-formula $\phi = \phi(!x, \oz)$, and any definite negative PIA $\mathcal{L}_{\mathrm{LE}}$-formula $\psi = \psi(!x, \oz)$ such that $x$ occurs in them exactly once, the $\mathcal{L}^\ast_\mathrm{LE}$-formulas $\mathsf{LA}(\phi)(u, \oz)$ and $\mathsf{RA}(\psi)(u, \oz)$ (for $u \in Var - (x \cup \oz)$) are defined by simultaneous recursion as follows:
			
			\begin{center}
				\begin{tabular}{r c l}
					
					$\mathsf{LA}(x)$ &= &$u$;\\
				%	$\mathsf{LA}(\Box \phi(x, \oz))$ &= &$\mathsf{LA}(\phi)(\Diamondblack u, \overline{z})$;\\
				%	$\mathsf{LA}(\psi(\oz) \rightarrow \phi(x, \oz))$ &= &$\mathsf{LA}(\phi)(u \wedge \psi(\oz), \oz)$;\\
				%	$\mathsf{LA}(\phi_1(\oz) \vee \phi_2(x, \oz))$ &= &$\mathsf{LA}(\phi_2)(u - \phi_1(\oz), \oz)$;\\
				%	$\mathsf{LA}(\psi(x, \oz)\rightarrow\phi(\oz))$ &= &$\mathsf{RA}(\psi)(u \rightarrow \phi(\oz), \oz)$;\\
					$\mathsf{LA}(g(\overline{\phi_{-j}(\oz)},\phi_j(x,\oz), \overline{\psi(\oz)}))$ &= &$\mathsf{LA}(\phi_j)(g^{\flat}_{j}(\overline{\phi_{-j}(\oz)},u, \overline{\psi(\oz)} ), \oz)$;\\
					$\mathsf{LA}(g(\overline{\phi(\oz)}, \overline{\psi_{-j}(\oz)},\psi_j(x,\oz)))$ &= &$\mathsf{RA}(\psi_j)(g^{\flat}_{j}(\overline{\phi(\oz)}, \overline{\psi_{-j}(\oz)},u), \oz)$;\\
					&&\\
					$\mathsf{RA}(x)$ &= &$u$;\\
					%$\mathsf{RA}(\Diamond \psi(x, \oz))$ &= &$\mathsf{RA}(\psi)(\blacksquare u, \overline{z})$;\\
					%$\mathsf{RA}(\psi(x, \oz) - \phi(\oz))$ &= &$\mathsf{RA}(\psi)(\phi(\oz) \vee u, \oz)$;\\
					%$\mathsf{RA}(\psi_1(\oz) \wedge \psi_2(x, \oz))$ &= &$\mathsf{RA}(\psi_2)(\psi_1(\oz) \rightarrow u, \oz)$;\\
					%$\mathsf{RA}(\psi(\oz) - \phi(x, \oz))$ &= &$\mathsf{LA}(\phi)(\psi(\oz) - u, \oz)$;\\
					$\mathsf{RA}(f(\overline{\psi_{-j}(\oz)},\psi_j(x,\oz), \overline{\phi(\oz)}))$ &= &$\mathsf{RA}(\psi_j)(f^{\sharp}_{j}(\overline{\psi_{-j}(\oz)},u, \overline{\phi(\oz)} ), \oz)$;\\
					$\mathsf{RA}(f(\overline{\psi(\oz)}, \overline{\phi_{-j}(\oz)},\phi_j(x,\oz)))$ &= &$\mathsf{LA}(\phi_j)(f^{\sharp}_{j}(\overline{\psi(\oz)}, \overline{\phi_{-j}(\oz)},u), \oz)$.\\
				\end{tabular}
			\end{center}
			Above, $\overline{\phi_{-j}}$ denotes the vector obtained by removing the $j$th coordinate of $\overline{\phi}$.
			\end{definition}

\begin{lemma}
\label{lemma:polarities of la-ra}
For every definite positive PIA $\mathcal{L}_{\mathrm{LE}}$-formula $\phi = \phi(!x, \oz)$, and any definite negative PIA $\mathcal{L}_{\mathrm{LE}}$-formula $\psi = \psi(!x, \oz)$ such that $x$ occurs in them exactly once, 
\begin{enumerate}
\item if $+x\prec +\phi$ then  $\mathsf{LA}(\phi)(u, \oz)$ is monotone in $u$ and for each $z$ in $\oz$, $\mathsf{LA}(\phi)(u, \oz)$ has the opposite polarity to the polarity of $\phi$ in $z$;
\item if $-x\prec +\phi$ then  $\mathsf{LA}(\phi)(u, \oz)$ is antitone in $u$ and for each $z$ in $\oz$, $\mathsf{LA}(\phi)(u, \oz)$ has the same polarity as $\phi$ in $z$;
\item if $+x\prec +\psi$ then  $\mathsf{RA}(\psi)(u, \oz)$ is monotone in $u$ and for each $z$ in $\oz$, $\mathsf{RA}(\psi)(u, \oz)$ has the opposite polarity to the polarity of $\psi$ in $z$;
\item if $-x\prec +\psi$ then  $\mathsf{RA}(\psi)(u, \oz)$ is antitone in $u$ and for each $z$ in $\oz$, $\mathsf{RA}(\psi)(u, \oz)$ has the same polarity as $\psi$ in $z$.
\end{enumerate}
\end{lemma}		
\begin{proof}
By simultaneous induction on $\phi$ and $\psi$. If $\phi = \psi = x$, then the assumptions of item 1 and 3 are satisfied; then $\mathsf{RA}(\psi) = \mathsf{LA}(\phi) = u$ is clearly monotone in $u$ and the second part of the statement is vacuously satisfied. 
As to the inductive step, if $\phi (!x, \oz)= g(\overline{\phi'_{-j}(\oz)},\phi'_j(x,\oz),\overline{\psi'(\oz)})$, with each  $\phi'$ in $\overline{\phi'}$ being positive PIA and  each $\psi'$ in $\overline{\psi'}$ being negative PIA, then  $g_j^\flat\in \mathcal{F}^\ast$ is monotone in its $j$th coordinate and has the opposite polarity of $\epsilon_g$ in all the other coordinates. Hence, $g_j^\flat(\overline{\phi'_{-j}(\oz)},u,\overline{\psi'(\oz)})$ has the opposite polarity of $\phi (!x, \oz)$ in each $z$ in $\oz$. Two cases can occur: (a) if  $+x\prec +\phi_j$, then by induction hypothesis, $\mathsf{LA}(\phi_j)(u', \oz)$ is monotone in $u'$, and has the opposite polarity of $\phi_j$ in every $z$ in $\oz$. Hence,
\[ \mathsf{LA}(\phi) = \mathsf{LA}(\phi_j)(g_j^\flat(\overline{\phi'_{-j}(\oz)},u,\overline{\psi'(\oz)})/u', \oz) \]
is monotone in $u$ and has the opposite polarity to the polarity of $\phi$ in each $z$ in $\oz$.
(b) if $-x\prec +\phi_j$, then  by induction hypothesis, $\mathsf{LA}(\phi_j)(u', \oz)$ is antitone in $u'$, and has the same polarity as $\phi_j$ in every $z$ in $\oz$. 
 Hence,
\[ \mathsf{LA}(\phi) = \mathsf{LA}(\phi_j)(g_j^\flat(\overline{\phi'_{-j}(\oz)},u,\overline{\psi'(\oz)})/u', \oz) \]
is antitone in $u$ and has the same polarity as  $\phi$ in each $z$ in $\oz$.
The remaining cases are $\phi: =g(\overline{\phi'(\oz)}, \overline{\psi'_{-h}(\oz)}, \psi_h(x, \oz))$, $\psi: = f(\overline{\phi'_{-j}(\oz)}, \phi'_j(x, \oz), \overline{\psi'(\oz)})$, and $\psi: = f(\overline{\phi'(\oz)}, \overline{\psi'_{-h}(\oz)}, \psi'_h(x, \oz))$ and are shown in a similar way.
\end{proof}

\subsection{The language of non-distributive ALBA}\label{Subsec:Expanded:Land}

The expanded language of perfect LEs will include the connectives corresponding to all the residual of the original connectives, as well as a denumerably infinite set of sorted variables $\mathsf{NOM}$ called {\em nominals},
ranging over the completely join-irreducible elements of perfect LEs (or, constructively, on the closed elements of the constructive canonical extensions, as in \cite{CoPa-constructive}), and a denumerably infinite set of
sorted variables $\mathsf{CO\text{-}NOM}$, called {\em co-nominals}, ranging over the completely meet-irreducible elements of perfect LEs (or, constructively on the open elements of the constructive canonical extensions). The elements of $\mathsf{NOM}$ will be denoted with $\nomi, \nomj$, possibly indexed, and those of
$\mathsf{CO\text{-}NOM}$ with $\cnomm, \cnomn$, possibly indexed.

%Via duality, nominals and co-nominals range over subsets of states of relational structures. For the sake of correspondence it is essential that these subsets be first-order definable. For example, in the case of RS-frames they are Galois closures of singletons while, in the distributive case, they become principal upsets in distributive frames or intuitionistic Kripke frames. This also allows us syntactic access to (co)states, albeit indirectly. We discuss of an alternative interpretations of nominals and co-nominals as closed and open elements in the paragraph on constructive canonicity in Section \ref{sec:conclusions}.

Let us introduce the expanded language formally: the \emph{formulas} $\phi$ of $\mathcal{L}_\mathrm{LE}^{+}$ are given by the following recursive definition:
\begin{center}
\begin{tabular}{r c |c|c|c|c|c|c c c c c c c}
$\phi ::= $ &$\nomj$ & $\cnomm$ & $\psi$ & $\phi\wedge\phi$ & $\phi\vee\phi$ & $f(\overline{\phi})$ &$g(\overline{\phi})$
\end{tabular}
\end{center}
with $\psi  \in \mathcal{L}_\mathrm{LE}$, $\nomj \in \mathsf{NOM}$ and $\cnomm \in \mathsf{CO\text{-}NOM}$,  $f\in \mathcal{F}^*$ and $g\in \mathcal{G}^*$.
%\marginnote{questa notazione e tavole possiamo evitarle perche' non si usano}
As in the case of $\mathcal{L}_\mathrm{LE}$, we can form inequalities and quasi-inequalities based on $\mathcal{L}_\mathrm{LE}^{+}$. 
\begin{comment}
Let $\mathcal{L}_\mathrm{LE}^{+\leq}$ and $\mathcal{L}_\mathrm{LE}^{+\mathit{quasi}}$ respectively denote the set of inequalities between terms in $\mathcal{L}_\mathrm{LE}^+$, and  the set of quasi-inequalities formed out of $\mathcal{L}_\mathrm{LE}^{+\leq}$.
%Members of $\mathcal{L}_\mathrm{LE}^+$, $\mathcal{L}_\mathrm{LE}^{+\leq}$, and $\mathcal{L}_\mathrm{LE}^{+\mathit{quasi}}$  not containing any propositional variables (but possibly containing nominals and co-nominals) will be called \emph{pure}.

%\noindent Summing up, we will be working with six sets of syntactic objects, as reported in the following table:

\begin{center}
\begin{tabular}{| c || l | l |}
\hline &Base language & Expanded Language\\
%
\hline \hline Formulas / terms &${\mathcal{L}_\mathrm{LE}}$ &${\mathcal{L}_\mathrm{LE}^+}$\\
%
\hline Inequalities &$\mathcal{L}_\mathrm{LE}^{\leq}$ &$\mathcal{L}_\mathrm{LE}^{+\leq}$\\
%
\hline Quasi-inequalities &$\mathcal{L}_\mathrm{LE}^{\mathit{quasi}}$ &$\mathcal{L}_\mathrm{LE}^{+\mathit{quasi}}$\\
%
\hline
\end{tabular}
\end{center}
\end{comment}
If $\mathbb{A}$ is a perfect LE,  then an \emph{assignment} for ${\mathcal{L}_\mathrm{LE}^+}$ on $\mathbb{A}$ is a map $V: \mathsf{PROP} \cup \mathsf{NOM} \cup \mathsf{CO\mbox{-}NOM} \rightarrow \bbA$ sending propositional variables to elements of $\mathbb{A}$, sending nominals to $\jty(\mathbb{A})$ and co-nominals to $\mty(\mathbb{A})$. For any LE $\mathbb{A}$, an \emph{admissible assignment}\label{admissible:assignment} for ${\mathcal{L}_\mathrm{LE}^+}$  on $\mathbb{A}$ is an assignment $V$ for ${\mathcal{L}_\mathrm{LE}^+}$ on $\mathbb{A}^{\delta}$, such that $V(p) \in \bba$ for each $p \in \mathsf{PROP}$. In other words, the assignment $V$ sends propositional variables to elements of the subalgebra $\bba$, while nominals and co-nominals get sent to the completely join-irreducible (resp.~closed) and the completely meet-irreducible (resp.~open) elements of $\bbas$, respectively. %This means that the value of ${\mathcal{L}_\mathrm{LE}}$-terms under an admissible assignment will belong to $\bba$, whereas ${\mathcal{L}_\mathrm{LE}^+}$-terms in general will not. 

\subsection{Non-distributive ALBA on analytic inductive LE-inequalities}\label{Spec:Alg:Section} 

In this subsection, we describe a successful ALBA-run on an analytic $(\Omega, \epsilon)$-inductive 
 $\mathcal{L}_\mathrm{LE}$-inequality $\phi \leq \psi$. The procedure described below serves {\em both} to compute the first order correspondent of the given inequality in various semantic settings, as discussed e.g.~in \cite{CoPa-nondist, CoPa-Dist, conradie2019rough, conradie2020non}, {\em and} to compute the shape of the analytic structural rules corresponding to the given inequality, as discussed in \cite{GMPTZ, syntactic-completeness}.
 
 The run proceeds in three stages. The first stage preprocesses $\phi \leq \psi$ by eliminating all uniformly  occurring propositional variables, and applying distribution and splitting rules exhaustively. This produces a finite set of inequalities, $\phi'_i \leq \psi'_i$, $1 \leq i \leq n$, from which  ALBA forms the \emph{initial quasi-inequalities}. % $\bigamp S_i \Rightarrow \sf{Ineq}_i$, %compactly represented as tuples $(S_i, \sf{Ineq}_i)$ referred as \emph{systems},  with each $S_i$ initialized to the empty set and $\sf{Ineq}_i$ initialized to $\phi'_i \leq \psi'_i$.

The second stage (called the reduction  stage) transforms the quasi-inequalities %$S_i$ and $\mathsf{Ineq}_i$ 
through the application of transformation rules, which are listed below. The aim is to eliminate all  propositional variables %from $S_i$ and $\mathsf{Ineq}_i$ 
in favour of terms built from constants,  nominals and co-nominals (for an expanded discussion on the general reduction strategy, the reader is referred to \cite{UnifCorresp,ConPalSur}). A system for which this has been done will be called \emph{pure} or \emph{purified}. The actual eliminations are effected through the Ackermann-rules, while the other rules are used to bring the quasi-inequalities into the appropriate shape which make these applications possible. %Once all propositional variables have been eliminated, this phase terminates and returns the pure quasi-inequalities % $\bigamp S_i \Rightarrow \mathsf{Ineq}_i$.

The third stage either reports failure if some system could not be purified, or else returns the conjunction of the pure quasi-inequalities %$\bigamp S_i \Rightarrow \mathsf{Ineq}_i$, 
which we denote by $\mathsf{ALBA}(\phi \leq \psi)$.
We now outline each of the three stages in more detail.

\subsection{Stage 1: Preprocessing and initialization} ALBA receives an analytic $(\Omega, \epsilon)$-inductive 
 $\mathcal{L}_\mathrm{LE}$-inequality $\phi \leq \psi$ as input. It applies the following {\bf rules for elimination of monotone variables}  to $\phi \leq \psi$ exhaustively, in order to eliminate any propositional variables which occur uniformly:
 
 \[ \frac{\alpha(p) \leq \beta(p)}{\alpha(\top) \leq \beta(\top)} \ \ \ \ \ \ \ \frac{\gamma(p) \leq \delta(p)}{\gamma(\bot) \leq \delta(\bot)}\]
for $\alpha(p) \leq \beta(p)$ $1$-uniform in $p$ and $\gamma(p) \leq \delta(p)$ $\partial$-uniform in $p$, respectively (see the discussion after Definition \ref{def: signed gen tree}).

Next, ALBA exhaustively distributes $f\in \mathcal{F}$ over $+\vee$ in its positive coordinates and over $-\wedge$ in its negative coordinates, and  $g\in \mathcal{G}$ over $-\wedge$ in its positive coordinates and over $+\vee$ in its negative coordinates, so as to bring occurrences of $+\vee$ and $-\wedge$ to the surface wherever this is possible, and then eliminate them via exhaustive applications of {\em splitting} rules.
\paragraph{Splitting-rules}

\[\frac{\alpha \leq \beta \wedge \gamma
}{\alpha \leq \beta \quad \alpha \leq \gamma} \ \ \ \ \ \ \  \frac{\alpha \vee \beta \leq \gamma}{\alpha \leq \gamma \quad \beta \leq \gamma}
\]

This gives rise to  a set of {\em definite} analytic inductive inequalities $\{\phi_i' \leq \psi_i'\mid 1\leq i\leq n\}$, each of which will be treated separately.

Next, in each PIA-subformula of each such definite analytic inductive inequality, ALBA exhaustively distributes $-f\in \mathcal{F}$ over $-\vee$ in its positive coordinates and over $+\wedge$ in its negative coordinates, and  $+g\in \mathcal{G}$ over $+\wedge$ in its positive coordinates and over $-\vee$ in its negative coordinates, so as to bring occurrences of $+\vee$ and $-\wedge$ as close as possible to the root of each PIA subformula. Let $(\varphi\leq \psi)[\overline{\alpha}/!\overline{x}, \overline{\beta}/!\overline{y},\overline{\gamma}/!\overline{z}, \overline{\delta}/!\overline{w}]$ denote one of the inequalities resulting from this step (we suppress the indices). Now ALBA transforms $(\varphi\leq \psi)[\overline{\alpha}/!\overline{x}, \overline{\beta}/!\overline{y},\overline{\gamma}/!\overline{z}, \overline{\delta}/!\overline{w}]$ into the following \emph{initial quasi-inequality} (the soundness of these steps on perfect LEs, or constructive canonical extensions, has been discussed in \cite[Section 6]{CoPa-nondist} and \cite[Section 5]{CoPa-constructive}):
\begin{equation}
\label{eq:initial quasi-inequality}
 \forall\overline{\nomj} \forall\overline{\cnomm}\forall\overline{\nomi}\forall\overline{\cnomn}(( \overline{\nomj}\leq \overline{\alpha}\ \&\ \overline{\beta}\leq \overline{\cnomm} \ \&\ \overline{\nomi}\leq \overline{\gamma} \ \&\ \overline{\delta}\leq \overline{\cnomn})\Rightarrow (\varphi\leq \psi)[!\overline{\nomj}/!\overline{x}, !\overline{\cnomm}/!\overline{y},!\overline{\nomi}/!\overline{z}, !\overline{\cnomn}/!\overline{w}] ).
\end{equation}
 In the quasi-inequality above, symbols such as $\overline{\nomj}\leq \overline{\alpha}$ denote the conjunction of inequalities of the form $\nomj_k\leq \alpha_k$ for each $\nomj_k$ in $\overline{\nomj}$ and $\alpha_k$ in $\overline{\alpha}$. 
%the \emph{initial quasi-inequalities} $\bigamp S_i \Rightarrow \sf{Ineq}_i$, compactly represented as tuples $(S_i, \sf{Ineq}_i)$ referred as \emph{systems},  with each $S_i$ initialized to the empty set and $\sf{Ineq}_i$ initialized to $\phi'_i \leq \psi'_i$. 
Before passing each initial quasi-inequality  separately to stage 2 (described below), by exhaustively applying splitting rules to the top-most nodes of the formulas in $\overline{\alpha}$ and $\overline{\beta}$, we transform each quasi-inequality into one of similar shape as \eqref{eq:initial quasi-inequality} and in which each $\alpha$ in $\overline{\alpha}$ and each $\beta$ in $\overline{\beta}$ contains at most one critical occurrence. Hence, w.l.o.g.~we can assume that each $\alpha$ in $\overline{\alpha}$ and $\beta$ in $\overline{\beta}$ contains exactly one $\epsilon$-critical occurrence (since in case any of them does not, the corresponding inequality will be $\epsilon^\partial$-uniform, and hence it can be assimilated to the inequalities $\overline{\nomi}\leq \overline{\gamma}$ or $\overline{\delta}\leq \overline{\cnomn}$). Hence, we can represent the resulting quasi-inequality as follows:  % where we will suppress indices $i$.

\begin{equation}
\label{eq:initial quasi-inequality p and q}
\begin{array}{l}
 \forall\overline{\nomj} \forall\overline{\cnomm}\forall\overline{\nomi}\forall\overline{\cnomn}(( \overline{\nomj}\leq \overline{\alpha_p}\ \&\  \overline{\nomj}\leq \overline{\alpha_q}\ \&\ \overline{\beta_p}\leq \overline{\cnomm} \ \&\ \overline{\beta_q}\leq \overline{\cnomm} \ \&\ \overline{\nomi}\leq \overline{\gamma} \ \&\ \overline{\delta}\leq \overline{\cnomn})\\
\> \> \> \> \> \ \ \ \ \ \ \ \ \ \ \ \ \ \ \ \ \ \ \Rightarrow (\varphi\leq \psi)[!\overline{\nomj}/!\overline{x}, !\overline{\cnomm}/!\overline{y},!\overline{\nomi}/!\overline{z}, !\overline{\cnomn}/!\overline{w}] ),
 \end{array}
\end{equation}
where $\overline{p}$ (resp.~$\overline{q}$) is the vector of the atomic propositions in $\phi\leq \psi$ such that $\varepsilon(p) = 1$ (resp.~$\varepsilon(q) = \partial$), and %e.g.~$p$ is the one (if any) critical variable occurring in $\alpha_p$. 
the subscript in each PIA-formula in $\overline{\alpha}$ and $\overline{\beta}$ indicates the unique $\varepsilon$-critical propositional variable occurrence contained in that formula.

\subsection{Stage 2: Reduction and elimination}\label{Sec:ReductionElimination}

The aim of this stage is to eliminate all occurring propositional variables from a given initial quasi-inequality \eqref{eq:initial quasi-inequality}. This is done by means of the splitting rules, introduced above, as well as the following  \emph{residuation rules} and \emph{Ackermann-rules}. The rules applied in this subsection are collectively called \emph{reduction rules}. The terms and inequalities in this subsection are from $\mathcal{L}_\mathrm{LE}^{+}$.

\paragraph{Residuation rules} These rules operate on the inequalities in $S$, by rewriting a chosen inequality in $S$ into another inequality. For every $f\in \mathcal{F}$ and $g\in \mathcal{G}$, and any $1\leq i\leq n_f$ and $1\leq j\leq n_g$,

\[ \frac{f(\phi_1,\ldots,\phi_i,\ldots,\phi_{n_f}) \leq \psi }{\phi_i\leq f^\sharp_i(\phi_1,\ldots,\psi,\ldots,\phi_{n_f})} \ \epsilon_f(i) = 1  \ \ \ \ \ \ \  \frac{f(\phi_1,\ldots,\phi_i,\ldots,\phi_{n_f}) \leq \psi }{f^\sharp_i(\phi_1,\ldots,\psi,\ldots,\phi_{n_f})\leq \phi_i} \ \epsilon_f(i) = \partial \]
\[ \frac{\psi\leq g(\phi_1,\ldots,\phi_i,\ldots,\phi_{n_g})}{g^\flat_i(\phi_1,\ldots,\psi,\ldots,\phi_{n_g})\leq \phi_i}\ \epsilon_g(i) = 1  \ \ \ \ \ \ \ 
\frac{\psi\leq g(\phi_1,\ldots,\phi_i,\ldots,\phi_{n_g})}
{\phi_i\leq g^\flat_i(\phi_1,\ldots,\psi,\ldots,\phi_{n_g})} \ {\epsilon_g(i) = \partial} \]

\paragraph{Right Ackermann-rule} $\phantom{a}$%If $\mathsf{Ineq} = \phi \leq \psi$ and $\phi = \phi'(\gamma / !x)$, and
\[ \frac{ (\{\alpha_i \leq p \mid 1 \leq i \leq n \} \cup \{ \beta_j(p)\leq \gamma_j(p) \mid 1 \leq j \leq m \}, \;\; \mathsf{Ineq})}{(\{ \beta_j(\bigvee_{i=1}^n \alpha_i)\leq \gamma_j(\bigvee_{i=1}^n \alpha_i) \mid 1 \leq j \leq m \},\;\; \mathsf{Ineq})}\ {(RAR)}
\]
where:
\begin{itemize}
\item $p$ does not occur in $\alpha_1, \ldots, \alpha_n$ or in $\mathsf{Ineq}$,
\item $\beta_{1}(p), \ldots, \beta_{m}(p)$ are positive in $p$, and
\item $\gamma_{1}(p), \ldots, \gamma_{m}(p)$ are negative in $p$.

\end{itemize}

\paragraph{Left Ackermann-rule}$\phantom{a}$%If $\mathsf{Ineq} = \phi \leq \psi$ and $\phi = \phi'(\gamma / !x)$, and
\[ \frac{(\{ p \leq \alpha_i \mid 1 \leq i \leq n \} \cup \{ \beta_j(p)\leq \gamma_j(p) \mid 1 \leq j \leq m \}, \;\; \mathsf{Ineq})}{\{ \beta_j(\bigwedge_{i=1}^n \alpha_i)\leq \gamma_j(\bigwedge_{i=1}^n \alpha_i) \mid 1 \leq j \leq m \},\;\; \mathsf{Ineq})}\ (LAR)
\]
where:
\begin{itemize}
\item $p$ does not occur in $\alpha_1, \ldots, \alpha_n$ or in $\mathsf{Ineq}$,
\item $\beta_{1}(p), \ldots, \beta_{m}(p)$ are negative in $p$, and
\item $\gamma_{1}(p), \ldots, \gamma_{m}(p)$ are positive in $p$.

\end{itemize}

By applying adjunction and residuation rules on all PIA-formulas $\alpha$ and $\beta$, the antecedent of \eqref{eq:initial quasi-inequality p and q} can be equivalently written as follows (cf.~Definition \ref{def: RA and LA}):  
\begin{equation}
\label{eq: adj anind}
\begin{array}{l}
 \overline{\mathsf{LA}(\alpha_p)[\nomj/u, \overline{p},\overline{q}]}\leq\overline{p}\ \&\   \overline{\mathsf{RA}(\beta_p)[\cnomm/u, \overline{p},\overline{q}]}\leq \overline{p}\ \&\ \overline{q}\leq \overline{\mathsf{LA}(\alpha_q)[\nomj/u, \overline{p},\overline{q}]}\ \\

  \> \>  \> \ \&\ \overline{q}\leq\overline{\mathsf{RA}(\beta_q)[\cnomm/u, \overline{p},\overline{q}]}
\&\ \overline{\nomi}\leq \overline{\gamma}\ \&\ \overline{\delta}\leq \overline{\cnomn}.
\end{array}
 \end{equation}
Notice that the `parametric' (i.e.~non-critical) variables in $\overline{p}$ and $\overline{q}$ actually occurring in each formula $\mathsf{LA}(\alpha_p)[\nomj/u, \overline{p},\overline{q}]$, $\mathsf{RA}(\beta_p)[\cnomm/u, \overline{p},\overline{q}]$, $\mathsf{LA}(\alpha_q)[\nomj/u, \overline{p},\overline{q}]$, and $\mathsf{RA}(\beta_q)[\cnomm/u, \overline{p},\overline{q}]$ are those that are strictly $<_\Omega$-smaller than the (critical) variable indicated in the subscript of the given PIA-formula. After applying adjunction and residuation as indicated above, the resulting quasi-inequality is in Ackermann shape relative to the $<_\Omega$-minimal variables.

For every $p\in\overline{p}$ and  $q\in\overline{q}$ let us define the sets $\mathsf{Mv}(p)$ and $\mathsf{Mv}(q)$ by recursion on $<_\Omega$ as follows:
\begin{itemize}
	\item
	$\mathsf{Mv}(p):=\{\mathsf{LA}(\alpha_p)[\nomj_k/u,\overline{\mathsf{mv}(p)}/\overline{p},\overline{\mathsf{mv}(q)}/\overline{q}], \mathsf{RA}(\beta_p)[\cnomm_h/u,\overline{\mathsf{mv}(p)}/\overline{p},\overline{\mathsf{mv}(q)}/\overline{q}]\mid 1\leq k\leq n_{i_1}, 1\leq h\leq n_{i_2}, \overline{\mathsf{mv}(p)}\in\prod_p \mathsf{Mv}(p),\overline{\mathsf{mv}(q)}\in\prod_q \mathsf{Mv}(q)  \}$
	\item $\mathsf{Mv}(q):=\{\mathsf{LA}(\alpha_q)[\nomj_h/u,\overline{\mathsf{mv}(p)}/\overline{p},\overline{\mathsf{mv}(q)}/\overline{q}], \mathsf{RA}(\beta_q)[\cnomm_k/u,\overline{\mathsf{mv}(p)}/\overline{p},\overline{\mathsf{mv}(q)}/\overline{q}]\mid 1\leq h\leq m_{j_1}, 1\leq k\leq m_{j_2}, \overline{\mathsf{mv}(p)}\in\prod_p \mathsf{Mv}(p),\overline{\mathsf{mv}(q)}\in\prod_q{\mathsf{Mv}(q)}  \}$
\end{itemize}
where,  $n_{i_1}$ (resp.~$n_{i_2}$) is the number of occurrences of $p$ in $\alpha$s (resp.~in $\beta$s) for every $p\in\overline{p}$, and $m_{j_1}$ (resp.~$m_{j_2}$) is the number of occurrences of $q$ in $\alpha$s (resp.~in $\beta$s) for every $q\in\overline{q}$. 
By induction on $<_\Omega$, we can apply the Ackermann rule exhaustively so as to eliminate all variables $\overline{p}$  and  $\overline{q}$.  Then the antecedent of the resulting {\em purified} quasi-inequality has the following form: % (recall that by assumption all formulas in  $\overline{\gamma}$ and $\overline{\delta}$ agree with  $\epsilon^\partial$):
 
\begin{equation}
\label{eq: after Ackermann anind}
\overline{\nomi}\leq \overline{\gamma}\left[\overline{\bigvee\mathsf{Mv}(p)}/\overline{p}, \overline{\bigwedge\mathsf{Mv}(q)}/\overline{q}\right]\quad \quad \overline{\delta}\left [\overline{\bigvee\mathsf{Mv}(p)}/\overline{p}, \overline{\bigwedge\mathsf{Mv}(q)}/\overline{q}\right]\leq \overline{\cnomn}.\end{equation}
Up to now, we have only made use of the assumption that the initial inequality is inductive, and not also analytic.  The next step is not needed for the elimination of propositional variables, since we have already reached a successful elimination. However, it will turn out to be useful when discussing canonicity. % \marginnote{Is it so???}

By assumption, $\varepsilon(p) = 1$ for every $p$ in $\overline{p}$ and $\varepsilon(q) = \partial$ for every $q$ in $\overline{q}$; recalling that every $+ \gamma$ (resp. $-\delta$) agrees with  $\epsilon^\partial$ and that $\gamma$ (resp. $\delta$) is positive (resp. negative) PIA for every $\gamma \in \overline{\gamma}$ (resp. $\delta \in \overline{\delta}$) (this is precisely what the analiticity assumption yields), the following semantic equivalences hold for each $\gamma$ in $\overline{\gamma}$ and $\delta$ in $\overline{\delta}$:
{\footnotesize
\[\gamma\left[\overline{\bigvee\mathsf{Mv}(p)}/\overline{p}, \overline{\bigwedge\mathsf{Mv}(q)}/\overline{q}\right] = \bigwedge \left\lbrace\gamma\left[\overline{\mathsf{mv}(p)}/\overline{p}, \overline{\mathsf{mv}(q)}/\overline{q}\right] \mid  \overline{\mathsf{mv}(p)} \in \prod_p{\mathsf{Mv}(p)}, \overline{\mathsf{mv}(q)} \in \prod_q{\mathsf{Mv}(q)}  \right\rbrace .\]
\[ \delta\left[\overline{\bigvee\mathsf{Mv}(p)}/\overline{p}, \overline{\bigwedge\mathsf{Mv}(q)}/\overline{q}\right] = \bigvee \left\lbrace\delta\left[\overline{\mathsf{mv}(p)}/\overline{p}, \overline{\mathsf{mv}(q)}/\overline{q}\right] \mid  \overline{\mathsf{mv}(p)} \in \prod_p{\mathsf{Mv}(p)}, \overline{\mathsf{mv}(q)} \in \prod_q{\mathsf{Mv}(q)} \right\rbrace.\]
}
Hence, by applying splitting, for every $\gamma$  in $\overline{\gamma}$ and $\delta$  in $\overline{\delta}$, the corresponding inequalities  in \eqref{eq: after Ackermann anind} can be equivalently replaced by (at most) $\Sigma_{n, m}(n_i m_j)$ inequalities of the form
\begin{equation}\label{last eq}
\nomi\leq \gamma\left[\overline{\mathsf{mv}(p)}/\overline{p}, \overline{\mathsf{mv}(q)}/\overline{q}\right]\quad \quad \delta\left[\overline{\mathsf{mv}(p)}/\overline{p}, \overline{\mathsf{mv}(q)}/\overline{q}\right]\leq\cnomn,
\end{equation}
 where $\gamma\left[\overline{\mathsf{mv}(p)}/\overline{p}, \overline{\mathsf{mv}(q)}/\overline{q}\right]$ is strictly syntactically open and $\delta\left[\overline{\mathsf{mv}(p)}/\overline{p}, \overline{\mathsf{mv}(q)}/\overline{q}\right]$ is strictly syntactically closed (cf.~Definition \ref{strictlySyn:Opn:Clsd:Definition} and Lemma \ref{Syn:Shape:Lemma}). %\marginnote{bisogna aggiungere il lemma dove questo viene provato}

%%%%%%%%%%%%%%%%
%%%%%%%%%%%%%%%%
%%%%%%%%%%%%%%%%
%%%%%%%%%%%%%%%%
%%%%%%%%%%%%%%%%

\section{Slanted LE-algebras and their canonical extensions}\label{subsec:slanted lattice:expansions}

\subsection{Basic definitions and properties}
\begin{definition}
\label{def:c-slanted o-slanted}
Let $A, B$ be lattices. For any $n_f\in \mathbb{N}$ and any order-type $\epsilon_f$ on $n_f$,  a coordinatewise finitely join-preserving $n_f$-ary map $f: B^{\epsilon_f}\to A^\delta$ is  {\em c-slanted} if its range is included in  $\kbbas$. When $B = A$, the map $f$ is a {\em c-slanted operation on} $A$. For any $n_g\in \mathbb{N}$ and any order-type $\epsilon_g$ on $n_g$,  a coordinatewise finitely meet-preserving $n_g$-ary map $g: B^{\epsilon_g}\to A^\delta$ is {\em o-slanted} if its range is included in  $\obbas$. When $B = A$, the map $g$ is an {\em o-slanted operation on} $A$.
\end{definition}

By definition, slanted maps are normal, in the sense of Definition \ref{def:DLE}, as maps $B^\epsilon\to A^\delta$. Examples of (properly) c-slanted (resp.~o-slanted) operations arise as the restrictions to the original algebra of the left (resp.~right) adjoints and residuals of the $\pi$-extensions (resp.~$\sigma$-extensions) of standard normal $g$-type (resp.~$f$-type) operations (cf.~\cite[Lemma 10.6]{CoPa-nondist}) when the signature $(\mathcal{F}, \mathcal{G})$ is not closed under adjoints and residuals.
	\begin{definition}
		\label{def:slanted LE}
		For any LE-signature $(\mathcal{F}, \mathcal{G})$, a {\em  slanted (distributive) lattice expansion} (abbreviated as slanted (D)LE or s-(D)LE) is a tuple $\bba = (A, \mathcal{F}^\bbA, \mathcal{G}^\bbA)$ such that $A$ is a bounded (distributive) lattice, $\mathcal{F}^\bbA = \{f^\bbA\mid f\in \mathcal{F}\}$ and $\mathcal{G}^\bbA = \{g^\bbA\mid g\in \mathcal{G}\}$, such that every $f^\bbA\in\mathcal{F}^\bbA$ (resp.\ $g^\bbA\in\mathcal{G}^\bbA$) is an $n_f$-ary (resp.\ $n_g$-ary) c-slanted (resp.~o-slanted) operation on $A$. A {\em  slanted Boolean algebra expansion} (abbreviated as slanted BAE or s-BAE) is a structure $\bba = (A, \mathcal{F}^\bbA, \mathcal{G}^\bbA)$ such that $\mathcal{F}^\bbA$ and $\mathcal{G}^\bbA$ are as above, and $A$ is a Boolean algebra.
		%An LE is {\em normal} if every $f^\bbA\in\mathcal{F}^\bbA$ (resp.\ $g^\bbA\in\mathcal{G}^\bbA$) preserves finite (hence also empty) joins (resp.\ meets) in each coordinate with $\epsilon_f(i)=1$ (resp.\ $\epsilon_g(i)=1$) and reverses finite (hence also empty) meets (resp.\ joins) in each coordinate with $\epsilon_f(i)=\partial$ (resp.\ $\epsilon_g(i)=\partial$).\footnote{\label{footnote:DLE vs DLO} Normal LEs are sometimes referred to as {\em  lattices with operators} (LOs). This terminology derives from the setting of Boolean algebras with operators, in which operators are understood as operations which preserve finite (hence also empty) joins in each coordinate. Thanks to the Boolean negation, operators are typically taken as primitive connectives, and all the other operations are reduced to these. However, this terminology results somewhat ambiguous in the lattice setting, in which primitive operations are typically maps which are operators if seen as $\bbA^\epsilon\to \bbA^\eta$ for some order-type $\epsilon$ on $n$ and some order-type $\eta\in \{1, \partial\}$. Rather than speaking of lattices with $(\varepsilon, \eta)$-operators, we then speak of normal LEs.} Let $\mathbb{LE}$ be the class of LEs. Sometimes we will refer to certain LEs as $\mathcal{L}_\mathrm{LE}$-algebras when we wish to emphasize that these algebras have a compatible signature with the logical language we have fixed.
	\end{definition}
	
	Slanted LEs generalize the standard notion of normal LE (cf.~Definition \ref{def:DLE}), as follows. Via the canonical embedding $e: A\to A^\delta$, and using compactness, it is not difficult to see that $e[A] =    K(A^\delta)\cap O(A^\delta)$ for any lattice $A$. Hence,  any standard normal operation $h$ on $A$ gives rise to  a slanted  operation $e\cdot h$ on $A$ which will be c-slanted  if $h$ is coordinatewise finitely join-preserving or meet-reversing,  and o-slanted if if $h$ is coordinatewise finitely meet-preserving or join-reversing. Conversely, any slanted operation on $A$ the range of which is included in $K(A^\delta)\cap O(A^\delta) = e[A]$ gives rise to a normal operation on $A$ in the standard sense. Hence, any standard LE $\bba$ can be `lifted' to a slanted LE $\bba^\star$ in the same signature by pre-composing all operations of $\bba$ with $e$, and any slanted LE $\mathbb{S}$ based on $A$ such that all  its operations target $K(A^\delta)\cap O(A^\delta) = e[A]$ gives rise to a standard LE $\mathbb{S}_\star$ in the same signature, and moreover, $(\bba^\star)_\star = \bba$ and $(\mathbb{S}_\star)^\star = \mathbb{S}$.
	%Based on this correspondence,  similar considerations apply in regard to how the definitions of canonical extensions of slanted maps (cf.~Definition \ref{def: sigma and pi extensions of slanted}), canonical extensions of slanted algebras (cf.~Definition \ref{def:canext slanted}),  satisfaction and validity of formulas or inequalities on slanted algebras (cf.~Definition \ref{def:slanted satisfaction and validity}) generalize the corresponding standard definitions.
	 
In the remainder of the paper,
we will abuse notation and write e.g.\ $f$ for $f^\bbA$ when this causes no confusion.
Slanted LEs constitute the main semantic environment of the present paper.

\begin{example}\label{Exem:gen_imp} Examples of slanted BAEs and LEs arise in connection with subordination algebras \cite{BezSouVe}, quasi-modal algebras \cite{Quasimodal} and generalized implication lattices \cite{CaCeJa}.  The slanted algebras arising from subordination and quasi-modal algebras will be described in detail in Section \ref{sec:subordination}. Let us consider here the case of generalized implications. %\marginnote{aggiungiamo sezione per quasi-modal - Ho fatto un remark nella sezione per le subordination - La 5.8, non so se e sufficiente ?}

A \emph{generalized implication lattice} \cite{CaCeJa} is a pair $\mathbb{L} = (L, \Rightarrow)$ such that  $L$ is a bounded distributive lattice, and  $\Rightarrow: L\times L\to \mathcal{I}(L)$ (where $\mathcal{I}(L)$ denotes the set of the ideals of $L$) satisfies the following conditions:  for every $a,b$ and $c \in L$,
\begin{enumerate}
\item $(a \Rightarrow b) \cap (a\Rightarrow c) = a\Rightarrow (b \wedge c)$,
\item $(a\Rightarrow b) \cap (b \Rightarrow c) = (a \vee b) \Rightarrow c$,
\item $(a\Rightarrow b) \cap (b\Rightarrow c) \subseteq a \Rightarrow c$,
\item $a \Rightarrow a = L$.
\end{enumerate}

For every generalized implication lattice $\mathbb{L}$, let $\mathbb{L}^*:=( L, g_{\Rightarrow})$  be its associated slanted algebra, where $g_{\Rightarrow} : L\times L \rightarrow L^\delta$ is defined by the assignment $(a,b)\mapsto \bigvee\lbrace c \in L \mid c \in a\Rightarrow b \rbrace$. It can be readily verified that $g_{\Rightarrow}$ is a binary o-slanted operator of order-type  $(\partial, 1)$ satisfying the inequalities $1\leq g_{\Rightarrow}(a,a)$ and $g_{\Rightarrow}(a,b) \wedge g_{\Rightarrow}(b,c) \leq g_{\Rightarrow}(a,c)$ for every $a,b,c \in L$ which are analytic Sahlqvist and analytic inductive respectively. Conversely, if $\mathbb{A} = (L, g)$ is an s-DLE s.t.~$\mathcal{F} = \varnothing$ and $\mathcal{G}: = \{g\}$ with $n_g = 2$ and $\epsilon_g = (\partial, 1)$  satisfying the properties verified by $g_{\Rightarrow}$, then $\mathbb{A}_*: =(L, \Rightarrow_g)$, where $\Rightarrow_g: L\times L\to \mathcal{I}(L)$ is defined by the assignment  $(a, b) \mapsto \lbrace c \in L \mid c \leq g(a,b) \rbrace$, is a generalized implication lattice. It is routine to show that $(\mathbb{L}^*)_* = \mathbb{L}$ for every generalized implication lattice $\mathbb{L}$, and $(\mathbb{A}_*)^* = \mathbb{A}$  for every s-DLE as above. 
\end{example}

As done in  \cite[Section 2.3]{GeJo04} and \cite[Section 5]{PaSoZh15},  the $\sigma$- and $\pi$-extensions of slanted $n$-ary operations of a given bounded lattice  $\bba$ are defined not as maps $(\bba^n)^\delta\to (\bbas)^\delta$ as in the standard definition (cf.~\cite[Definition 4.1]{GH01}), but as maps $(\bba^n)^\delta\to \bbas$. Towards the formal definition, recall (cf.~Section \ref{Perfect_Le}) that 
in order to extend operations of any arity which are monotone or antitone in each coordinate from a lattice $\bba$ to its canonical extension, treating the case
of {\em monotone} and {\em unary} operations suffices: %\marginnote{e' un po' un casino, credo sia meglio farlo con gli order-type invece di ridurci al caso unario, vedi il bordello che abbiamo fatto nella sez 5 di \cite{PaSoZh15}}
\begin{definition}
\label{def: sigma and pi extensions of slanted}
Let $A, B$ be bounded lattices. For every unary, c-slanted map $f : B \to A^\delta$, the $\sigma$-{\em extension} of $f$ is the map $f^\sigma: B^\delta \to A^\delta$ defined firstly by declaring, for every $k\in K(B^\delta)$,
$$f^\sigma(k):= \bigwedge\{ f(a)\mid a\in B\mbox{ and } k\leq a\},$$ and then, for every $u\in B^\delta$,
$$f^\sigma(u):= \bigvee\{ f^\sigma(k)\mid k\in K(B^\delta)\mbox{ and } k\leq u\}.$$
For every unary, o-slanted map $g : B \to A^\delta$, the $\pi$-{\em extension} of $g$ is the map $g^\pi: B^\delta \to A^\delta$ defined firstly by declaring, for every $o\in O(B^\delta)$,
$$g^\pi(o):= \bigvee\{ g(a)\mid a\in B\mbox{ and } a\leq o\},$$ and then, for every $u\in B^\delta$,
$$g^\pi(u):= \bigwedge\{ g^\pi(o)\mid o\in O(B\delta)\mbox{ and } u\leq o\}.$$
\end{definition}
It immediately follows by denseness and the definition above that, if $e: A\to A^\delta$ is the canonical embedding, then $e^\sigma = e^\pi = id_{A^\delta}$. Likewise, it can be readily verified that, for every (standard)  map $h: B\to A$  which is coordinatewise finitely join-preserving or meet-reversing (resp.~coordinatewise finitely meet-preserving or join-reversing), $(e\cdot h)^\sigma   = h^\sigma$ (resp.~$(e\cdot h)^\pi  = h^\pi$). Conversely, as discussed above, any c-slanted (resp.~o-slanted) map $h: B\to A^\delta$ the range of which is included in $K(A^\delta)\cap O(A^\delta) = e[A]$ gives rise to a normal map $h_{\star}: B\to A$ to which the standard definitions of $\sigma$- and $\pi$-extensions apply, and it can be readily verified that $h^\sigma  = (h_\star)^\sigma$ (resp.~$h^\pi  = (h_\star)^\pi$).

 Let us spell out and further simplify the definition above when $B: = A^\epsilon$ for any order-type $\epsilon$ on $n\geq 1$.
First, recall that taking the order-dual interchanges closed and open elements:
$K({(A^\delta)}^\partial) = O(A^\delta)$ and $O({(A^\delta)}^\partial) =\kbbas$;  similarly, $K({(A^n)}^\delta) =\kbbas^n$, and $O({(A^n)}^\delta) =\obbas^n$. Hence,  $K({(A^\delta)}^\epsilon) =\prod_i K(A^\delta)^{\epsilon(i)}$ and $O({(A^\delta)}^\epsilon) =\prod_i O(A^\delta)^{\epsilon(i)}$ for every LE $\bba$ and every  $\epsilon$, where
\begin{center}
\begin{tabular}{cc}
$K(A^\delta)^{\epsilon(i)}: =\begin{cases}
K(A^\delta) & \mbox{if } \epsilon(i) = 1\\
O(A^\delta) & \mbox{if } \epsilon(i) = \partial\\
\end{cases}
$ &
$O(A^\delta)^{\epsilon(i)}: =\begin{cases}
O(A^\delta) & \mbox{if } \epsilon(i) = 1\\
K(A^\delta) & \mbox{if } \epsilon(i) = \partial.\\
\end{cases}
$\\
\end{tabular}
\end{center}
Letting  $\leq^\epsilon$ denote the product order on $(A^\delta)^\epsilon$, we have for every $f\in \mathcal{F}$, $g\in \mathcal{G}$,  $\overline{k} \in K({(A^\delta)}^{\epsilon_f})$, $\overline{o} \in O({(A^\delta)}^{\epsilon_f})$, $\overline{u}\in (A^\delta)^{n_f}$, and $\overline{v}\in (A^\delta)^{n_g}$,
\begin{center}
\begin{tabular}{l l}
$f^\sigma (\overline{k}):= \bigwedge\{ f( \overline{a})\mid \overline{a}\in A^{\epsilon_f}\mbox{ and } \overline{k}\leq^{\epsilon_f} \overline{a}\}$ & $f^\sigma (\overline{u}):= \bigvee\{ f^\sigma( \overline{k})\mid \overline{k}\in K({(A^\delta)}^{\epsilon_f})\mbox{ and } \overline{k}\leq^{\epsilon_f} \overline{u}\}$ \\
$g^\pi (\overline{o}):= \bigvee\{ g( \overline{a})\mid \overline{a}\in A^{\epsilon_g}\mbox{ and } \overline{a}\leq^{\epsilon_g} \overline{o}\}$ & $g^\pi (\overline{v}):= \bigwedge\{ g^\pi( \overline{o})\mid \overline{o}\in O({(A^\delta)}^{\epsilon_g})\mbox{ and } \overline{v}\leq^{\epsilon_g} \overline{o}\}$. \\
\end{tabular}
%
%\begin{tabular}{cc}
%$\Diamond^\sigma u:= \bigvee\{ \Diamond k\mid k\in \kbbas\mbox{ and } k\leq u\}$ & $\Box^\pi u:= \bigwedge\{ \Box o\mid o\in \obbas\mbox{ and } u\leq o\}$\\
%${\lhd}^\sigma u:= \bigvee\{ {\lhd}o\mid o\in \obbas\mbox{ and } u\leq o\}$ & ${\rhd}^\pi u:= \bigwedge\{ {\rhd}k\mid k\in \kbbas\mbox{ and } k\leq u\}$\\
%$u\circ^\sigma v:= \bigvee\{ k\circ k'\mid k, k'\in \kbbas, k\leq u \mbox{ and } k'\leq v\}$ & $u\star^\pi v:= \bigwedge\{o\star o'\mid o, o'\in \obbas, u\leq o \mbox{ and } v\leq o'\}$.\\
%\end{tabular}
\end{center}

%\marginnote{qui bisogna integrare la presentazione della Sez 5 di \cite{PasoZh15}}

\begin{lemma}
\label{lemma:basic properties of extensions of slanted}
For every lattice $\mathbb{A}$, any  c-slanted operation $f$ on $\bba$ of arity $n_f$ and order-type $\epsilon_f$,  and any  o-slanted operation $g$ on $\bba$ of arity $n_g$ and order-type $\epsilon_g$,
\begin{enumerate} 
\item $f^\sigma$ is $\epsilon_f$-monotone and $g^\pi$ is $\varepsilon_g$-monotone;
\item $f^\sigma$ is completely join-preserving in all coordinates $i$ such that $\epsilon_f(i) = 1$ and completely meet-reversing in all coordinates $i$ such that $\epsilon_f(i) = \partial$;
 \item $g^\pi$ is completely meet-preserving in all coordinates $i$ such that $\epsilon_g(i) = 1$ and completely join-reversing in all coordinates $i$ such that $\epsilon_g(i) = \partial$.
 \end{enumerate}

%It is easy to see that the $\sigma$- and $\pi$-extensions of $\varepsilon$-monotone maps are $\varepsilon$-monotone. More remarkably,  the $\sigma$-extension of a map which sends (finite) joins or meets in the domain to (finite) joins in the codomain sends {\em arbitrary} joins or meets in the domain to {\em arbitrary} joins in the codomain. Dually, the $\pi$-extension of a map which sends (finite) joins or meets in the domain to (finite) meets in the codomain sends {\em arbitrary} joinsor meets in the domain to {\em arbitrary} meets in the codomain.
\end{lemma}
\begin{proof}
%\marginnote{Mi sembra che funziona, no ? Se e il caso, bisognerebbe vedere come definire $g$ nel caso non unare per conservare $f(a)$ chiuso\\A: si' fantastico!! Ho fatto una piccola aggiunta perche' usiamo che $f$ e' finitamente distributiva, sei d'accordo?}
As to item 1, let $u, v\in (\bbas)^{\epsilon_f}$. If $u\leq v$, then by denseness, for every $k\in K((\bbas)^{\epsilon_f}))$, if $k\leq u$ then $k\leq v$. Hence $f^\sigma (u): = \bigvee\lbrace f^\sigma(k)\mid k\in K((\bbas)^{\epsilon_f}) \text{ and } k\leq u\rbrace\leq \bigvee\lbrace f^\sigma(k)\mid k\in K((\bbas)^{\epsilon_f}) \text{ and } k\leq v\rbrace: = f^\sigma (v)$. The proof of the $\epsilon_g$-monotonicity of $g^\pi$ is dual. 

The arguments for  proving the remaining items in the standard setting  (cf.~\cite[Lemma 4.6]{GH01}) can be straightforwardly generalized to the present setting. However, we are going to adopt a simpler method, which is constructive and for which we do not need to appeal to the restricted distributive law. Namely, since $\bbas$ is a complete lattice, it is enough to show that the right residuals  (resp.~Galois residuals) of $f^\sigma$ exist in each coordinate. For the sake of keeping the notation simple, let us show that if $f$ is binary and of order-type $\epsilon_f = \epsilon=(1,\partial)$, the right residual of $f$ in the first coordinate (which needs to be of order-type $(1, 1)$)  exists. Let $g_1 : \bbas\times \bbas \rightarrow \bbas$ be defined as follows: $g_1(o,o') := \bigvee \lbrace a \in A \mid f^{\sigma}(a,o')\leq o \rbrace$ for all $o,o' \in O(\bbas)$ and $g_1(v_1,v_2): = \bigwedge \lbrace g(o_1,o_2)\mid o_i\in \obbas \text{ and }v_i\leq o_i\rbrace$ for all $v_1,v_2 \in \bbas$.\footnote{If $f: \bba^{\epsilon_f}\to \bbas$, then for every $1\leq i\leq n_f$ such that $\epsilon_f(i) = 1$ we let $g_i: (\bbas)^{\epsilon_{g_i}}\to \bbas$ be defined as follows: $g_i(\overline{o}) := \bigvee \lbrace a \in A \mid f^\sigma(\overline{o}[a/o_i]) \leq o_i \rbrace$ for every $\overline{o} \in O(\bbas)^{\epsilon_{g_i}}$ and $g_i(\overline{v}): = \bigwedge \lbrace g_i(\overline{o})\mid o\in {\obbas}^{\epsilon_{g_i}} \text{ and }\overline{v}\leq^{\epsilon_{g_i}} \overline{o}\rbrace$ for any $v\in \bbas$, where $\epsilon_{g_i}(i) = 1$  and $\epsilon_{g_i}(j) = \epsilon_f^\partial(j)$ if $j \neq i$. For every $1\leq i\leq n_f$ such that $\epsilon_f(i) = \partial$ we let $g_i: (\bbas)^{\epsilon_{g_i}}\to \bbas$ be defined as follows: $g_i(\overline{k}) := \bigwedge \lbrace a \in A \mid f^\sigma(\overline{k}[a/o_i]) \leq o_i \rbrace$ for every $\overline{k} \in K(\bbas)^{\epsilon_{g_i}}$ and $g_i(\overline{v}): = \bigvee \lbrace g_i(\overline{k})\mid k\in {\kbbas}^{\epsilon_{g_i}} \text{ and }\overline{k}\leq^{\epsilon_{g_i}} \overline{v} \rbrace$ for any $v\in \bbas$, where $\epsilon_{g_i}(j) = \epsilon_f(j)$ for every $1\leq j\leq n_f$.} Let us show that, for every $k\in K(\bbas) $ and all $o, o'\in \obbas$,
\begin{equation}
\label{eq: adjoint fsigma}
f^\sigma(k,o')\leq o\quad \text{ iff }\quad k\leq g_1(o,o').
\end{equation}
From left to right, if $\bigwedge \lbrace f(a,b) \mid a,b \in A \text{ and }  k \leq a \text{ and }  b \leq o'  \rbrace = : f^\sigma(k,o') \leq o$, then by compactness (recall that $f(a,b)\in \kbbas$) this implies that $f(a_1,b_1) \wedge \cdots \wedge f(a_n,b_n) \leq o$ for some $a_1,\ldots,a_n\in A$ and $b_1,\ldots,b_n\in A$ such that $k \leq a_i$ and  $b_i \leq o'$ for every $1\leq i\leq n$. Since $f$ is $\epsilon$-monotone, letting $b := b_1 \vee \cdots \vee b_n$ and $a := a_1 \wedge \cdots \wedge a_n$, this implies that $k \leq a$, $b \leq o'$ and $f(a,b) \leq f(a_1,b_1) \wedge \cdots \wedge f(a_n,b_n)\leq o$. Hence,  $f^\sigma(a,o'): = \bigwedge \lbrace f(a, b)\mid b \in A \text{ and }   b \leq o' \rbrace \leq f(a, b) \leq o$, and hence $k \leq a \leq \bigvee \lbrace a \in A \mid f^{\sigma}(a,o')\leq o \rbrace= : g_1(o,o')$, as required.

For the converse direction, if $ k \leq g_1(o,o') := \bigvee \lbrace a \in A \mid f^{\sigma}(a,o')\leq o \rbrace,$ then by compactness, $k_1\leq a_1\vee\cdots \vee a_n$ for some $a_1,\ldots, a_n\in A$  such that $\bigwedge \lbrace f(a_i,b) \mid b \in A \text{ and } b \leq o' \rbrace = :f^\sigma(a_i,o') \leq o$ for every $1\leq i\leq n$. Hence, by compactness, for every $1\leq i \leq n $, there exist some $b^1_i,\ldots,b^{n_i}_i \in A$ such that $b^j_i \leq o'$ for every $1\leq j\leq n_i$ and 
\[ f(a_i,b^1_i) \wedge \cdots \wedge f(a_i,b^{n_i}_i) \leq o. \] 
For each $1\leq i\leq n$, let $b_i := b^1_i \vee \cdots \vee b^{n_i}_i$. Hence, $b_i \leq o'$ and, by the antitonicity of $f$ in its second coordinate, $f(a_i,b_i) \leq f(a_i,b^1_i) \wedge \cdots \wedge f(a_i,b^{n_i}_i) \leq o$ for every $1\leq i\leq n$. Hence, letting $b := b_1 \vee \cdots \vee b_n$, and $a := a_1 \vee \cdots \vee a_n$, we have $b \leq o'$ and $k \leq a$, and moreover, %since $f$ is finitely join-preserving in its first coordinate and antitone in its second coordinate,
\begin{center}
\begin{tabular}{r l ll}
&$f^\sigma(k,o')$ \\ $:=$ & $\bigwedge \lbrace f(a,b) \mid a,b \in A \text{ and }  k \leq a \text{ and }  b \leq o'  \rbrace$\\
 $ \leq$ &  $ f(a,b)$\\
 $ =$ & $ f(a_1 \vee \cdots \vee a_n, b)$ & {$a := a_1 \vee \cdots \vee a_n$}\\
 $ = $& $f(a_1,b) \vee \cdots \vee f(a_n,b)$ &{$f$ finitely join preserving in its first coord.}\\
 $ \leq$ &  $ f(a_1,b_1) \vee \cdots \vee f(a_n,b_n)$ & {$b := b_1 \vee \cdots \vee b_n$ and $\epsilon_f(2) = \partial$} \\
 $ \leq$ &  $ o$ & $f(a_i,b_i) \leq o$ for every $1\leq i\leq n$\\ 
\end{tabular}
\end{center}
as required. Let us show that, for all $u, u', v\in \bbas$,
\[f^\sigma(u,u')\leq v \quad \text{ iff }\quad u\leq g_1(v,u').\]
Let $u, u', v\in \bbas$. From left to right,  if $f^\sigma(u,u')\leq v$, to show that $u_1\leq g(v,u'): = \bigwedge \lbrace g(o,o')\mid o,o'\in \obbas \text{ and } v\leq o \text{ and } u'\leq o' \rbrace$ it is enough to show that $k \leq g(o,o')$ for all $k \in K(\bbas)$ such that $k \leq u$ and for all $o,o' \in O(\bbas)$ such that $v\leq o$ and $u' \leq o'$. Since $f^\sigma$ is $\epsilon$-monotone, for any such $k, o$ and $o'$ we have $f^\sigma(k,o') \leq f^\sigma(u,u') \leq v \leq o$. By \eqref{eq: adjoint fsigma}, this implies that $k \leq g_1(o,o')$ as required.  From right to left, if $u\leq g_1(v,u') :  = \bigwedge \lbrace g(o,u')\mid o,o'\in \obbas \text{ and } v\leq o \text{ and } u'\leq o' \rbrace$, to show that $f^\sigma(u,u')\leq v$, we need to show that $f^\sigma(k,o')\leq o$ for every $k\in \kbbas$ s.t.~$k\leq u$  and every $o,o'\in \obbas$ s.t.~$v\leq o$ and $u' \leq o'$. By assumption, $k \leq u \leq g(v,u') \leq g(o,o')$ which, by \eqref{eq: adjoint fsigma}, implies $f^\sigma(k,o') \leq o$, as required.
\end{proof}

As in the standard case (cf.~discussion before Definition \ref{def:canext LE standard}), we define the canonical extension of a slanted $\mathcal{L}_\mathrm{LE}$-algebra so as to meet the desideratum that it be a perfect (resp.~complete, in the constructive setting) {\em standard}  $\mathcal{L}_\mathrm{LE}$-algebra. In the light of the lemma above, we are again justified in choosing the $\sigma$-extensions of connectives in $\mathcal{F}$ and the $\pi$-extensions of connectives in $\mathcal{G}$, which motivates the following
\begin{definition}
\label{def:canext slanted}
%For every LE $\bba = (\bbb, \mathcal{F}, \mathcal{G})$, the {\em canonical extension} of $\bba$
The canonical extension of a slanted
$\mathcal{L}_\mathrm{LE}$-algebra $\bbA = (A, \mathcal{F}^\bbA, \mathcal{G}^\bbA)$ is the  $\mathcal{L}_\mathrm{LE}$-algebra
$\bbA^\delta: = (A^\delta, \mathcal{F}^{\bbA^\delta}, \mathcal{G}^{\bbA^\delta})$ such that, for all $f\in \mathcal{F}$ and $g\in \mathcal{G}$, the operations $f^{\bbA^\delta}$ and $g^{\bbA^\delta}$ are defined as the 
$\sigma$-extension of $f^{\bbA}$ and as the $\pi$-extension of $g^{\bbA}$ respectively, as in Definition \ref{def: sigma and pi extensions of slanted}.
\end{definition}
It immediately follows from the definition above and Lemma \ref{lemma:basic properties of extensions of slanted} that the canonical extension of a slanted LE $\bba$ is  a perfect LE  (cf.~Definition \ref{def:perfect LE}) (resp.~complete LE, in the constructive setting) in the standard sense.

Also, from the discussions after Definitions \ref{def:slanted LE} and \ref{def: sigma and pi extensions of slanted}, it readily follows that $(\bba^\star)^\delta = \bba^\delta$ for every standard LE $\bba$, and that  $(\mathbb{S}_\star)^\delta = \mathbb{S}^\delta$ for every slanted LE $\mathbb{S}$ based on a bounded lattice $A$ and such that all  its operations target $K(A^\delta)\cap O(A^\delta) = e[A]$.

\subsection{Slanted LE-algebras as models of LE-inequalities}\label{sec_models}

Fix  an arbitrary LE-signature $(\mathcal{F}, \mathcal{G})$. From the discussion of the previous section, it is clear that, for any slanted $\mathcal{L}_\mathrm{LE}$-algebra $\bba$, any {\em assignment} into $\bba$, i.e.~any map $v:\mathsf{PROP}\to \bba$, uniquely extends to an  $\mathcal{L}_\mathrm{LE}$-homomorphism $v: \mathbf{Fm}\to \bbas$ (abusing notation, the same symbol for the given assignment also denotes its homomorphic extension).  Hence,
\begin{definition}
\label{def:slanted satisfaction and validity}
 An $\mathcal{L}_\mathrm{LE}$-inequality $\phi\leq\psi$ is {\em satisfied} in a slanted $\mathcal{L}_\mathrm{LE}$-algebra $\bba$ under the assignment $v$ (notation: $(\bba, v)\models \phi\leq\psi$) if $(\bbas, e\cdot v)\models \phi\leq\psi$ in the usual sense, where $e\cdot v$ is the assignment on $\bbas$ obtained by composing the canonical embedding $e: \bba\to \bbas$ to the assignment $v:\mathsf{PROP}\to \bba$. 
 
 Moreover, $\phi\leq\psi$ is {\em valid} in $\bba$ (notation: $\bba\models \phi\leq\psi$) if $(\bbas, e\cdot v)\models \phi\leq\psi$ for every assignment $v$ into $\bba$  (notation: $\bbas\models_{\bba} \phi\leq\psi$). We will often refer to assignments  into $\bba$ as {\em admissible} assignments.
\end{definition}
From the definition above, and the discussion after Definition \ref{def:canext slanted}, it immediately follows that any LE-inequality $\phi\leq\psi$ is  valid in $\bba$ iff $\phi\leq\psi$ is  valid in $\bba^\star$ for every standard LE $\bba$, and that $\phi\leq\psi$ is  valid in $\mathbb{S}$ iff  $\phi\leq\psi$ is  valid in $\mathbb{S}_\star$  for every slanted LE $\mathbb{S}$ based on a bounded lattice $A$ and such that all  its operations target $K(A^\delta)\cap O(A^\delta) = e[A]$. This shows that the notion of validity on slanted algebras generalizes standard validity in the appropriate way.

Notice that, whether constructive or non-constructive, the canonical extension of any slanted $\mathcal{L}_\mathrm{LE}$-algebra $\bba$ is an $\mathcal{L}_\mathrm{LE}^+$-algebra in the standard sense. Hence, given the definition above, any slanted $\mathcal{L}_\mathrm{LE}$-algebra is also a {\em slanted} $\mathcal{L}_\mathrm{LE}^+$-algebra, in the sense that the definition above makes the machinery of $\bbas$ available for the interpretation of the language $\mathcal{L}_\mathrm{LE}^+$ on $\bba$ in the sense specified in the definition above. 

Recall that by definition, $f^{\bbas} = (f^{\bba})^\sigma$ for each $f\in \mathcal{F}$, and $g^{\bbas} = (g^{\bba})^\pi$ for each $g\in \mathcal{G}$. 
We are now in a position to define the notion of {\em slanted canonicity} (abbreviated as {\em s-canonicity}) for $\mathcal{L}_\mathrm{LE}$-sequents/inequalities:

\begin{definition}[Slanted canonicity of $\mathcal{L}_\mathrm{LE}$-inequalities]
\label{def:slanted canonicity}
 An $\mathcal{L}_\mathrm{LE}$-inequality $\phi\leq\psi$ is {\em s-canonical} if for every  slanted $\mathcal{L}_\mathrm{LE}$-algebra $\bba$,
 \[ \bbas\models_{\bba} \phi\leq\psi \quad \mbox{ implies }\quad \bbas\models \phi\leq\psi.\]
\end{definition}
From the definition above, and the discussion after Definition \ref{def:slanted satisfaction and validity}, it immediately follows that standard canonicity translates into a {\em relativized} notion of slanted canonicity: namely, slanted canonicity relative to the slanted algebras in the range of the constructor $(-)^\star$. Conversely, slanted canonicity {\em relativized} to the domain of definition of the constructor $(-)_\star$ corresponds to the notion of standard canonicity.

%	For every LE $\bba$, the symbol $\vdash$ is interpreted as the lattice order $\leq$. A sequent $\phi\vdash\psi$ is valid in $\bba$ if $h(\phi)\leq h(\psi)$ for every homomorphism $h$ from the $\mathcal{L}_\mathrm{LE}$-algebra of formulas over $\mathsf{PROP}$ to $\bba$. The notation $\mathbb{LE}\models\phi\vdash\psi$ indicates that $\phi\vdash\psi$ is valid in every LE. Then, by means of a routine Lindenbaum-Tarski construction, it can be shown that the minimal LE-logic $\mathbf{L}_\mathrm{LE}$ is sound and complete with respect to its correspondent class of algebras $\mathbb{LE}$, i.e.\ that any sequent $\phi\vdash\psi$ is provable in $\mathbf{L}_\mathrm{LE}$ iff $\mathbb{LE}\models\phi\vdash\psi$. 

\section{Slanted canonicity of analytic inductive LE-inequalities} %\marginnote{bisogna uniformare: o usiamo sequenti oppure inequalities; forse e' meglio sequenti?}
\label{section:canonicity}

  %In  \cite{CoPa-dist}, the analogous proof for the distributive setting was given in terms of descriptive general frames. Here, however, we will proceed purely algebraically.
 %We write $\bbas \models_{\mathbb{A}} \phi \leq \psi$ to indicate that $\bbas, v \models \phi \leq \psi$ for all \emph{admissible} assignments $v$, as defined in Subsection \ref{Subsec:Expanded:Land}, page \pageref{admissible:assignment}. Recall that pivotal executions of ALBA are defined on page \pageref{pivotal:approx:rule:application} in Section \ref{Sec:ReductionElimination}.
 This section is aimed at showing that every analytic inductive formula is s-canonical (in the sense of Definition \ref{def:slanted canonicity}). We first give the statement of the canonicity theorem and its proof, and subsequently prove the proposition needed in the previously stated proof and its requisite preliminaries.

\begin{theorem}\label{Thm:ALBA:Canonicity}
For any language $\mathcal{L}_\mathrm{LE}$, all analytic inductive $\mathcal{L}_\mathrm{LE}$-inequalities  are s-canonical.
%All $\mathcal{L}_\mathrm{LE}$-inequalities on which ALBA succeeds pivotally are canonical.
\end{theorem}
\begin{proof}
Let $\phi \leq \psi$ be an analytic inductive $\mathcal{L}_\mathrm{LE}$-inequality,
fix a slanted $\mathcal{L}_\mathrm{LE}$-algebra  $\bba$, and let $\bbas$ be its canonical extension. As discussed in Section \ref{Spec:Alg:Section}, ALBA succeeds in reducing $\phi \leq \psi$ to a set $\mathsf{ALBA}(\phi \leq \psi)$ of pure quasi-inequalities in the expanded language $\mathcal{L}^+_\mathrm{LE}$. 
%Let $\phi \leq \psi$ be an $\mathcal{L}_\mathrm{LE}$-inequality on which ALBA succeeds pivotally. 
The required canonicity proof is summarized in the following U-shaped diagram:\\
\begin{center}
\begin{tabular}{l c l}
$\bba \ \ \models \phi \leq \psi$ & &$\bbas \models \phi \leq \psi$\\
$\ \ \ \ \ \ \ \ \ \ \ \Updownarrow$ \\
$\bbas \models_{\bba} \phi \leq \psi$ & & \ \ \ \ \ \ \ \ \ \ \  $\Updownarrow $\\
$\ \ \ \ \ \ \ \ \ \ \ \Updownarrow$\\
$\bbas \models_{\bba} \mathsf{ALBA}(\phi \leq \psi)$
&\ \ \ $\Leftrightarrow$ \ \ \ &$\bbas \models \mathsf{ALBA}(\phi \leq \psi)$
\end{tabular}
\end{center}
The upper bi-implication on the left is due to the definition of validity on slanted LEs (cf.~Definition \ref{def:slanted satisfaction and validity}).
The lower bi-implication on the left is given by Proposition \ref{Propo:Crtns:CanExtns} below. The horizontal bi-implication follows from the facts that, by assumption, $\mathsf{ALBA}(\phi \leq \psi)$ is pure, and that, when restricted to pure formulas, the ranges of admissible and arbitrary assignments coincide. The bi-implication on the right is due to \cite[Theorem 6.1]{CoPa-nondist} (which can be applied since the canonical extension of a slanted LE is a standard LE).
\end{proof}
Towards the proof of  Proposition \ref{Propo:Crtns:CanExtns}, the following definitions and lemmas will be useful:

\begin{definition}\label{Syn:Opn:Clsd:Definition}
The sets $\mathsf{SC}$ and $\mathsf{SO}$ of \emph{syntactically closed}  and \emph{syntactically open}  $\mathcal{L}_\mathrm{LE}^{+}$-terms are defined simultaneously as follows: for every $f^\ast \in \mathcal{F}^\ast$, $f \in \mathcal{F}$, $g^\ast \in \mathcal{G}^\ast$, and $g \in \mathcal{G}$,
\begin{align*}
\mathsf{SC} \ni \varphi ::&= p \mid \nomj \mid \top \mid \bot \mid  \varphi \vee \varphi \mid \phi \wedge \phi \mid f^\ast(\overline{\varphi}, \overline{\psi}) \mid g(\overline{\varphi}, \overline{\psi}) \\
\mathsf{SO} \ni \psi ::&= p \mid \cnomm \mid \top \mid \bot \mid  \psi \vee \psi \mid \psi \wedge \psi \mid g^\ast(\overline{\psi}, \overline{\varphi}) \mid f(\overline{\psi}, \overline{\varphi}).
\end{align*}
Recall that, when writing $h(\overline{\chi}, \overline{\xi})$, we let $\overline{\chi}$ represent all the coordinates of $h$ such that $\epsilon_{h}(i) = 1$ and $\overline{\xi}$ represent all the coordinates of $h$ such that $\epsilon_{h}(i) = \partial$.
\end{definition}
The previous definition identifies the syntactic shape of the terms, the (formal) topological\footnote{This denomination stems from the original definition of canonical extensions as algebraic structures encoding information about Stone duals of Boolean algebras  \cite{jonsson1951boolean}.} properties of which guarantee the soundness of the Ackermann rules under admissible assignments in the setting of standard (i.e.~non slanted) LEs. The  following definition identifies a more restricted syntactic shape of LE-terms  which aims at guaranteeing the soundness of the Ackermann rules under admissible assignments in the setting of  {\em slanted} LEs;  this  restriction consists in imposing the same constraints both to  the connectives of the original language and to  those of the expanded language.
\begin{definition}\label{strictlySyn:Opn:Clsd:Definition}
The sets $\mathsf{SSC}$ and $\mathsf{SSO}$ of \emph{strictly syntactically closed} (ssc) and \emph{strictly syntactically open} (sso)  $\mathcal{L}_\mathrm{LE}^{+}$-terms are defined simultaneously as follows: for every $f^\ast \in \mathcal{F}^\ast$, and $g^\ast \in \mathcal{G}^\ast$,
\begin{align*}
\mathsf{SSC} \ni \varphi ::&= p \mid \nomj \mid \top \mid \bot \mid  \varphi \vee \varphi \mid \phi \wedge \phi \mid f^\ast(\overline{\varphi}, \overline{\psi}),  \\
\mathsf{SSO} \ni \psi ::&= p \mid \cnomm \mid \top \mid \bot \mid  \psi \vee \psi \mid \psi \wedge \psi \mid g^\ast(\overline{\psi}, \overline{\varphi}).
\end{align*}
\end{definition}

From the definition above, it immediately follows that 
\begin{lemma}
\label{lemma:composition of ssc-sso}
For all ssc formulas $\phi(\overline{!x}, \overline{!y})$ and all sso formulas $\psi(\overline{!x}, \overline{!y})$ which are positive in  any $x$ in $\overline{!x}$ and negative in any $y$ in $\overline{!y}$, and all tuples $\overline{\phi'}$ and $\overline{\psi'}$ of ssc formulas and sso formulas respectively,
\begin{enumerate}
\item $\phi[\overline{\phi'}/\overline{!x}, \overline{\psi'}/\overline{!y}]$ is ssc;
\item $\psi[\overline{\psi'}/\overline{!x}, \overline{\phi'}/\overline{!y}]$ is sso.
\end{enumerate}
\end{lemma}
\begin{lemma}
\label{lemma:adjoints and ssc-sso}
If $\alpha(!x)$ is a definite positive PIA $\mathcal{L}_\mathrm{LE}$-formula and $\beta(!x)$ is a definite negative PIA $\mathcal{L}_\mathrm{LE}$-formula, then \begin{enumerate}
\item $\alpha$ is sso and $\beta$ is ssc.
\item If $+x\prec +\alpha$ and $+x\prec +\beta$, then $\mathsf{LA}(\alpha)[\nomj/!u]$ is ssc and $\mathsf{RA}(\beta)[\cnomm/!u]$ is sso. 
\item If $-x\prec +\alpha$ and $-x\prec +\beta$, then $\mathsf{LA}(\alpha)[\nomj/!u]$ is sso and $\mathsf{RA}(\beta)[\cnomm/!u]$ is ssc. 
%\marginnote{missing cases: if $\alpha(!x): = {\rhd}x$? then $\mathsf{LA}(\alpha)(u) = {\blacktriangleright} u$, which is also sso}
\end{enumerate}
\end{lemma}
\begin{proof} %\marginnote{this proof needs to be revised as in the proof of lemma 1.19 because the def of la and ra need to have more parameters. Ripensandoci, forse qui i parametri non servono quindi la dimostraz puo' rimanere cosi' com'e'. cosa ne dici?}
1. Straightforward by simultaneous induction on $\alpha$ and $\beta$.

2. and 3. We proceed by simultaneous induction on $\alpha$ and $\beta$.

If $\alpha = \beta = x$, then the assumptions of item 2 are satisfied; then $\mathsf{LA}(\alpha)[\nomj/!u] = \nomj/u$ is clearly ssc and $\mathsf{RA}(\beta)[\cnomm/!u] = \cnomm/u$ is clearly sso. 

As to the inductive step, if $\alpha = g(\overline{\phi}, \overline{\psi})$, with each $\phi$ in $\overline{\phi}$ positive PIA (hence, by item 1, sso) and  each $\psi$ in $\overline{\psi}$ negative PIA (hence, by item 1, ssc), and the only occurrence of $x$ is in  ${\phi}_h$, then $\phi_h$ is positive PIA, and moreover, $g_h^\flat\in \mathcal{F}^\ast$ is positive in its $h$th coordinate and has the opposite polarity of $\epsilon_g$ in all the other coordinates. Hence, $g_h^\flat(\overline{\phi_{-h}},\nomj/!u,\overline{\psi})$ is ssc. Two cases can occur: (a) if $+x\prec +\alpha$, then $+x\prec +\phi_h$, hence by induction hypothesis, $\mathsf{LA}(\phi_h)[\nomi/!u']$ is ssc, and moreover, $+u'\prec \mathsf{LA}(\phi_h)(u')$ (cf.~Lemma \ref{lemma:polarities of la-ra}). Hence,
\[ \mathsf{LA}(\alpha)[\nomj/!u] = \mathsf{LA}(\phi_h)[g_h^\flat(\overline{\phi_{-h}},\nomj/!u,\overline{\psi})/!u'] \]
is ssc (cf.~Lemma \ref{lemma:composition of ssc-sso}).
(b) if $-x\prec +\alpha$, then $-x\prec +\phi_h$, hence by induction hypothesis, $\mathsf{LA}(\phi_h)[\nomi/!u']$ is sso, and moreover, $-u'\prec \mathsf{LA}(\phi_h)(u')$ (cf.~Lemma \ref{lemma:polarities of la-ra}). Hence,
\[ \mathsf{LA}(\alpha)[\nomj/!u] = \mathsf{LA}(\phi_h)[g_h^\flat(\overline{\phi_{-h}},\nomj/!u,\overline{\psi})/!u'] \]
is sso (cf.~Lemma \ref{lemma:composition of ssc-sso}).
The remaining cases are $\alpha=g(\overline{\phi}, \overline{\psi})$ such that the only occurrence of $x$ is in $\psi_h$, $\beta = f(\overline{\phi}, \overline{\psi})$ with $x$ occurring in $\phi_h$ or $\psi_h$, and are shown in a similar way.
\end{proof}

In the following two lemmas, $\alpha$, $\beta_1,\ldots,\beta_n$ and $\gamma_1,\ldots,\gamma_n$ are $\mathcal{L}_\mathrm{LE}^{+}$-terms. We work under the assumption that the values of all parameters occurring in them (propositional variables, nominals and conominals) are given by some fixed admissible assignment. %\marginnote{Modificato statement del lemma}
Recall that every slanted $\mathcal{L}_\mathrm{LE}$-algebra is also an slanted $\mathcal{L}_\mathrm{LE}^+$-algebra (cf.~discussion after Definition \ref{def:slanted satisfaction and validity}).
\begin{lemma}[Righthanded Ackermann lemma for admissible assignments]\label{Ackermann:Dscrptv:Right:Lemma}
Let $\alpha$ be ssc, $p \not \in \mathsf{PROP}(\alpha)$, let $\beta_{1}(p), \ldots, \beta_{n}(p)$ be ssc and positive in $p$, and let $\gamma_{1}(p), \ldots, \gamma_{n}(p)$ be sso and negative in $p$.  Then, for every slanted $\mathcal{L}_\mathrm{LE}$-algebra $\bba$ and every admissible assignment $v$ into $\bba$, 
\[
(\bba, v) \models {\beta_i}(\alpha/p) \leq {\gamma_i}(\alpha/p) \textrm{ for all } 1 \leq i \leq n
\]
iff there exists some $p$-variant $v'$ of $v$  into $\bba$ such that
\[
(\bba, v') \models \alpha \leq p \textrm{ and } (\bba, v') \models {\beta_i}(p) \leq {\gamma_i}(p) \textrm{ for all } 1 \leq i \leq n.
\]
\end{lemma}

\begin{lemma}[Lefthanded Ackermann lemma for admissible assignments]\label{Ackermann:Dscrptv:Left:Lemma}
Let $\alpha$ be sso, $p \not \in \mathsf{PROP}(\alpha)$, let $\beta_{1}(p), \ldots, \beta_{n}(p)$ be ssc and negative in $p$, and let $\gamma_{1}(p), \ldots, \gamma_{n}(p)$ be sso and positive in $p$.  Then, for every slanted $\mathcal{L}_\mathrm{LE}$-algebra $\bba$ and every admissible assignment $v$ into $\bba$, 
\[
(\bba, v) \models {\beta_i}(\alpha/p) \leq {\gamma_i}(\alpha/p) \textrm{ for all } 1 \leq i \leq n
\]
iff there exists some admissible $p$-variant $v'$ of $v$  into $\bba$ such that
\[
(\bba, v') \models p \leq \alpha \textrm{ and }  (\bba, v') \models {\beta_i}(p) \leq {\gamma_i}(p) \textrm{ for all } 1 \leq i \leq n.
\]
\end{lemma}
The two lemmas above are proved in Section \ref{Sec: topological Ackermann}.

\begin{lemma}\label{Syn:Shape:Lemma} %\marginnote{the shape of mv(p) and mv(q) needs to be edited, same as proof of lemma 3.4}

Executing ALBA on an analytic inductive $\mathcal{L}_\mathrm{LE}$-inequality $(\phi \leq \psi)[\overline{\alpha}/!\overline{x}, \overline{\beta}/!\overline{y},\overline{\gamma}/!\overline{z}, \overline{\delta}/!\overline{w}]$, as indicated in Section \ref{Spec:Alg:Section}, we obtain  quasi-inequalities each of which is such that each inequality in its antecedent, which as discussed at the end of Section \ref{Sec:ReductionElimination} is of either of the following forms:
\begin{equation}\label{eq:Lemma:SSO}
\nomi\leq \gamma\left(\overline{\mathsf{mv}(p)}/\overline{p}, \overline{\mathsf{mv}(q)}/\overline{q}\right)\quad \quad \delta\left(\overline{\mathsf{mv}(p)}/\overline{p}, \overline{\mathsf{mv}(q)}/\overline{q}\right)\leq\cnomn,
\end{equation}
is such that its left-hand side is ssc and its right-hand side is sso. 
%then the left-hand side of  $\mathsf{Ineq}$ as well as the left-hand side of every non-pure inequality  in $S$ is syntactically closed, while the corresponding right-hand sides  are syntactically open.
\end{lemma}
\begin{proof}
Clearly, $\nomi$ is ssc and $\cnomn$ is sso. Let us show that the formula $\gamma\left(\overline{\mathsf{mv}(p)}/\overline{p}, \overline{\mathsf{mv}(q)}/\overline{q}\right)$ is sso while $\delta\left(\overline{\mathsf{mv}(p)}/\overline{p}, \overline{\mathsf{mv}(q)}/\overline{q}\right)$ is ssc. Recall from Notation \ref{notation: analytic inductive} that $\gamma(\overline{p}, \overline{q})$ (resp.~$\delta(\overline{p}, \overline{q})$) is a positive (resp.~negative) PIA term, and both  $\gamma$ and $\delta$ are  $\epsilon^\partial$-uniform as subterms of the original analytic inductive inequality.   Recall that $\epsilon (p) = 1$ for every variable $p$ in $\overline{p}$ and $\epsilon (q) = \partial$ for each $q$ in $\overline{q}$. Hence, $-p\prec +\gamma$ and $+q\prec +\gamma$, and $-p\prec -\delta$ and $+q\prec -\delta$  for each $p$ in $\overline{p}$ and  each $q$ in $\overline{q}$. Lemma \ref{lemma:adjoints and ssc-sso}.1 implies that $\gamma(\overline{p}, \overline{q})$ is sso and $\delta(\overline{p}, \overline{q})$) is ssc. Hence, the proof is complete if we show that  $\mathsf{mv}(p)$ is ssc for every variable $p$ such that $\epsilon (p) = 1$ and  $\mathsf{mv}(q)$ is sso for every variable $q$ such that $\epsilon (q) = \partial$. Recall (cf.~Subsection \ref{Spec:Alg:Section}) that for every $p$ in $\overline{p}$, the formula $\mathsf{mv}(p)$ is either of the form $\mathsf{LA}(\alpha_p)[\nomj_k/u,\overline{\mathsf{mv}(p)}/\overline{p},\overline{\mathsf{mv}(q)}/\overline{q}]$ for some definite positive PIA formula $\alpha_p$ (and hence $+p\prec +\alpha_p$), or $\mathsf{mv}(p)$ is of the form $\mathsf{RA}(\beta_p)[\cnomm_h/u,\overline{\mathsf{mv}(p)}/\overline{p},\overline{\mathsf{mv}(q)}/\overline{q}]$ for some definite negative PIA formula $\beta_p$ (and hence $+p\prec -\beta_p$). 
Likewise, for every $q$ in $\overline{q}$, the formula $\mathsf{mv}(q)$ is either of the form $\mathsf{LA}(\alpha_q)[\nomj_k/u,\overline{\mathsf{mv}(p)}/\overline{p},\overline{\mathsf{mv}(q)}/\overline{q}]$ for some definite positive PIA formula $\alpha_q$ (and hence $-q\prec +\alpha_q$), or $\mathsf{mv}(q)$ is of the form $\mathsf{RA}(\beta_q)[\cnomm_h/u,\overline{\mathsf{mv}(p)}/\overline{p},\overline{\mathsf{mv}(q)}/\overline{q}]$ for some definite negative PIA formula $\beta_q$ (and hence $-q\prec -\beta_q$). The proof proceeds by induction on $<_{\Omega}$.  If $p$  is $<_{\Omega}$-minimal, then the form of $\mathsf{mv}(p)$ simplifies to either $\mathsf{LA}(\alpha_p)[\nomj_k/u]$ for some positive PIA formula $\alpha_p$ such that $+p\prec +\alpha_p$, or to  $\mathsf{RA}(\beta_p)[\cnomm_h/u]$ for some negative PIA formula $\beta_p$ such that $+p\prec -\beta_p$. In either case, items 2 and 3 of Lemma \ref{lemma:adjoints and ssc-sso} guarantee that $\mathsf{mv}(p)$ is ssc. Similarly, items 2 and 3 of Lemma \ref{lemma:adjoints and ssc-sso} guarantee that $\mathsf{mv}(q)$ is sso when $q$  is $<_{\Omega}$-minimal. The inductive step follows from items 2 and 3 of Lemma \ref{lemma:adjoints and ssc-sso}, the inductive hypothesis, and the polarities of the coordinates of the formulas $\mathsf{LA}(\alpha_p)$, $\mathsf{LA}(\alpha_q)$,  $\mathsf{RA}(\beta_p)$, and $\mathsf{RA}(\beta_q)$ (cf.~Lemma \ref{lemma:polarities of la-ra}); as an example, consider the case in which  $\mathsf{mv}(q)$ is of the form $\mathsf{LA}(\alpha_q)[\nomj_k/u,\overline{\mathsf{mv}(p)}/\overline{p},\overline{\mathsf{mv}(q)}/\overline{q}]$ for some positive PIA formula $\alpha_q$ (and hence $-q\prec +\alpha_q$). Then by Lemma \ref{lemma:adjoints and ssc-sso}.3, the formula $\mathsf{LA}(\alpha_q)[\nomj_k/!u,\overline{p},\overline{q}]$, which, by Lemma \ref{lemma:polarities of la-ra} is antitone in $u$ and $\overline{p}$ and monotone in $\overline{q}$, is sso; hence, by induction hypothesis and Lemma \ref{lemma:composition of ssc-sso}, $\mathsf{mv}(q): = \mathsf{LA}(\alpha_q)[\nomj_k/!u,\overline{\mathsf{mv}(p)}/\overline{p},\overline{\mathsf{mv}(q)}/\overline{q}]$ is sso.
 \end{proof}
%
%\marginnote{aggiunge ipotesi che la input ineq e' analytic inductive e quindi e' vero che ALBA succeeds on pivotal e inoltre per il lemma anteriore le Ackermann sono tutte su ssc e sso e quindi le ackermann sono applicate soundly per i lemma precedenti}
\begin{prop}[Correctness of executions of \textrm{ALBA} on analytic inductive inequalities under admissible assignments into slanted algebras]\label{Propo:Crtns:CanExtns} For any  analytic inductive $\mathcal{L}_\mathrm{LE}$-inequality $\phi \leq \psi$, if $\mathsf{ALBA}(\phi \leq \psi)$ denotes the set of pure $\mathcal{L}^+_\mathrm{LE}$-quasi-inequalities   generated by the \textrm{ALBA}-runs discussed in Section \ref{Spec:Alg:Section}, then for every slanted $\mathcal{L}_\mathrm{LE}$-algebra $\bba$,
\[\bbas \models_{\bba} \phi \leq \psi\quad \mbox{ iff } \quad\bbas \models_{\bba} \mathsf{ALBA} (\phi \leq \psi).\]
\end{prop}
\begin{proof}
The proof is similar to the correctness proof of ALBA runs under arbitrary assignments in the standard setting (see e.g.~\cite[Correctness Theorem]{CoPa-Dist} and \cite[Correctness theorem]{CoPa-nondist}). 
The only significant difference is that the Ackermann-rules are generally not invertible under admissible assignments, not even on standard algebras (cf.\ \cite[Example 9.1]{CoPa-Dist}), which, as discussed after Definition \ref{def:slanted LE}, correspond to a proper subclass of slanted algebras. However, by Lemmas \ref{Ackermann:Dscrptv:Left:Lemma} and \ref{Ackermann:Dscrptv:Right:Lemma}, when the left-hand and right-hand sides of all non-pure inequalities involved in the application of an Ackermann-rule are, respectively, ssc and sso, the rule is sound and invertible under admissible assignments. By Lemma \ref{Syn:Shape:Lemma}, this requirement on the syntactic shape is always satisfied when the rule is applied in the \textrm{ALBA}-runs discussed in Section \ref{Spec:Alg:Section}.
\end{proof}

\section{Transfer of canonicity for DLE-inequalities}
\label{sec:Sahlqvist canonicity via translation}
In \cite{CoPaZh19},  G\"odel-McKinsey-Tarski type translations (GMT-type translations) are used to obtain Sahlqvist correspondence and canonicity as transfer results in a number of settings. Specifically, GMT-type translations $\tau_\epsilon$ are defined parametrically in each order-type on a set $\mathsf{PROP}$ of propositional variables  so as to preserve the syntactic shape of $(\Omega, \epsilon)$-inductive inequalities in passing from arbitrary DLE-languages to corresponding target Boolean algebra expansion languages (BAE-languages) enriched with additional S4-modalities $\Diamond_{\geq}$ and $\Box_{\leq}$. While {\em correspondence} via translation is obtained in full generality for inductive inequalities in arbitrary DLE-languages (cf.~\cite[Theorem 6.1]{{CoPaZh19}}), the {\em canonicity} via translation of inductive inequalities is obtained in  \cite{{CoPaZh19}} only in the restricted setting of normal modal expansions of {\em bi-Heyting} algebras (bHAEs) (cf.~\cite[Theorem 7.1]{{CoPaZh19}}). The argument can be summarized by means of the following diagram: for every bHAE $\bba$ and every $(\Omega, \epsilon)$-inductive inequality $\phi\leq\psi$ of compatible signature, a BAE $\mathbb{B}$ exists such that the vertical bi-implications hold. Hence, the canonicity of $\phi\leq \psi$ follows from the fact that the BAE-inequality $\tau_\epsilon(\phi)\leq\tau_\epsilon(\psi)$ is an $(\Omega, \epsilon)$-inductive inequality, and that every such inequality has been shown to be canonical within generalized Sahlqvist theory in the framework of classical (i.e.~Boolean) modal logic (cf.~\cite{CoGoVaSEQMA}).

\begin{center}
	\begin{tabular}{l c l}
		
		$\bbA\models\phi\leq\psi$ & & $\bbas\models \phi\leq\psi$ \\
		%&&\\
		$\ \ \ \ \Updownarrow$   & & $\ \ \ \ \ \Updownarrow $ \\
		%&&\\
		$\mathbb{B}\models\tau_\epsilon(\phi)\leq\tau_\epsilon(\psi)$ &\ \ \ $\Leftrightarrow$ \ \ \  &  $\mathbb{B}^\delta\models\tau_\epsilon(\phi)\leq\tau_\epsilon(\psi)$\\
	\end{tabular}
\end{center}
As explained in \cite[Section 7.2]{CoPaZh19}, this argument could not be carried beyond the setting of bHAEs only because, although $\tau_\epsilon(\phi)\leq\tau_\epsilon(\psi)$ has the appropriate (inductive) syntactic shape, if $\bba$ is not a bHAE, the algebraic interpretation of the S4-modalities $\Diamond_{\geq}$ and $\Box_{\leq}$ in $\mathbb{B}$ turns out to be {\em slanted}  (according to the terminology introduced in the present paper), and the then state-of-the-art theory of canonicity would not account for inequalities between terms built out of slanted connectives. However, we are now in a position to apply Theorem \ref{Thm:ALBA:Canonicity} to justify the horizontal bi-implication of the diagram above, and hence to obtain the canonicity of a restricted class of analytic inductive inequalities in arbitrary DLE-signatures as a transfer result of the {\em slanted canonicity} of analytic inductive BAE-inequalities.  In what follows, we recall the definition of $\tau_\epsilon$, and then define the class of analytic inductive DLE-inequalities $\phi\leq \psi$ such that $\tau_\epsilon(\phi)\leq\tau_\epsilon(\psi)$ is analytic inductive.

\paragraph{Parametrized translation} Recall from \cite[Section 5.2.1]{CoPaZh19} that, for any normal DLE-signature $\mathcal{L}_{\mathrm{DLE}} = \mathcal{L}_{\mathrm{DLE}}(\mathcal{F}, \mathcal{G})$, the signature of the target language of the parametric GMT-type translations $\tau_{\epsilon}$ is the  normal BAE-signature $\mathcal{L}^{\circ}_{\mathrm{BAE}} = \mathcal{L}_{\mathrm{BAE}}(\mathcal{F}^{\circ}, \mathcal{G}^{\circ})$ where $\mathcal{F}^{\circ}: = \{\Diamond_\geq\}\cup
\{f^{\circ}\mid f\in \mathcal{F}\}$, and $\mathcal{G}^{\circ}: = \{\Box_\leq\}\cup
\{g^{\circ}\mid g\in \mathcal{G}\}$, and for every $f\in \mathcal{F}$ (resp.\ $g\in \mathcal{G}$), the  connective $f^{\circ}$ (resp.\ $g^{\circ}$) is such that $n_{f^{\circ}} = n_{f}$ (resp.\ $n_{g^{\circ}} = n_{g}$) and $\epsilon_{f^{\circ}}(i) = 1$ for each $1\leq i\leq n_f$ (resp.\ $\epsilon_{g^{\circ}}(i) = 1$ for each $1\leq i\leq n_g$).

The target language  for the parametrized GMT translations over $\mathsf{Prop}$ is given by
\[\mathcal{L}^\circ_{\mathrm{BAE}}\ni\alpha:: = p\mid \bot \mid \alpha\vee \alpha\mid\alpha\wedge \alpha \mid\neg\alpha\mid f^\circ(\overline{\alpha}) \mid g^\circ(\overline{\alpha})\mid \Diamond_{\geq}\alpha\mid \Box_{\leq}\alpha.\]

For any order-type $\epsilon$ on $\mathsf{PROP}$, the translation $\tau_\epsilon : \mathcal{L}_{\mathrm{DLE}}\to \mathcal{L}^\circ_{\mathrm{BAE}}$ is defined by the following recursion:

\begin{center}
	\begin{tabular}{c c c}
		$
		\tau_\epsilon(p) = \begin{cases} \Box_{\leq}  p &\mbox{if } \epsilon(p) = 1 \\
		\Diamond_{\geq} p & \mbox{if } \epsilon(p) = \partial, \end{cases}
		$
		& $\quad$ &

		\begin{tabular}{r c l}
			$\tau_\epsilon (\bot)$ &  = & $\bot$ \\
			$\tau_\epsilon (\top)$ &  = & $\top$ \\
			$\tau_\epsilon (\phi\wedge \psi)$ &  = & $\tau_\epsilon (\phi)\wedge \tau_\epsilon(\psi)$  \\
			$\tau_\epsilon (\phi\vee \psi)$ &  = & $\tau_\epsilon (\phi)\vee \tau_\epsilon(\psi)$  \\
			$\tau_\epsilon (f(\overline{\phi}))$ &  = & $\Diamond_{\geq}f^\circ(\overline{\tau_\epsilon (\phi)}^{\epsilon_f})$  \\
			$\tau_\epsilon (g(\overline{\phi}))$ &  = & $\Box_{\leq}g^\circ(\overline{\tau_\epsilon (\phi)}^{\epsilon_g})$  \\
		\end{tabular}
		\\
	\end{tabular}
\end{center}
where for each order-type $\eta$ on $n$ and any $n$-tuple $\overline{\psi}$ of $\mathcal{L}^\circ_{\mathrm{BAE}}$-formulas,  $\overline{\psi}^\eta$ denotes the $n$-tuple $(\psi'_i)_{i = 1}^n$, where $\psi_i' = \psi_i$  if $\eta(i) = 1$ and $\psi_i' = \neg \psi_i$  if $\eta(i) = \partial$.

\medskip 

It is clear from its definition that $\tau_{\epsilon}$ is intended to preserve the (good or excellent) shape of the $\epsilon$-critical branches of $(\Omega, \epsilon)$-inductive inequalities; however, $\tau_{\epsilon}$ will systematically destroy the good shape of non-critical branches (i.e.~$\epsilon^\partial$-critical branches) by inserting Skeleton nodes $+\Diamond_{\geq}$ and $-\Box_{\leq}$ in the scope of PIA nodes, whenever the given $\epsilon^\partial$-critical variable originally occurs in the scope of a PIA-connective. This motivates the following
\begin{definition}
\label{def: transferable}
For every order-type $\epsilon$ on $\mathsf{PROP}$,  an $(\Omega, \epsilon)$-analytic inductive inequality $(\varphi\leq \psi)[\overline{\alpha}/!\overline{x}, \overline{\beta}/!\overline{y},\overline{\gamma}/!\overline{z}, \overline{\delta}/!\overline{w}]$ (cf.~Notation \ref{notation: analytic inductive})  is $\tau_{\epsilon}$-{\em transferable} if for every maximal  positive (resp.~negative) $\varepsilon^{\partial}$-uniform PIA-subformula $\gamma$ in $\overline{\gamma}$ (resp.~$\delta$ in $\overline{\delta}$), either $\gamma = q$ (resp.~$\delta = p$) for some $q\in \mathsf{PROP}$ (resp.~$p\in \mathsf{PROP}$)  such that $\epsilon(q) = \partial$ (resp.~$\epsilon(p) = 1$), or $\gamma$ (resp.~$\delta$) does not contain atomic propositions at all.
\end{definition}
\begin{example} The inequality 
$\Diamond (\Box p_1\wedge \Box\Box p_2)\leq \Box (\Diamond \top \vee p_2)\wedge \Box (p_1\vee \Diamond \Diamond \top)$ is  $\tau_{\epsilon}$-transferable analytic $\epsilon$-Sahlqvist  for $\epsilon (p_1, p_2) = (1, 1)$. Its $\tau_\epsilon$-translation is the following analytic $\epsilon$-Sahlqvist inequality:
\[\Diamond_\geq \Diamond^\circ (\Box_\leq \Box^\circ \Box_\leq p_1\wedge \Box_\leq \Box^\circ\Box_\leq \Box^\circ \Box_\leq p_2)\leq \Box_\leq \Box^\circ (\Diamond_\geq \Diamond^\circ \top \vee \Box_\leq p_2)\wedge \Box_\leq \Box^\circ (\Box_\leq p_1\vee \Diamond_\geq \Diamond^\circ\Diamond_\geq \Diamond^\circ \top).\]
\end{example}
From the definition above, it  immediately follows that
\begin{prop}
For every $\tau_{\epsilon}$-transferable $(\Omega, \epsilon)$-analytic inductive $\mathcal{L}_{\mathrm{DLE}}$-inequality $\varphi\leq \psi$, the $\mathcal{L}_{\mathrm{BAE}}$-inequality $\tau_\epsilon(\phi)\leq\tau_\epsilon(\psi)$ is analytic inductive, and hence s-canonical (cf.~Theorem \ref{Thm:ALBA:Canonicity}).
\end{prop}
Hence, we can extend \cite[Theorem 7.1]{CoPaZh19} as follows:

\begin{theorem}[Canonicity  via translation]\label{theor:canon via transl bi HAE}
	For any order-type $\epsilon$ and any strict order $\Omega$ on $\mathsf{PROP}$,   the slanted canonicity theorem of analytic $(\Omega, \epsilon)$-inductive $\mathcal{L}^\circ_{\mathrm{BAE}}$-inequalities transfers to the standard canonicity of $\tau_{\epsilon}$-transferable analytic $(\Omega, \epsilon)$-inductive $\mathcal{L}_{\mathrm{DLE}}$-inequalities.
\end{theorem}

%definire i sottoins di analytic indictve la cui $\tau_\epsilon$ e' analytic inductive: caratterizz: le $\epsilon^\partial$-PIA si riducono a variabili proposizionali  oppure sono formule che non contengono variabili proposiz 

\section{Canonicity in the setting of subordination algebras}
\label{sec:subordination}
%\marginnote{devo ancora modificare questo preambolo per adattarlo alle nuove definizioni di tense BAE e cosi' via}
In \cite{DeHA}, the canonicity of a  subclass of Sahlqvist formulas (the so-called s-Sahlqvist formulas, cf.~Definition \ref{def:s-Sahlqvist}) in the signature of tense modal logic is shown w.r.t.~the semantics of subordination algebras and their canonical extensions. In this section, we obtain a strengthening of this result as a consequence of Theorem \ref{Thm:ALBA:Canonicity}, via the following steps: (a) equivalently presenting subordination algebras as a class of {\em slanted BAEs} (cf.~Definitions \ref{def:sub-algebras and slanted BAEs} and \ref{def: perfect sub-algebras and perfect BAEs}, and Proposition \ref{prop:subordination and slanted}); (b) verifying that  satisfaction and validity of tense formulas/inequalities are preserved and reflected across this equivalent presentation (cf.~Proposition \ref{prop:validity subordination and slanted});  (c) verifying that the algebraic canonicity of tense formulas in the setting of subordination algebras can be reduced to their {\em slanted} canonicity (cf.~Proposition \ref{prop: reducing sub-canonicity to s-canonicity}); (d)  recognizing s-Sahlqvist formulas as a proper subclass of analytic inductive formulas of classical tense logic (cf.~Proposition \ref{prop:s-sahlqvist are analytic sahlqvist}). Having understood the canonicity of s-Sahlqvist formulas in the setting of subordination algebras as  an instance of slanted canonicity makes it possible to consider various extensions of this result which we discuss  in the conclusions.

\begin{definition}[Subordination algebra]  A \emph{subordination algebra} is a pair $\bbs = (A, \prec)$ where $A$ is a Boolean algebra and $\prec$ is a binary relation on $A$ verifying the following conditions for all $a, b, c, d\in A$: 
\begin{enumerate}
\item[S1.] $0 \prec 0$ and $1 \prec 1$,
\item[S2.] $a, b \prec c$ implies $a \vee b \prec c $,
\item[S3.] $a \prec b,c $ implies $ a \prec b \wedge c$,
\item[S4.] $a \prec b \leq c \prec d$ implies $a \prec d$.
\end{enumerate}
\end{definition}
Properties S1-S4 imply that ${\prec} (a, -): = \{b\in A\mid a\prec b\}$ is a filter of $A$  and  ${\prec} (-, a): = \{b\in A\mid b\prec a\}$ is an ideal of $A$ for every $a\in A$. In what follows, we will sometimes use the notations ${\prec} (S, -): = \bigcup\{ {\prec} (a, -)\mid a\in S\}$ for any $S\subseteq A$, and ${\prec} (x, -): = \bigcup\{ {\prec} (a, -)\mid x\leq a\}$ for any $x\in \jty(A^\delta)$.

\begin{definition} 
\label{def: complete perfect SA} A subordination algebra $\bbs = (A, \prec)$ is \emph{complete} (resp.~\emph{perfect}) if $A$ is complete (resp.~complete and atomic), and  $\prec$ satisfies the following infinitary versions of conditions S2 and S3: for all $a \in A$ and $S \subseteq A$,
\begin{enumerate}
\item[S2$^\infty$.] if $ s \prec a$ for all $s \in S$, then $\bigvee S \prec a$;
\item[S3$^\infty$.] if $a \prec s $ for all $s \in S$, then $a \prec \bigwedge S$.
\end{enumerate}
The {\em  (constructive) canonical extension} of a subordination algebra  $\bbs = (A, \prec)$ (cf.~\cite[Definition 1.10]{DeHA}) is the structure $\bbss: = (A^\delta, {\prec}^\delta)$ such that  $A^\delta$ is the canonical extension of $A$ and ${\prec}^\delta$ is the binary relation defined as follows: 
\begin{enumerate}
\item if $k \in K(A^\delta)$ and $o \in O(A^\delta)$, then $k \prec^\delta o$  if $k \leq a \prec b \leq o$ for some $a,b \in A$,
\item if $u,v \in A^\delta$, then $u \prec^\delta v$ if  for all $k \in K(A^\delta)$ and $o \in O(A^\delta)$, $v \leq o$ and $k \leq u$ imply $k \prec^\delta o$.
\end{enumerate}
\end{definition}
The (constructive) canonical extension of a subordination algebra is a perfect (resp.~complete) subordination algebra (cf.~\cite[Definitions 1.7 and 1.10]{DeHA}). 

Recall that a {\em tense BAE} is a BAE $\bba = (A, \Diamond, \blacksquare)$  such that $\Diamond a\leq b$ iff $a\leq\blacksquare b$ for every $a, b\in A$. For any such tense BAE, we let $\Box$ and $\Diamondblack$ denote the modal operators dual to  $\Diamond $ and $\blacksquare$ respectively. That is, $\Box a: = \neg \Diamond \neg a$ and $\Diamondblack a: = \neg \blacksquare \neg a$ for any $a\in A$.
%\marginnote{forse dovremmo separare le due corrispondenze come nella proposta che segue}
Perfect (resp.~complete) subordination algebras can be associated with perfect (resp.~complete)  tense BAEs as follows:%\marginnote{ho generalizzato questa def al caso costruttivo e aggiunto il $\blacksquare$. Ma forse si possono aggiungere anche $\Diamondblack$ and $\square$? E fare la corrispondenza tra subordination e slanted per la segnatura con i 4 operatori invece che 2?}

\begin{definition}
\label{def: perfect sub-algebras and perfect BAEs}
For every perfect (resp.~complete) subordination algebra $\bbs = (A, \prec)$, its associated perfect (resp.~complete) tense BAE is $\bbs^+: = (A, \Diamond^{+}, \blacksquare^+)$ where $\Diamond^{+}: A\to A$ is defined by the assignment $u\mapsto \bigwedge \{v\in A\mid u\prec v\}$ and  $\blacksquare^+ : A \rightarrow A$ is defined by the assignment $u \mapsto \bigvee\lbrace v \in A \mid v \prec u \rbrace$; for every perfect (resp.~complete) tense BAE $\bba = (A, \Diamond, \blacksquare)$, we let $\bba_+:  = (A, \prec_+)$, where $u \prec_+ v$ iff $\Diamond u\leq v$, or equivalently, iff $u\leq \blacksquare v$ for all $u, v$ in $A$.
\end{definition}

\begin{definition}
A {\em tense slanted BAE} is a slanted BAE $\mathbb{S} = (A, \Diamond, \blacksquare)$ such that $A$ is a Boolean algebra, $\Diamond: A\to A^\delta$ is a c-slanted finitely join-preserving map, $\blacksquare: A\to A^\delta$ is an o-slanted finitely join-preserving map and moreover, for every $a, b\in A$, \[\Diamond a\leq b \quad \mbox{ iff }\quad a\leq\blacksquare b.\] For such an s-algebra, we let $\Box: A\to A^\delta$ denote the o-slanted operator defined by the assignment $a\mapsto \neg^{A^\delta} \Diamond \neg^{A} a$ and $\Diamondblack: A\to A^\delta$ denote the c-slanted operator defined by the assignment $a\mapsto \neg^{A^\delta} \blacksquare \neg^{A} a$. It is straightforward to show that $\Diamondblack a\leq b$ iff $a\leq\Box b$ for every $a, b\in A$.
\end{definition}

\begin{lemma}\label{lem_Adelta_tense} If $\bbs = (A, \Diamond, \blacksquare)$ is a tense slanted BAE, then its canonical extension $\bbs^\delta =(A^\delta, \Diamond^\delta, \blacksquare^\delta)$ is a perfect tense BAE.
\end{lemma}
\begin{proof}
Let $k \in K(A^\delta)$ and $o \in O(A^\delta)$ such that $\Diamond^\sigma k \leq o$, that is $ \bigwedge \lbrace \Diamond a \mid k \leq a \in A \rbrace \leq \bigvee \lbrace b \in A \mid b \leq o \rbrace$. By compactness and the monotonicity of $\Diamond$, this implies that $\Diamond a_0 \leq b_0$ for some $a_0 \geq k$ and $b_0 \leq o$. So, by adjunction,  $a_0 \leq \blacksquare b_0$. Hence 
\[ k = \bigwedge \lbrace a \in A \mid k \leq a \rbrace \leq a_0 \leq \blacksquare b_0 \leq \bigvee \lbrace \blacksquare b \mid A \ni b \leq o \rbrace = \blacksquare^\pi o. \]
Let $u,v \in A^\delta$ such that $\Diamond^\delta u \leq v$. Then $\Diamond^\sigma k \leq o$, and hence (cf.~argument above) $k \leq \blacksquare^\pi o$, for all  $K(A^\delta)\ni k \leq u$ and all  $O(A^\delta)\ni o \geq v$. Therefore,  
\[ u = \bigvee \lbrace k \in K(A^\delta) \mid k \leq u \rbrace \leq \bigwedge \lbrace \blacksquare^\pi o \mid v \leq o \in O(A^\delta) \rbrace = \blacksquare^\delta o, \] as required.
Dually, one shows that $u \leq \blacksquare^\delta v$ implies $\Diamond^\delta u \leq v$ for all $u,v \in A^\delta$, which completes the proof that $\bbs^\delta$ is a tense algebra.
\end{proof}

Subordination algebras can be equivalently presented as  tense slanted BAEs as follows: %\marginnote{sei d'accordo con questa definizione?}
\begin{definition}
\label{def:sub-algebras and slanted BAEs}
For every subordination algebra $\bbs = (A, \prec)$, its associated tense slanted BAE is $\bbs^*: = (A, \Diamond_{\prec}, \blacksquare_\prec)$ where $\Diamond_{\prec}: A\to A^\delta$ is defined by the assignment $a\mapsto \bigwedge \{b\in A\mid a\prec b\}\in K(A^\delta)$ and $\blacksquare_\prec : A \rightarrow A^\delta$  by the assignment $a \mapsto \bigvee\lbrace b \in A \mid b \prec a \rbrace \in O(A^\delta)$; for every tense slanted BAE $\bba = (A, \Diamond, \blacksquare)$, its associated subordination algebra is $\mathbb{A}_*: = (A, \prec_\Diamond)$, where $a \prec_\Diamond b$ iff $\Diamond a\leq b$ iff $a\leq\blacksquare b$.
\end{definition} %\marginnote{ho notato che tu usi anche la def di $\prec_\blacksquare$ come primitiva. Mi sembra che non ce ne sia bisogno, perche' $a \prec_\Diamond b$ iff $\Diamond a\leq b$ iff $a\prec_\blacksquare b$ iff (per definizione) $a \prec_\blacksquare b$, ossia secondo questa definizione, $\prec_\blacksquare$ e' semplicemente una notazione alternativa. Sei d'accordo?}

Notice that the defining assignments of $\Diamond^{+}$ and $\Diamond_{\prec}$ (resp.~of $\blacksquare^+$ and $\blacksquare_\prec$) are verbatim `the same' (however,  the meets and joins are taken in different algebras) but the functional types of $\Diamond^{+}$ and $\Diamond_{\prec}$ (resp.~of $\blacksquare^+$ and $\blacksquare_\prec$) are different.

\begin{prop}
\label{prop:subordination and slanted}
For every subordination algebra $\bbs = (A, \prec)$ and every tense slanted BAE $\bba = (A, \Diamond, \blacksquare)$,
\begin{enumerate}
\item $\bbs^*$ is a tense slanted BAE, and if $\bbs$ is perfect, then $\bbs^+$ is a perfect tense BAE in the standard sense;
\item $\bba_*$ is a subordination algebra, and if $\bba$ is a perfect  tense BAE in the standard sense, then $\bba_+$ is a perfect subordination algebra;
\item $(\bbs^*)_*= \bbs$ and if $\bbs$ is perfect, then $(\bbs^+)_+= \bbs$;
\item $(\bba_*)^*= \bba$ and if $\bba$ is perfect, then $(\bba_+)^+= \bba$;
\item $(\Diamond_{\prec})^\delta = \Diamond_{\prec^\delta}$ and $(\blacksquare_{\prec})^\delta = \blacksquare_{\prec^\delta}$;
\item $\prec_{\Diamond^\delta} = ({\prec_\Diamond})^\delta$;
\item $(\bbss)^+ = (\bbs^*)^\delta$;
\item $(\bba^\delta)_+ = (\bba_*)^\delta$;
\item if $\bbs$ is perfect, then $(\bbs^\delta)^+ = (\bbs^+)^\delta$;
\item if $\bba$ is perfect, then $(\bba^\delta)_+ = (\bba_+)^\delta$.
\end{enumerate}
\end{prop}
\begin{proof}
\begin{enumerate}
\item By construction, $\Diamond_\prec$ and $\blacksquare_\prec$ are c-slanted and o-slanted respectively. Hence, it is enough to show that they are normal and satisfy the tense condition. The identities $\Diamond_\prec 0 = 0 $ and $\blacksquare_\prec 1 = 1$ follow directly from S1. Moreover, for any $a,b \in A$, axiom S4 implies that ${\prec}(a,-) \cup {\prec}(b,-) \supseteq {\prec}(a\vee b,-)$, which implies that $\Diamond_\prec a \vee \Diamond_\prec b \leq \Diamond_\prec (a\vee b)$. Conversely, $ \Diamond_\prec a \vee \Diamond_\prec b = \bigwedge \lbrace c \vee d \mid a \prec c \text{ and } b \prec d \rbrace$. From S2 and S4, if $a \prec c$ and $b \prec d$ then $ a \vee b \prec c \vee d$. Hence, $\Diamond_\prec (a \vee b) \leq \Diamond_\prec a \vee \Diamond_\prec b$, as required. Similarly, one  shows that $\blacksquare_{\prec}(a \wedge b) = \blacksquare_{\prec} a \wedge \blacksquare_{\prec} b$. Finally, for every $a,b \in A$, 
\begin{equation}\label{eq_tense}
\Diamond_\prec a \leq b\quad  \text{ iff } \quad a \prec b\quad \text{ iff } \quad a \leq \blacksquare_\prec b.
\end{equation}
 Indeed, by construction, $a \prec b$ implies $\Diamond_\prec a \leq b$ and $a \leq \blacksquare_\prec b$. Moreover, if $\Diamond_\prec a \leq b$, then compactness and the definition of $\Diamond_\prec$ imply that $a \prec c \leq b$ for some $c \in A$, which implies $a \prec  b$ by S4.  Similarly, one shows that $a \leq \blacksquare_\prec b$ implies $a \prec b$, which completes the proof that $\Diamond_\prec$ and $\blacksquare_\prec$ satisfy the tense condition.

 The proof of the second part of the statement (when $\bbs$ is perfect) is  similar with a slight difference: the equivalence \eqref{eq_tense} arises from the completeness of $\prec$ rather than from the compactness of $A^\delta$. Indeed, $\Diamond^+ a \leq b$ implies that $a \prec \bigwedge \lbrace c \in A \mid a \prec c \rbrace = \Diamond^+a \leq b$, which implies $a \prec  b$ by S4.
 \item It is routine to show that $\prec_\Diamond$ (resp. $\prec_+$) satisfies conditions S1 to S4 (resp. their infinitary versions).
\item By definition, the underlying Boolean algebras of $\bbs$ and $(\bbs^*)_*$ (resp. $(\bbs^+)_+$ if $\bbs$ is perfect) are identical. Moreover, equivalences \eqref{eq_tense}, already proven in item 1, % that $a \prec b$ if and only if $\Diamond_\prec a \leq b$ if and only if $a \leq \blacksquare_\prec b$ (resp. $\Diamond^+ a \leq b$ and $a \leq \blacksquare^+ b$). 
imply that the subordination relations of $\bbs$ and $(\bbs^*)_*$ (resp.~$(\bbs^+)_+$)  coincide.
\item The tense BAEs $(\bba_*)^*$ and $\bba$ share the same underlying Boolean algebra. Hence, $(\bba_*)^* = \bba$ if and only if $\Diamond_{\prec_\Diamond} a: =  \bigwedge \lbrace b \in A \mid \Diamond a \leq b \rbrace  = \Diamond a $, and $ \blacksquare_{\prec_\Diamond} a := \bigvee \lbrace b \in A \mid \Diamond b \leq a \rbrace= \bigvee \lbrace b \in A \mid b \leq \blacksquare a \rbrace = \blacksquare a $. These identities immediately follow from  $\Diamond a\in K(A^\delta)$ and $\blacksquare a\in O(A^\delta)$.

For the perfect case, the equalities $\Diamond a = \Diamond_{\prec_\Diamond} a$ and $\blacksquare a = \blacksquare_{\prec_\Diamond} a$ are trivially verified, since $\Diamond a$ and $\blacksquare a$ are elements of $A$ and the infimum and supremum are taken in $A$ itself. 
\item Let  us preliminarily show that $(\Diamond_\prec)^\sigma k = \Diamond_{\prec^\delta}k$ for any $k\in K(A^\delta)$. %By definition, this identity is equivalent to the following one:
%\[ \bigwedge \lbrace \Diamond_\prec a \mid a \in A \text{ and } k \leq a\rbrace = \bigwedge\lbrace u \in A^\delta \mid k \prec^\delta u \rbrace. \]
In order to show that $\bigwedge\lbrace u \in A^\delta \mid k \prec^\delta u \rbrace = :\Diamond_{\prec^\delta} k \leq (\Diamond_\prec)^\sigma k: = \bigwedge \lbrace \Diamond_\prec a \mid a \in A \text{ and } k \leq a\rbrace$ it is enough to show that $k \prec^\delta \Diamond_\prec a$ for all $a\in A$ such that $k \leq a$. Since $k$ is closed, by definition (cf.~item 2 of Definition \ref{def: complete perfect SA}) this is equivalent to showing that $k \prec^\delta o$ for every $o\in O(A^\delta)$  such that $\Diamond_\prec a \leq o$. By compactness, 
$\Diamond_\prec a: = \bigwedge \lbrace b \in A \mid a \prec b \rbrace  \leq o$ implies that $b_1 \wedge \ldots \wedge b_n \leq o$ for some $b_1,\ldots,b_n \in {\prec}(a,-)$. Hence, by axiom S3,
$ k \leq a \prec b_1 \wedge \ldots \wedge b_n \leq o$, which shows that $k \prec^\delta o$, as required.

Conversely, note first that, by denseness, 
\[ \Diamond_{\prec^\delta} k :=\bigwedge \lbrace u \in A^\delta \mid k \prec^\delta u \rbrace = \bigwedge \lbrace o \in O(A^\delta) \mid k \prec^\delta o \rbrace.\] 
Hence, to prove $(\Diamond_\prec)^\sigma k\leq \Diamond_{\prec^\delta} k$, it is enough to show that $(\Diamond_\prec)^\sigma k \leq o$ for every $o \in O(A^\delta)$  such that $k \prec^\delta o$. For such an $o$, by definition, $k \leq a \prec b \leq o$ for some $a,b \in A$. Hence, by definition,  $(\Diamond_\prec)^\sigma k \leq \Diamond_\prec a \leq b \leq o$, as required. 
The identity $(\Diamond_\prec)^\sigma u = \Diamond_{\prec^\delta}u$ for all $u \in A^\delta$ follows straightforwardly from 
 $(\Diamond_\prec)^\sigma k = \Diamond_{\prec^\delta}k$ for all $k \in K(A^\delta)$ using the denseness of $A^\delta$ and the complete join-preservation of $\Diamond_{\prec^\delta}$ and $(\Diamond_\prec)^\delta$.

Dually, one shows that $(\blacksquare_\prec)^\pi o = \blacksquare_{\prec^\delta} o$ for all $o \in O(A^\delta)$ and therefore, $(\blacksquare_\prec)^\pi u = \blacksquare_{\prec^\delta} u$ for all $u \in A^\delta$.
\item Let us preliminarily show that $k \prec_{\Diamond^\delta} o$ iff $k \mathrel{(\prec_\Diamond)^\delta} o$ for every  $k \in K(A^\delta)$ and $o \in O(A^\delta)$.
%By Lemma \ref{lem_Adelta_tense},  $\bbas$ is a tense algebra, which gives rise to the subordination relation $\prec_{\Diamond^\delta}$ associated with both $\Diamond^\delta$ and $\blacksquare^\delta$. Let $k \in K(A^\delta)$ and $o \in O(A^\delta)$ such that 
If $k \prec_{\Diamond^\delta} o$, that is \[\bigwedge \lbrace \Diamond a \mid a \in A \text{ and } k \leq a \rbrace =:  \Diamond^\delta k\leq o = \bigvee \lbrace b \in A \mid b \leq o \rbrace, \]
then, by compactness and since $\Diamond$ is monotone,  $\Diamond a \leq b$ (i.e.~$a\prec_{\Diamond} b$) for some $a \in A$ and $b \in A$ such that $k \leq a$ and $b \leq o$.  Hence, $k \mathrel{(\prec_\Diamond)^\delta} o$. Conversely, if $k \mathrel{(\prec_\Diamond)^\delta} o$, i.e.~if $\Diamond a \leq b$ for some $a,b \in A$ such that $k \leq a$  and $b \leq o$, then $\Diamond^\delta k\leq \Diamond a \leq b\leq o$, which yields $k \prec_{\Diamond^\delta} o$, as required.
Let us  show that $u \prec_{\Diamond^\delta} v$ iff $u \mathrel{(\prec_\Diamond)^\delta} v$ for all  $u,v \in A^\delta$. We have %Let us fix such $u$ and $v$. %By definition, $u \mathrel{(\prec_\Diamond)^\delta} v $ iff $k \mathrel{(\prec_\Diamond)^\delta} o$ for any $k \in K(A^\delta)$ and $ o \in O(A^\delta)$ such that $k \leq u $ and $ v \leq o$. By the claim shown above, this is equivalent  to $k \prec_{\Diamond^\delta} o$ for any $k \in K(A^\delta)$ and $ o \in O(A^\delta)$ such that $k \leq u $ and $ v \leq o$. As shown in item 2,  $\prec_{\Diamond^\delta}$ is a perfect subordination relation, hence the condition above is equivalent to $u = \bigvee \lbrace k \in K(A^\delta) \mid k \leq u \rbrace \prec_{\Diamond^\delta} \bigwedge \lbrace o \in O(A^\delta) \mid v \leq o \rbrace = v$

\begin{align}
& u \mathrel{(\prec_\Diamond)^\delta} v \label{eqeq1} \\
\iff & k \mathrel{(\prec_\Diamond)^\delta} o \text{ for any $k \in K(A^\delta)$ and $ o \in O(A^\delta)$ such that $k \leq u $ and $ v \leq o$ } \label{eqeq2}\\
\iff & \text{$k \prec_{\Diamond^\delta} o$ for any $k \in K(A^\delta)$ and $ o \in O(A^\delta)$ such that $k \leq u $ and $ v \leq o$} \label{eqeq3} \\
\iff & \text{$\bigvee \lbrace k \in K(A^\delta) \mid k \leq u \rbrace \prec_{\Diamond^\delta} \bigwedge \lbrace o \in O(A^\delta) \mid v \leq o \rbrace $} \label{eqeq4} \\
\iff & u  \prec_{\Diamond^\delta}  v\label{eqeq5}
\end{align}
where $\eqref{eqeq1} \iff \eqref{eqeq2}$ is the definition of $(\prec_\Diamond)^\delta$, $\eqref{eqeq2} \iff \eqref{eqeq3}$ is obtained via the preliminary claim, $\eqref{eqeq3} \iff \eqref{eqeq4}$ follows from axioms S2$^\infty$  and S3$^\infty$ for $\prec_{\Diamond^\delta}$ (cf.~item 2) and finally $\eqref{eqeq4} \iff \eqref{eqeq5}$ is denseness.

\item and (8)  $A^\delta$ is the Boolean algebra  underlying  $(\bbs^\delta)^+$, $(\bbs^*)^\delta$, $(\bba^\delta)_+$ and $(\bba_*)^\delta$. Moreover, the modal operators of $(\bbs^\delta)^+$ and $(\bbs^*)^\delta$ are respectively $\Diamond_{\prec^\delta}$ and $\blacksquare_{\prec^\delta}$ and $(\Diamond_\prec)^\delta$ and $(\blacksquare_\prec)^\delta$ which coincide pairwise (cf.~item 5). Finally, the subordination relations of $(\bba^\delta)_+$ and $(\bba_*)^\delta$ are respectively $\prec_{\Diamond^\delta}$ and $(\prec_\Diamond)^\delta$ which coincide,  (cf.~item 6). 
\item[(9)] and (10) The proofs are relatively similar to the ones of the non-perfect case with slightly different justifications: as in item 1, the completeness of $\prec$, $\Diamond$ and $\blacksquare$ is used instead of the compactness of $A^\delta$. As an example, we prove item 9 and leave item 10 to the reader. As remarked above, $(\bbs^\delta)^+$ and $(\bbs^+)^\delta$ have $A^\delta$ as their underlying Boolean algebras. Hence,  to finish the proof, let us show that the modal operators coincide. Since $(\Diamond^+)^\delta$ and $\Diamond_{\prec^\delta} $ are completely join-preserving, by denseness it enough to show that  for every $k \in K(A^\delta)$,  \[ \Diamond_{\prec^\delta} k := \bigwedge \lbrace o \in O(A^\delta) \mid  k \leq a \prec b \leq o \text{ for some } a,b \in A  \rbrace = \bigwedge \lbrace \Diamond^+ a \mid k \leq a \in A \rbrace := (\Diamond^+)^\delta k. \] If $k \leq a $, then $b:=\Diamond^+ a\in A\subseteq O(A^\delta)$ %is an element of $A$, hence an open element of $A^\delta$, 
and $k \leq a \prec \Diamond^+ a \leq \Diamond^+ a$, which implies that $\Diamond_{\prec^\delta}k\leq (\Diamond^+)^\delta k$. Conversely, if $o \in O(A^\delta)$ is such that $k \leq a \prec b \leq o$ for some $a,b \in A$, then $\Diamond^+ a \leq b \leq o$ and hence $(\Diamond^+)^\delta k \leq \Diamond^+ a \leq b \leq o\leq \Diamond_{\prec^\delta} k$. Dually, one shows that $\blacksquare_{\prec^\delta} o = (\blacksquare^+)^\delta o$ for every $o \in O(A^\delta)$, which is enough to prove that $\blacksquare_{\prec^\delta} $ and $ (\blacksquare^+)^\delta $ coincide.
\end{enumerate}
\end{proof}

\begin{remark}\label{rem_tense_open_closed}In Proposition \ref{prop:subordination and slanted}, we showed that subordination algebras can be equivalently presented as tense slanted BAEs. %However, if $\bba = (A, \Diamond, \blacksquare)$ is a slanted tense BAE, it is not necessary to consider both the operators $\Diamond$ and $\blacksquare$ to be able to construct the subordination relation of $\bba_*$. 
In fact, %since subordination relations can be defined in terms of one modal operator, 
subordination algebras can be also equivalently presented  both as  slanted BAEs of the form $\bba_c =(A, \Diamond)$ (which we refer to as {\em closed} slanted BAEs), and as slanted BAEs of the form $\bba_o = (A, \blacksquare)$ (which we refer to as {\em open} slanted BAEs). Hence,   closed, open and tense slanted BAEs  are all equivalent presentations. These equivalences can of course be described without using subordination algebras as mediators. Namely, a slanted tense BAE $\bba = (A, \Diamond, \blacksquare)$ is mapped to the closed slanted BAE $\bba_c = (A, \Diamond)$ while a closed slanted BAE $\bba = (A, \Diamond)$ is mapped to the tense algebra $\bba_t = (A, \Diamond, \blacksquare_\Diamond)$ where $\blacksquare_\Diamond$ is the restriction to $A$ of the adjoint of $\Diamond^\delta$. The equivalence between tense and open BAEs is defined similarly.
\end{remark}

\begin{remark} %If $\bbs = (A, \prec)$ is a subordination algebra, then $\bbs^\thicksim:=(A, \blacksquare_\prec)$ where the operator $\blacksquare_\prec : A \rightarrow A^\delta$ is defined by $a \mapsto \vee\lbrace b \in A \mid b \prec a \rbrace \in O(A^\delta)$ is another slanted BAE associated to $\bbs$. Conversely, if $\bba = (A,\blacksquare)$ is a slanted BAE, then its associated subordination algebra is $\bba_\thicksim:=(A, \prec_\blacksquare)$ where $a \prec_\blacksquare b$ if and only if $a \leq \blacksquare b$. There is a proposition analogue to Proposition \ref{prop:subordination and slanted} for what we could call the "open case".

%Finally, for a subordination algebra $\bbs = (A, \prec)$, the operator $\blacksquare_\prec$ is the adjoint of $\Diamond_\prec$ since we clearly have 
%\[ \Diamond_\prec a \leq b \Leftrightarrow a \prec b \Leftrightarrow a \leq \blacksquare_\prec b. \]

Open slanted BAEs (in the sense of Remark \ref{rem_tense_open_closed}) are isomorphic to the quasi-modal algebras developed by Celani in \cite{Quasimodal}. Recall that a \emph{quasi-modal algebra} is a pair $\mathbb{Q} =(B, \Delta)$ where $B$ is a Boolean algebra, $\Delta: B\to \mathcal{I}(B)$, where $\mathcal{I}(B)$ denotes the set of the ideals of $B$, satisfying the following conditions: $\Delta(a \wedge b) = \Delta a \cap \Delta b$ and $\Delta 1 = A$. It is then clear that the order-isomorphism between the ideals of $B$ and open elements of $B^\delta$ (cf.~\cite[Theorem 2.5]{DUGePa2005}) can be used to establish an equivalence between quasi-modal algebras and open slanted BAEs (see Example \ref{Exem:gen_imp}). But this equivalence is not surprising, given that subordination algebras and quasi-modal algebras are known to be equivalent (cf.~e.g.~\cite[Theorem 15]{Celani}).
\end{remark}

Let $\mathcal{L} = \mathcal{L}(\mathcal{F}, \mathcal{G})$ be the BAE language such that $\mathcal{F} = \{\Diamond,  \Diamondblack\}$ and $\mathcal{G} = \{\square, \blacksquare\}$, all modal connectives being unary and positive.
Satisfaction and validity of $\mathcal{L}$-formulas/inequalities on subordination algebras can be defined in terms of  Definition \ref{def: perfect sub-algebras and perfect BAEs} as follows:
\begin{definition}
For every subordination algebra $\bbs = (A, \prec)$ every assignment $v:\mathsf{PROP}\to A$, and every modal inequality $\phi\leq\psi$,
\[(\bbs, v)\models \phi\leq\psi \quad\mbox{ iff }\quad ((\bbs^\delta)^+, e\cdot v)\models \phi\leq\psi \]
where $e: A \to A^\delta$ is the canonical embedding. As to validity,
\[\bbs\models \phi\leq\psi \quad\mbox{ iff }\quad (\bbs^\delta)^+\models_{\bbs} \phi\leq\psi. \]
\end{definition} 

\begin{prop}
\label{prop:validity subordination and slanted}
For every (perfect) subordination algebra $\bbs = (A, \prec)$ every slanted (resp.~perfect) BAE $\bba = (A, \Diamond)$, and every  $\mathcal{L}$-inequality $\phi\leq\psi$,
\begin{enumerate}
\item $\bbs\models \phi\leq\psi\quad$ iff $\quad\bbs^*\models \phi\leq\psi$;
\item $\bba\models \phi\leq\psi\quad$ iff $\quad\bba_*\models \phi\leq\psi$;
\item $\bbs\models \phi\leq\psi\quad$ iff $\quad\bbs^+\models \phi\leq\psi$;
\item $\bba\models \phi\leq\psi\quad$ iff $\quad\bba_+\models \phi\leq\psi$.
\end{enumerate}
\end{prop}
\begin{proof}
For item 1, we recall that $\bbs \models \phi \leq \psi$ if and only if $(\bbs^\delta)^+ \models_{\bbs} \phi \leq \psi$. We also recall that, by Proposition \ref{prop:subordination and slanted}, we have  $(\bbss)^+ = (\bbs^*)^\delta$. Hence, we have $\bbs \models \phi \leq \psi$ if and only if $(\bbs^*)^\delta \models_{\bbs} \phi \leq \psi$. The conclusion now follows from the fact that $\bbs$ and $\bbs^*$ have the same underlying Boolean algebra. Items 2 to 4 are proved similarly.
\end{proof}

\begin{prop}
\label{prop: reducing sub-canonicity to s-canonicity}
For every s-canonical $\mathcal{L}$-inequality $\phi\leq\psi$ and every subordination algebra $\bbs$,
\[\bbs\models\phi\leq\psi\quad \Leftrightarrow \quad\bbss\models \phi\leq\psi.\]
\end{prop}
\begin{proof} The argument can be summarized by means of the following diagram:
\begin{center}
\begin{tabular}{l c l}
$\bbs \ \ \models \phi \leq \psi$ & &$\ \ \ \ \ \ \ \bbss \models \phi \leq \psi$\\
&&$\ \ \ \ \ \ \ \ \ \ \ \Updownarrow$ \\
%
%$\bbas \models_{\bba} \phi \leq \psi$ & & \ \ \ \ \ \ \ \ \ \ \  $\Updownarrow $\\
%
$ \ \ \ \ \ \Updownarrow$&& $((\bbs^*)^\delta)_+\models \phi\leq\psi$\\
&&$\ \ \ \ \ \ \ \ \ \ \ \Updownarrow$ \\
$\bbs^*\models \phi\leq\psi$
&\ \ \ $\Leftrightarrow$ \ \ \ &$(\bbs^*)^\delta\models\phi\leq\psi$
\end{tabular}
\end{center}
The bi-implication on the left is due to Proposition \ref{prop:validity subordination and slanted}.1; the horizontal bi-implication holds by assumption;  the lower bi-implication on the right is due to Proposition \ref{prop:validity subordination and slanted}.4;  the upper bi-implication on the right is due to Proposition \ref{prop:subordination and slanted}.7.
\end{proof}
Hence, as an immediate consequence Proposition \ref{prop: reducing sub-canonicity to s-canonicity} and Theorem \ref{Thm:ALBA:Canonicity}, we get the following
\begin{cor}
For every analytic inductive $\mathcal{L}$-inequality $\phi\leq\psi$ and every subordination algebra $\bbs$,
\[\bbs\models\phi\leq\psi\quad \Leftrightarrow \quad\bbss\models \phi\leq\psi.\]
\end{cor}
Finally, we show that the corollary above strengthens  \cite[Corollary 3.8]{DeHA}, by verifying that sub-Sahlqvist $\mathcal{L}$-formulas are a proper subclass of analytic inductive $\mathcal{L}$-formulas.
\begin{definition} 
\label{def:s-Sahlqvist} %Let $\mathcal{L} = \mathcal{L}(\mathcal{F}, \mathcal{G})$ be the BAE language such that $\mathcal{F} = \{\Diamond,  \Diamondblack\}$ and $\mathcal{G} = \{\square, \blacksquare\}$, all modal connectives being unary and positive.
\begin{enumerate}
\item An $\mathcal{L}$-formula  is \emph{closed} (resp.~\emph{open}) if it is built up from constants $\top$, $\bot$, propositional variables and their negations, by applying $\vee$, $ \wedge$, $\Diamond$ and $\Diamondblack$ (resp. $\vee$, $\wedge$, $\square$ and $\blacksquare$).
\item An $\mathcal{L}$-formula  is \emph{positive} (resp.~\emph{negative}) if it is built up from constants $\top$, $\bot$ and propositional variables (resp.~negations of propositional variables) by applying $\wedge$, $\vee$, $\Diamond$, $\square$, $ \Diamondblack$ and $\blacksquare$.
\item An $\mathcal{L}$-formula  is \emph{sub-positive} (resp.~\emph{sub-negative}) if it is built up from closed positive formulas (resp.~open negative formulas) by applying $\vee$, $\wedge$, $\square$ and $\blacksquare$ (resp.~$\vee$, $\wedge$, $\Diamond$ and $\Diamondblack$).
\item A {\em boxed atom} is an $\mathcal{L}$-formula built up from propositional variables by applying $\square$ and $\blacksquare$.
\item An $\mathcal{L}$-formula  is \emph{strongly positive} if it is a conjunction of  boxed atoms. %formulas of the form \[ \square^{\langle \mu \rangle} p := \square^{\mu_1}\blacksquare^{\mu_2}... \square^{\mu_k} p , \] where $p \in \mathsf{PROP}$ and $\mu \in \mathbb{N}^k$ for some $k \in \mathbb{N}$.
\item An $\mathcal{L}$-formula  is \emph{untied} if it is built up from strongly positive and sub-negative formulas using only $\wedge$, $\Diamond$ and $\Diamondblack$.
\item A \emph{sub-Sahlqvist formula} is an $\mathcal{L}$-formula of the form  $\psi[(\varphi_1 \rightarrow \varphi_2)/!x]$ where $\psi(!x)$ is a boxed atom, %\square^{\langle \mu \rangle}(\varphi_1 \rightarrow \varphi_2)$ with  
$\phi_1$ is untied, and $\phi_2$ is sub-positive. % and $\mu \in \mathbb{N}^l$ for some $l \in \mathbb{N}$.
\end{enumerate}
\end{definition}
%From the definition above and Definition \ref{} it immediately follows that
\begin{prop}
\label{prop:s-sahlqvist are analytic sahlqvist}
For every $\mathcal{L}$-formula $\phi$, letting $\epsilon$ denote the order-type constantly equal to 1,
\begin{enumerate}
\item If $\phi$ is closed (resp.~open), then $-\phi$ is PIA (resp.~Skeleton).
\item  If $\phi$ is positive (resp.~negative), then $-\phi$ (resp.~$+\phi$) is $\epsilon^\partial$-uniform.
\item If $\phi$ is closed positive (resp.~open negative), then $-\phi$ (resp.~$+\phi$) is  $\epsilon^\partial$-uniform PIA.
\item If $\phi$ is strongly positive, then $+\phi$ is  PIA, and each of its branches is excellent.
\item If $\phi$ is sub-positive (resp.~sub-negative), then  $-\phi$ (resp.~$+\phi$) is $\epsilon^\partial$-uniform and all of its branches are good.
\item If $\phi$ is untied, then $+\phi$ is analytic $\epsilon$-Sahlqvist. % for the order-type $\epsilon$ constantly equal to 1.
\item If $\phi$ is sub-Sahlqvist, then $-\phi$ is analytic $\epsilon$-Sahlqvist, hence so is $\top\leq \phi$.
\end{enumerate}
\end{prop}
\begin{proof}
Items 1 and 2 immediately follow from the definitions involved. Item 3 is an immediate consequence of items 1 and 2. Item 4 immediately follows from item 3 and the definition of excellent branch. Item 5 follows from item 3 and the definition of a good branch. Item 6 follows from the fact that, by items 4 and 5, any untied formula is built up from positive PIA-formulas every branch of which is excellent and $\epsilon^\partial$-uniform formulas every branch of which is good, by applying Skeleton connectives. Clearly, this application will maintain both the good shape of $\epsilon^\partial$-critical branches, and the excellent shape of $\epsilon$-critical branches.
Finally, item 7 follows from the observation that if $\psi(!x)$ is a boxed atom, then $-\psi[(y \rightarrow z)/!x]$ is a Skeleton formula, and hence,  replacing the placeholder variable $z$ with the $\epsilon^\partial$-uniform formula $\phi_2$ all of the branches  of which are good, and the placeholder variable $x$ with the analytic $\epsilon$-Sahlqvist formula $\phi_1$ will yield again an analytic $\epsilon$-Sahlqvist formula.
\end{proof}

\section{Conclusions}
\label{sec:conclusions}
In the present paper, we have explored the {\em topological} properties of a class of LE-inequalities, the {\em analytic inductive inequalities}, which has been originally introduced in \cite{GMPTZ} as a concrete syntactic approximation of the {\em proof-theoretic} notion of analyticity in the context of proper display calculi \cite{wansing2013displaying}. The theoretical background in which this connection  between topological and proof-theoretic properties could be established is {\em unified correspondence theory} \cite{UnifCorresp}, which applies algebraic and duality-theoretic techniques in the development of (generalized) Sahlqvist correspondence and canonicity results for nonclassical logics, and which has recently established systematic connections between generalized Sahlqvist theory and the core issue in structural proof theory of identifying large classes of analytic axioms and algorithmically computing their corresponding analytic structural rules, yielding precisely  the notion of analytic inductive inequalities.   
The main result of the present paper is that the topological properties induced by the syntactic shape of analytic inductive LE-inequalities guarantee their algebraic canonicity in the setting of {\em slanted} LE-algebras of the appropriate signature (cf.~Definition \ref{def:slanted LE}). This canonicity result connects and extends a number of recent canonicity results in very different areas: subordination algebras, quasi-modal algebras, and the transfer of canonicity via G\"odel-McKinsey-Tarski translations.

\paragraph{Slanted LEs as a comprehensive mathematical environment} In this paper, we attributed a name to  a notion (that of {\em slanted operations}, cf.~Definition \ref{def:c-slanted o-slanted}, from which the ensuing notion of {\em slanted algebra} derives) instances of which have  cropped up in the literature in many contexts and with different angles, scopes, and motivations, spanning from the theory of (generalized) canonical extensions of maps \cite{GeJo04} and their adjoints \cite{PaSoZh15}, to de Vries algebras \cite{deVries} and their generalizations (in the equivalent forms of quasi-modal algebras, \cite{Quasimodal}, pre-contact algebras \cite{Precontact} and subordination algebras \cite{BezSouVe}), and the G\"odel-McKinsey-Tarski translation \cite{CoPaZh19}. While the connection with duality-theoretic aspects is very much present in each of these contexts taken separately, the environment of slanted algebras as defined in the present paper makes it possible to provide a purely algebraic, modular and uniform reformulation and generalization of extant results, and explore, as we have started to do, generalized settings, such as the (constructive) `non-distributive' one of the present paper, also paving the way towards their investigation with duality-theoretic and topological techniques on relational structures based e.g.~on polarities and reflexive graphs (cf.~\cite{conradie2020non}). This line of investigation is ongoing.

\paragraph{Equivalence, morphisms and duality} Related to the previous point, the environment of slanted LE lends itself naturally to be investigated with universal algebraic and category-theoretic tools, starting with the definition of slanted homomorphisms as lattice homomorphisms $h:\mathbb{A}_1\to \mathbb{A}_2$ such as the following diagrams commute for every $f\in \mathcal{F}$ and $g\in \mathcal{G}$:

\begin{center}
	\begin{tabular}{rcl c c r c l}
		 $ \bbas_1$ & $\stackrel{h^\delta}{\longrightarrow}$ & $\bbas_2$ && $\bbas_1$ & $\stackrel{h^\delta}{\longrightarrow}$ & $\bbas_2$ \\
		\footnotesize{$f^{\bba_1}$} $\uparrow \> \>$ && $\uparrow$ \footnotesize{$f^{\bba_2}$}&& \footnotesize{$g^{\bba_1}$}$\uparrow\> \>$ && $\uparrow$\footnotesize{$g^{\bba_2}$} \\
		$\bba_1^{\epsilon_f}$ & $\stackrel{h^{\epsilon_f}}{\longrightarrow}$ &$\bba_2^{\epsilon_f}$ && $\bba_1^{\epsilon_g}$ & $\stackrel{h^{\epsilon_g}}{\longrightarrow}$ &$\bba_2^{\epsilon_g}$
	\end{tabular}
\end{center}
This line of investigation is ongoing.

\paragraph{From normal to non-normal settings} Although the best-known examples of applications of the theory of canonical extensions (e.g.~\cite{Jonsson-sahlqvist}) concern logics in which the additional operations are all normal (i.e.~coordinatewise preserving or reversing all finite joins, for $f$-type operations, or meets, for $g$-type operations), the theory itself applies to arbitrary maps \cite{GeJo04}, and has already been applied to develop canonicity, correspondence and proof-theoretic results for non-normal logics in several settings, including the Boolean \cite{Jinsheng}, the distributive \cite{PaSoZh15r} and the general lattice  \cite{CoPa-constructive, CCPZ}. In the present paper, we have addressed slanted canonicity in the setting of {\em normal} slanted LEs, in the sense indicated above (see also the discussion after Definition \ref{def:c-slanted o-slanted}).  A further direction that can be naturally pursued  in this algebraic context concerns the development of (constructive) slanted canonicity results in the context of  {\em non-normal} slanted algebras. This direction invests the study of the notion of {\em weakening relation} \cite{moshier2016relational} as  generalized subordination, and its possible applications in obtaining semantic cut elimination results generalizing those in e.g.~\cite{greco2018algebraic}.

\section{Appendix}\label{appendix}

\subsection{Topological properties of slanted operations and their residuals}

Fix a language $\mathcal{L}_\mathrm{LE}$, and a {\em slanted} $\mathcal{L}_\mathrm{LE}$-algebra  $\bba = (A, \mathcal{F}^\bba, \mathcal{G}^\bba)$ for the remainder of this section. This subsection collects the relevant order-theoretic and  topological properties of the additional operations of $\bba$ and their adjoints, which will be critical for the proof of the ``topological versions'' of the Ackermann lemmas in Section \ref{Sec: topological Ackermann}. These results are the straightforward  generalization to the setting of slanted algebras of properties that are  %essentially reformulations of results which are known in the dual setting of descriptive general frames (cf.\ \cite{ALBA}), or already present in the literature
well known to hold in the setting of normal LEs (e.g.\ \cite[Section 10]{CoPa-nondist}). %or follow easily from other results (viz.\ \cite[Lemma 3.4]{DGP}). \marginnote{ricontrollare questo perche' non credo sia vero}
%and well known to researchers in the area (Mai Gehrke, Sam van Gool).
In what follows, we use the terminology $\partial$-monotone (respectively $\partial$-antitone, $\partial$-positive, $\partial$-negative, $\partial$-open, $\partial$-closed) to mean its opposite, i.e.~antitone (respectively, monotone, negative, positive, closed, open). By $1$-monotone (respectively antitone, positive, negative, open, closed) we simply mean monotone (antitone, positive, negative, open, closed). Also in symbols, for example we will write $(O(\bbas))^1$ for $O(\bbas)$ and $(O(\bbas))^\partial$ for $K(\bbas)$, and similarly $(K(\bbas))^1$ for $K(\bbas)$ and $K(\bbas)^\partial$ for $O(\bbas)$. This convention generalizes to order-types and tuples in the obvious way.  Thus, for example,  $(O(\bbas))^{\epsilon}$ is the cartesian product of sets with $O(\bbas)$ as $i$th coordinate where $\epsilon_i = 1$ and $K(\bbas)$ for $j$th coordinate where $\epsilon_j = \partial$.

\begin{lemma}
\label{cor: open upset to open upset for white box} For all $f\in \mathcal{F}^\bba$, $g\in \mathcal{G}^\bba$, $\overline{k}\in (K(\bbas))^{\epsilon_f}$, and $\overline{o}\in (O(\bbas))^{\epsilon_g}$,  %$c, c_1, c_2 \in \kbbas$ and $o, o_1, o_2 \in \obbas$,
\begin{enumerate}
\item $g(\overline{ o}) \in O(\bbas)$,
\item $f(\overline{ k}) \in K(\bbas)$.
%\item $\rhd c \in \obbas$,
%\item $\lhd o \in \kbbas$,
%\item $o_1\star o_2\in \obbas$, and
%\item $c_1\circ c_2\in \kbbas$.
\end{enumerate}
\end{lemma}
\begin{proof}
These facts straightforwardly follow from  the fact that each $f\in \mathcal{F}^{\bbas}$ (resp.\  $g\in \mathcal{G}^{\bbas}$) is the $\sigma$-extension  (resp.\ $\pi$-extension) of the corresponding operation in $\bba$: for instance, $g(\overline{ o})= g^{\pi}(\overline{ o})  = \bigvee \{ g(\overline{a}) \mid \overline{a} \in \bba^{\epsilon_g} \textrm{ and } \overline{a} \leq^{\epsilon_g} \overline{o} \}$, and  $g(\overline{a}) \in O(\bbas)$ for each $\overline{a} \in \bba^{\epsilon_g}$.
\end{proof}

\begin{remark}\label{rem: counterexample} In the standard setting  in which any $f\in \mathcal{F}^\bba$ and $g\in \mathcal{G}^\bba$ maps tuples of clopen elements to clopen elements, it also holds (cf.~\cite[Lemma 10.2]{CoPa-nondist}) that for all %$f\in \mathcal{F}^\bba$, $g\in \mathcal{G}^\bba$, 
$\overline{k}\in (K(\bbas))^{\epsilon_f}$, and $\overline{o}\in (O(\bbas))^{\epsilon_g}$, %$1 \leq 1 \leq n_f$ and $1 \leq j \leq n_g$,   %$c, c_1, c_2 \in \kbbas$ and $o, o_1, o_2 \in \obbas$,
\begin{enumerate}
\item If $\overline{o} \in (O(\bbas))^{\epsilon_g}$, then $f(\overline{ o}) \in O(\bbas)$,
\item If $\overline{k}\in (K(\bbas))^{\epsilon_f}$, then $g(\overline{ k}) \in K(\bbas)$.
\end{enumerate}
Clearly, these properties do not hold in  the setting of slanted LEs, as, together with Lemma \ref{cor: open upset to open upset for white box} they would imply that any slanted operation maps tuples of clopen elements to clopen elements, which is not true.  For a counterexample, let $A$ be an infinite Boolean algebra and $x_0$ be an atom of $A^\delta$ which is not clopen. Then, we define  the c-slanted operator $\Diamond$ on $A$ defined by the assignment $\Diamond a: = a \vee x_0$  for each $a \in A$. It is clear that $ a \vee x_0$ is not open for every $a$ with $x_0 \not\leq a$, as it would imply that $x_0 = (a \vee x_0) \wedge \neg a$ is open. One can find a counterexample for a $g \in \mathcal{G}^{\bbA}$ in a similar fashion. 

In the standard setting, these properties are used in the proofs of the counterparts  of Lemmas \ref{Syn:Opn:Clsd:Appld:ClsdUp:Lemma} and \ref{Esakia:Syn:Clsd:Opn:Lemma} below (cf.~Lemmas 10.6 and 10.7 of \cite{CoPa-nondist}). However, rather than being formulated in terms of syntactically open and closed formulas, Lemmas \ref{Syn:Opn:Clsd:Appld:ClsdUp:Lemma} and \ref{Esakia:Syn:Clsd:Opn:Lemma} are formulated in terms of the more restricted notions of ssc and sso, which is why their proofs go through nonetheless.
\end{remark}

The proof of the following lemma is verbatim the same as the one of Lemma 10.3 in \cite{CoPa-nondist}, since in that proof, it is only needed that $g(\overline{a})$ is an open element and $f(\overline{a})$ is a closed element. For the sake of self-containdness, we report the proof.
\begin{lemma}\label{Blk:Diam:c:Clsd:Lemma}
For all $f\in \mathcal{F}$, $g\in \mathcal{G}$, $1\leq i\leq n_f$, and $1\leq j\leq n_g$,

\begin{enumerate}
\item If $\epsilon_g(j) = 1$, then $g^\flat_j(\overline{k})\in K(\bbas)$ for every $\overline{k}\in (K(\bbas))^{\epsilon_{g^\flat_j}}$;
\item If $\epsilon_g(j) = \partial$, then $g^\flat_j(\overline{o})\in O(\bbas)$ for every $\overline{o}\in (O(\bbas))^{\epsilon_{g^\flat_j}}$;
\item If $\epsilon_f(i) = 1$, then $f^\sharp_i(\overline{o})\in O(\bbas)$ for every $\overline{o}\in (O(\bbas))^{\epsilon_{f^\sharp_i}}$;
\item If $\epsilon_f(i) = \partial$, then $f^\sharp_i(\overline{k})\in K(\bbas)$ for every $\overline{k}\in (K(\bbas))^{\epsilon_{f^\sharp_i}}$.
%
%\item $\Diamondblack c \in \kbbas$,
%
%\item $\blacksquare o \in \obbas$,
%
%\item ${\blacktriangleleft} o \in \kbbas$,
%
%\item ${\blacktriangleright} c \in \obbas$,
%
%\item $c \circback o \in \obbas$,
%
%\item $o \circfor c \in \obbas$,
%
%\item $o \starback c \in \kbbas$, and
%
%\item $c \starfor o \in \kbbas$.
\end{enumerate}
\end{lemma}
\begin{proof}
1. By denseness, $g^\flat_j(\overline{k}) = \bigwedge \{ o \in O(\bbas) \mid g^\flat_j(\overline{k})  \leq o \}$. Let $Y := \{ o \in O(\bbas) \mid g^\flat_j(\overline{k})  \leq o \}$ and $X : = \{ a \in \bba \mid g^\flat_j(\overline{k})  \leq a \}$.
%$\Diamondblack c = \bigwedge \{ o \in \obbas \mid \Diamondblack c \leq o \}$. Let $Y = \{ o \in \obbas \mid \Diamondblack c \leq o \}$ and $X = \{a \in \bba \mid \Diamondblack c \leq a \}$.
To show that $g^\flat_j(\overline{k}) \in K(\bbas)$, it is enough to show that $\bigwedge X = \bigwedge Y$.

Since clopens are opens, $X \subseteq Y$, so $\bigwedge Y \leq \bigwedge X$. In order to show that $\bigwedge X \leq \bigwedge Y$, it suffices to show that for every $o \in Y$ there exists some $a \in X$ such that $a \leq o$. Let $o \in Y$, i.e., $g^\flat_j(\overline{k})  \leq o$. By residuation, $k_j \leq g(\overline{k}[o/k_j])$, where $\overline{k}[o/k_j]$ denotes the $n_g$-array obtained by replacing the $j$th coordinate of $\overline{k}$ by $o$. Notice that $\overline{k}[o/k_j]\in (O(\bbas))^{\epsilon_g}$. This immediately follows from the fact that by assumption, $\epsilon_{g^\flat_j}(l) = \epsilon_g(l) = 1$ if $l = j$ and $\epsilon_{g^\flat_j}(l) = \epsilon_g^\partial(l)$ if $l\neq j$.

Since $k_j \in K(\bbas)$, and $g(\overline{k}[o/k_j]) = g^\pi(\overline{k}[o/k_j]) = \bigvee \{g(\overline{a})\mid \overline{a}\in \bba^{\epsilon_g} \textrm{ and } \overline{a} \leq^{\epsilon_g} \overline{k}[o/k_j] \}$ and $g(\overline{ a })\in  O(\bbas)$, we may apply compactness and get that
$k_j \leq g(\overline{ a_1 })\vee \cdots \vee g(\overline{ a_n })$ for some $\overline{ a_1 },\ldots, \overline{ a_n }\in \bba^{\epsilon_g}$ s.t.\ $\overline{ a_1 },\ldots, \overline{ a_n } \leq^{\epsilon_g} \overline{k}[o/k_j]$. Let $\overline{ a} = \overline{ a_1 } \vee^{\epsilon_g} \cdots \vee^{\epsilon_g} \overline{ a_n }$. The $\epsilon_g$-monotonicity of $g$ implies that $k_j \leq g(\overline{ a_1 })\vee \cdots \vee g(\overline{ a_n })\leq  g(\overline{ a })$, and hence $g^\flat_j(\overline{a}[k_j/a_j])\leq a_j$. The proof is complete if we show that
$g^\flat_j(\overline{k}) \leq g^\flat_j(\overline{a}[k_j/a_j])$. By the $\epsilon_{g^\flat_j}$-monotonicity of $g^\flat_j$, it is enough to show that $\overline{k} \leq^{\epsilon_{g^\flat_j}}\overline{a}[k_j/a_j]$. Since the two arrays coincide in their $j$th coordinate, we only need to check that this is true for every $l \neq j$. Recall that $\epsilon_{g^\flat_j}(l) = \epsilon_g^\partial(l)$ if $l\neq j$. Hence, the statement immediately follows from this and the fact that, by construction, $\overline{ a }\leq^{\epsilon_g} \overline{k}[o/k_j]$.

2.\ 3.\ and 4.\ are order-variants of 1.
\end{proof}

The proofs of the  following lemmas are verbatim the same as the ones of Lemmas 10.4 and 10.5 in \cite{CoPa-nondist}.

\begin{lemma}\label{lemma:uncongenial for the whites}%\marginnote{no synchronization anymore, to be carefully checked, statement and proof} %\marginnote{proof completely redone. synchronized tuples not needed anymore. please let me know if you agree}
For all $f\in \mathcal{F}$ and $g\in \mathcal{G}$,
%For every up-directed collection $\mathcal{U}\subseteq \obbas$, every down-directed collection  $\mathcal{D}\subseteq \kbbas$,  every synchronized down-directed tuple $(\{d_i \mid i \in I \}, \{c_i \mid i \in I \})$, and every synchronized up-directed tuple $(\{u_i \mid i \in I \}, \{o_i \mid i \in I \})$,
%
\begin{enumerate}
\item $g(\bigvee^{\epsilon_g(1)} \mathcal{U}_1,\ldots, \bigvee^{\epsilon_g(n_g)} \mathcal{U}_{n_g}) = \bigvee \{g(u_1,\ldots,u_{n_g})\ |\ u_j \in \mathcal{U}_j\mbox{ for every } 1\leq j\leq n_g\}$ for every  $n_g$-tuple $(\mathcal{U}_1,\ldots, \mathcal{U}_{n_g})$ such that
    $\mathcal{U}_j\subseteq O(\bbas)^{\epsilon_g(j)}$ and $\mathcal{U}_j$ is $\epsilon_g(j)$-up-directed for each $1\leq j\leq n_g$.
\item $f(\bigwedge^{\epsilon_f(1)} \mathcal{D}_1,\ldots, \bigwedge^{\epsilon_f(n_f)} \mathcal{D}_{n_f}) = \bigwedge \{f(d_1,\ldots,d_{n_f})\ |\ d_j \in \mathcal{D}_j\mbox{ for every } 1\leq j\leq n_f\}$ for every  $n_f$-tuple $(\mathcal{D}_1,\ldots, \mathcal{D}_{n_f})$ such that
    $\mathcal{D}_j\subseteq K(\bbas)^{\epsilon_f(j)}$ and $\mathcal{D}_j$ is $\epsilon_f(j)$-down-directed for each $1\leq j\leq n_f$.
%
%\item $ \Diamond (\bigwedge \mathcal{D}) = \bigwedge \{\Diamond c \ |\ c \in \mathcal{D}\}$,
%
%\item ${\rhd} (\bigwedge \mathcal{D}) = \bigvee \{{\rhd} c \ |\ c \in \mathcal{D}\}$,
%
%\item ${\lhd} (\bigvee \mathcal{U}) = \bigwedge \{{\lhd} o \ |\  o \in \mathcal{U}\}$,
%
%\item $\bigvee_{i \in I} u_i \star \bigvee_{i \in I} o_i = \bigvee \{ u_i \star o_i \ | i \in I\}$, and
%
%\item $\bigwedge_{i \in I} d_i \circ \bigwedge_{i \in I} c_i = \bigwedge \{ d_i \circ c_i \ | i \in I\}$.
\end{enumerate}
\end{lemma}
\begin{proof}
1. The `$\geq$' direction easily follows from the $\epsilon_g$-monotonicity of $g$. Conversely, by denseness it is enough to show that if $c \in \kbbas$ and $c \leq g(\bigvee^{\epsilon_g(1)} \mathcal{U}_1,\ldots, \bigvee^{\epsilon_g(n_g)} \mathcal{U}_{n_g})$, then $c \leq g(u_1,\ldots,u_{n_g})$ for some tuple $(u_1,\ldots,u_{n_g})$ such that $u_j\in \mathcal{U}_j$ for each $1\leq j\leq n_g$. Hence, consider $c \leq g(\bigvee^{\epsilon_g(1)} \mathcal{U}_1,\ldots, \bigvee^{\epsilon_g(n_g)} \mathcal{U}_{n_g})$. Then $g^\flat_1(c, \overline{\bigvee^{\epsilon_g} \mathcal{U}}) \leq^{\epsilon(1)} \bigvee^{\epsilon_g(1)} \mathcal{U}_1$, where, to enhance readability, we suppress sub- and superscripts and write $\overline{\bigvee^{\epsilon_g} \mathcal{U}}$ for $(\bigvee^{\epsilon_g(2)} \mathcal{U}_2, \ldots, \bigvee^{\epsilon_g(n_g)} \mathcal{U}_{n_g})$. If $\epsilon_g(1) = 1$, then  $\epsilon_{g^\flat_1}(1) = 1$ and $\epsilon_{g^\flat_1}(l) = \epsilon_{g}^\partial(l)$ for every $2\leq l\leq n_g$. Hence $\mathcal{U}_l\subseteq O(\bbas)^{\epsilon_g(l)} = K(\bbas)^{\epsilon_{g^\flat_1}(l)}$, hence $\bigvee^{\epsilon_g(l)} \mathcal{U}_l = \bigwedge^{\epsilon_{g^\flat_1}(l)}\mathcal{U}_l\in K(\bbas)^{\epsilon_{g^\flat_1}(l)}$ for every $2\leq l\leq n_g$.
By Lemma \ref{Blk:Diam:c:Clsd:Lemma}(1), this implies that $g^\flat_1(c, \overline{\bigvee^{\epsilon_g} \mathcal{U}})\in K(\bbas)$. Hence, by  compactness,  $g^\flat_1(c, \overline{\bigvee^{\epsilon_g} \mathcal{U}})\leq \bigvee_{i=1}^n o_i$ for some $o_1,\ldots,o_n\in \mathcal{U}_1$. Since $\mathcal{U}_1$ is up-directed, $\bigvee_{i=1}^n o_i \leq u_1$ for some $u_1 \in \mathcal{U}_1$. Hence $c \leq g(u_1, \overline{\bigvee^{\epsilon_g} \mathcal{U}})$. The same conclusion can be reached via a similar argument if $\epsilon_g(1) = \partial$. Therefore,  $g^\flat_2(u_1, c, \overline{\bigvee^{\epsilon_g} \mathcal{U}}) \leq^{\epsilon_g(2)} \bigvee^{\epsilon_g(2)} \mathcal{U}_2$, where $\overline{\bigvee^{\epsilon_g} \mathcal{U}}$ now stands for $(\bigvee^{\epsilon_g(3)} \mathcal{U}_3, \ldots, \bigvee^{\epsilon_g(n_g)} \mathcal{U}_{n_g})$.  By applying the same reasoning, we can conclude that $c \leq g(u_1, u_2, \overline{\bigvee^{\epsilon_g} \mathcal{U}})$ for some $u_2\in \mathcal{U}_2$, and so on. Hence, we can then construct a sequence $u_j\in \mathcal{U}_j$ for $1\leq j\leq n_g$ such that $c\leq g(u_1,\ldots u_{n_g})$, as required.

%\leq \bigvee \{\Box o \ |\ o \in \mathcal{U}\}$. %The `$\geq$' direction easily follows from the monotonicity of $\Box$. Conversely, by denseness it is enough to show that if $c \in \kbbas$ and $c \leq \Box (\bigvee \mathcal{U})$, then $c \leq \Box o$ for some $o \in \mathcal{U}$. Indeed, if $c \leq \Box (\bigvee \mathcal{U})$ then $\Diamondblack c \leq \bigvee \mathcal{U}$. Since, by Lemma \ref{Blk:Diam:c:Clsd:Lemma}(1), $\Diamondblack c \in \kbbas$, it follows from compactness that $\Diamondblack c \leq \bigvee_{i=1}^n o_i$ for some $o_1,\ldots,o_n\in \mathcal{U}$. Since $\mathcal{U}$ is up-directed, $\bigvee_{i=1}^n o_i \leq o$ for some $o \in \mathcal{U}$. Hence $c \leq \Box o \leq \bigvee \{\Box o \ |\ o \in \mathcal{U}\}$.

2.\ is order-dual to 1.
%
%5. is an order-variant of 6.
%\marginnote{I leave  proof 6.\ to facilitate comparison}
%6. The `$\leq$' direction easily follows from the monotonicity of $\circ$. Conversely, by denseness it is enough to show that if $o \in \obbas$ and $\bigwedge_{i \in I} d_i \circ \bigwedge_{i \in I} c_i\leq o$, then $d_j\circ c_j\leq o$ for some $j \in I$. Indeed, if $\bigwedge_{i \in I} d_i \circ \bigwedge_{i \in I} c_i\leq o$ then $\bigwedge_{i \in I} c_i\leq \bigwedge_{i \in I} d_i\circback o$. Since, by Lemma \ref{Blk:Diam:c:Clsd:Lemma}(5), $\bigwedge_{i \in I} d_i\circback o \in \obbas$, it follows from compactness that $\bigwedge_{i \in I_1} c_i\leq \bigwedge_{i \in I_2} d_i\circback o$ for some finite $I_1, I_2\subseteq I$.\marginnote{I think this is wrong, since $\circback$ does not reverse meets in first coordinate} Hence, again by (anti)monotonicity, $\bigwedge_{i \in J} c_i\leq \bigwedge_{i \in J} d_i \circback o$, where $J = I_1\cup I_2$ is still a finite set of indices. By the assumption on the tuple $(\{d_i \mid i \in I \}, \{c_i \mid i \in I \})$, $c_j\leq d_j\circback o$ for some  $j\in I$, hence $d_j\circ c_j\leq o$.
\end{proof}

\begin{lemma}%\marginnote{no synchronization anymore, to be carefully checked, statement and proof}
\label{lemma: uncongenial for the blacks} For all $f\in \mathcal{F}$, $g\in \mathcal{G}$, $1\leq i\leq n_f$, and $1\leq j\leq n_g$,
\begin{enumerate}
\item If $\epsilon_g(j) = 1$, then \[ g^\flat_j(\bigwedge{}^{\epsilon_{g^\flat_j}(1)} \mathcal{D}_1,\ldots, \bigwedge{}^{\epsilon_{g^\flat_j}(n_g)} \mathcal{D}_{n_g}) = \bigwedge \{g^\flat_j(d_1,\ldots,d_{n_g})\ |\ d_h \in \mathcal{D}_h\mbox{ for every } 1\leq h\leq n_g\}\] for every  $n_g$-tuple $(\mathcal{D}_1,\ldots, \mathcal{D}_{n_g})$ such that
    $\mathcal{D}_h\subseteq K(\bbas)^{\epsilon_{g^\flat_j}(h)}$ and $\mathcal{D}_h$ is $\epsilon_{g^\flat_j}(h)$-down-directed for each $1\leq h\leq n_g$.

    \item If $\epsilon_g(j) = \partial$, then \[ g^\flat_j(\bigvee{}^{\epsilon_{g^\flat_j}(1)} \mathcal{U}_1,\ldots, \bigvee{}^{\epsilon_{g^\flat_j}(n_g)} \mathcal{U}_{n_g}) = \bigvee \{g^\flat_j(u_1,\ldots,u_{n_g})\ |\ u_h \in \mathcal{U}_h\mbox{ for every } 1\leq h\leq n_g\} \] for every  $n_g$-tuple $(\mathcal{U}_1,\ldots, \mathcal{U}_{n_g})$ such that
    $\mathcal{U}_h\subseteq O(\bbas)^{\epsilon_{g^\flat_j}(h)}$ and $\mathcal{U}_h$ is $\epsilon_{g^\flat_j}(h)$-up-directed for each $1\leq h\leq n_g$.

\item If $\epsilon_f(i) = 1$, then \[ f^\sharp_i(\bigvee{}^{\epsilon_{f^\sharp_i}(1)} \mathcal{U}_1,\ldots, \bigvee{}^{\epsilon_{f^\sharp_i}(n_f)} \mathcal{U}_{n_f}) = \bigvee \{f^\sharp_i(u_1,\ldots,u_{n_f})\ |\ u_h \in \mathcal{U}_h\mbox{ for every } 1\leq h\leq n_f\}\] for every  $n_f$-tuple $(\mathcal{U}_1,\ldots, \mathcal{U}_{n_f})$ such that
    $\mathcal{U}_h\subseteq O(\bbas)^{\epsilon_{f^\sharp_i}(h)}$ and $\mathcal{U}_h$ is $\epsilon_{f^\sharp_i}(h)$-up-directed for each $1\leq h\leq n_f$.
\item If $\epsilon_f(i) = \partial$, then \[ f^\sharp_i(\bigwedge{}^{\epsilon_{f^\sharp_i}(1)} \mathcal{D}_1,\ldots, \bigwedge{}^{\epsilon_{f^\sharp_i}(n_f)} \mathcal{D}_{n_f}) = \bigwedge \{f^\sharp_i(d_1,\ldots,d_{n_f})\ |\ d_h \in \mathcal{D}_h\mbox{ for every } 1\leq h\leq n_f\} \] for every  $n_f$-tuple $(\mathcal{D}_1,\ldots, \mathcal{D}_{n_f})$ such that
    $\mathcal{D}_h\subseteq K(\bbas)^{\epsilon_{f^\sharp_i}(h)}$ and $\mathcal{D}_h$ is $\epsilon_{f^\sharp_i}(h)$-down-directed for each $1\leq h\leq n_f$.
%
%
%\item $\blacksquare(\bigvee \mathcal{U}) = \bigvee \{\blacksquare o\ |\ o\in \mathcal{U}\}$.
%
%\item $\Diamondblack(\bigwedge \mathcal{D}) = \bigwedge \{\Diamondblack c \ |\ c\in \mathcal{D}\}$.
%
%\item ${\blacktriangleright}(\bigwedge \mathcal{D}) = \bigvee \{{\blacktriangleright} c \ |\ c \in \mathcal{D}\}$.
%
%\item ${\blacktriangleleft}(\bigvee \mathcal{U}) = \bigwedge \{{\blacktriangleleft} o \ |\  \in \mathcal{U}\}$.
%
%\item $\bigwedge_{i \in I} c_i \starfor \bigvee_{i \in I} o_i = \bigwedge \{c_i \starfor o_i \mid i \in I \}$.
%
%\item $\bigvee_{i \in I} o_i \starback \bigwedge_{i \in I} c_i = \bigwedge \{o_i \starback c_i \mid i \in I \}$.
%
%\item $\bigwedge_{i \in I} c_i \circback \bigvee_{i \in I} o_i = \bigvee \{c_i \circback o_i \mid i \in I \}$.
%
%\item $\bigvee_{i \in I} o_i \circfor \bigwedge_{i \in I} c_i = \bigvee \{o_i \circfor c_i \mid i \in I \}$.
%
\end{enumerate}
\end{lemma}
\begin{proof}
3.\ The `$\geq$' direction easily follows from the
$\epsilon_{f^\sharp_i}$-monotonicity of $f^\sharp_i$. For the converse inequality, by denseness it is enough to show that if we have
$c\leq f^\sharp_i(\bigvee^{\epsilon_{f^\sharp_i}(1)} \mathcal{U}_1,\ldots, \bigvee^{\epsilon_{f^\sharp_i}(n_f)} \mathcal{U}_{n_f})$ for a closed element $c$, then $c\leq f^\sharp_i(u_1,\ldots,u_{n_f})$ for some tuple $(u_1,\ldots,u_{n_f})$ such that $u_h \in \mathcal{U}_h$ for every $1\leq h\leq n_f$. By residuation,
$c\leq f^\sharp_i(\bigvee^{\epsilon_{f^\sharp_i}(1)} \mathcal{U}_1,\ldots, \bigvee^{\epsilon_{f^\sharp_i}(n_f)} \mathcal{U}_{n_f})$ implies that we have the inequality $f(
\bigvee^{\epsilon_{f^\sharp_i}(1)} \mathcal{U}_1,\ldots,c,\ldots, \bigvee^{\epsilon_{f^\sharp_i}(n_f)} \mathcal{U}_{n_f})\leq \bigvee^{\epsilon_{f^\sharp_i}(i)} \mathcal{U}_i$.
The assumption $\epsilon_f(i) = 1$ implies that  $\epsilon_{f^\sharp_i}(i) = 1$ and $\epsilon_{f^\sharp_i}(l) = \epsilon_{f}^\partial(l)$ for every $l\neq i$. Hence $\mathcal{U}_l\subseteq O(\bbas)^{\epsilon_{f^\sharp_i}(l)} = K(\bbas)^{\epsilon_f(l)}$, and  $\mathcal{U}_l$ is $\epsilon_f(l)$-down-directed %$\bigvee^{\epsilon_{f^\sharp_i}(l)} \mathcal{U}_l = \bigwedge^{\epsilon_{f}(l)}\mathcal{U}_l\in \kbbas^{\epsilon_f(l)}$
for every $l\neq i$. Recalling that $\bigvee^{\epsilon_{f^\sharp_i}(l)}$ coincides with $\bigwedge^{\epsilon_{f}(l)}$,
we can apply  Lemma \ref{lemma:uncongenial for the whites}(2) and get:
%\begin{center}
%\begin{tabular}{cl}
%& $
\[f(\bigvee{}^{\epsilon_{f^\sharp_i}(1)} \mathcal{U}_1,\ldots,c,\ldots, \bigvee{}^{\epsilon_{f^\sharp_i}(n_f)} \mathcal{U}_{n_f}) =
%$=$ & $
\bigwedge \{f(u_1,\ldots,c,\ldots, u_{n_f})\mid u_l\in \mathcal{U}_l \mbox{ for every } l\neq i\}.\]
%\end{tabular}
%\end{center}
Hence, by compactness, %(recall that $\epsilon_f(i)  = \epsilon_{f^\sharp_i}(i) = 1$ by assumption),
$f(\bigvee^{\epsilon_{f^\sharp_i}(1)} \mathcal{U}_1,\ldots,c,\ldots, \bigvee^{\epsilon_{f^\sharp_i}(n_f)} \mathcal{U}_{n_f})\leq \bigvee^{\epsilon_{f^\sharp_i}(i)} \mathcal{U}_i$ implies that
\[\bigwedge_{1\leq j\leq m}\{f(o_1^{(j)},\ldots,c,\ldots, o_{n_f}^{(j)})\mid o_l^{(j)}\in \mathcal{U}_l\mbox{ for all } l\neq i \}\leq o_i^{(1)}\vee\cdots \vee o_i^{(n)}\]
for some $o_i^{(1)},\ldots, o_i^{(n)}\in \mathcal{U}_i$.
The assumptions that $\epsilon_f(i) = 1$ and that each $\mathcal{U}_h$ is $\epsilon_{f^\sharp_i}(h)$-up-directed for every $1\leq h\leq n_f$ imply that $\mathcal{U}_i$ is up-directed and $\mathcal{U}_l$ is $\epsilon_f(l)$-down-directed for each $l\neq i$. Hence, some $u_1,\ldots, u_{n_f}$ exist such that  $u_l\leq^{\epsilon_f(l)}\bigwedge^{\epsilon_f(l)}_{1\leq j\leq m}o_l^{(j)}$  and $o_i^{(1)}\vee\cdots \vee o_i^{(n)}\leq u_i.$
The $\epsilon_f$-monotonicity of $f$ implies  the following chain of inequalities:
\begin{center}
\begin{tabular}{r c l}
$f(u_1,\ldots,c, \ldots, u_{n_f})$ & $\leq$ &$ f(\bigwedge^{\epsilon_f(1)}_{1\leq j\leq m}o_1^{(j)},\ldots,c, \ldots, \bigwedge^{\epsilon_f(n_f)}_{1\leq j\leq m}o_{n_f}^{(j)})$\\
& $\leq$ &$ \bigwedge_{1\leq j\leq m}\{f(o_1^{(j)},\ldots,c,\ldots, o_{n_f}^{(j)})\mid o_l^{(j)}\in \mathcal{U}_l\mbox{ for all } l\neq i \}$\\
& $\leq$ &$ o_i^{(1)}\vee\cdots \vee o_i^{(n)}$\\
& $\leq$ &$ u_i$,\\
\end{tabular}
\end{center}
which implies that $c\leq f^\sharp_i(u_1,\ldots,u_{n_f})$, as required.

1.\ 2.\ and 4.\ are order-variants of 3.
%
%5. The `$\leq$' direction easily follows from the fact that  $\starfor$ is monotone in the first coordinate and antimonotone in the second. For the converse inequality, by denseness it is enough to show that if $o\in \obbas$ and $\bigwedge_{i \in I} c_i \starfor \bigvee_{i \in I} o_i \leq o$, then $c_j\starfor o_j\leq o$ for some $j\in I$. By residuation, $\bigwedge_{i \in I} c_i \starfor \bigvee_{i \in I} o_i \leq o$ implies that $\bigwedge_{i \in I} c_i\leq o\star \bigvee_{i \in I} o_i$. Since $o\star \bigvee_{i \in I} o_i = o\star^{\pi} \bigvee_{i \in I} o_i = \bigvee \{ a\star a'\mid a\leq o \mbox{ and } a'\leq \bigvee_{i \in I} o_i\}\in \obbas$, by compactness $\bigwedge_{i \in I_1} c_i\leq o\star \bigvee_{j \in I_2} o_j$ for some finite $I_1, I_2\subseteq I$. Hence, again by monotonicity, $\bigwedge_{i \in J} c_i\leq o\star \bigvee_{i \in J} o_i$, where $J = I_1\cup I_2$ is still a finite set of indices. By the assumption on the tuple $(\{c_i \mid i \in I \}, \{o_i \mid i \in I \})$, $c_j\leq o\star o_j$ for some  $j\in I$, hence $c_j\starfor o_j\leq o$.
%
%6.\ 7.\ and 8.\ are order-variants of 5.
\end{proof}

\subsection{Proof of the restricted Ackermann lemmas (lemmas \ref{Ackermann:Dscrptv:Right:Lemma} and \ref{Ackermann:Dscrptv:Left:Lemma})}
\label{Sec: topological Ackermann}

For any $\mathcal{L}_{\mathrm{LE}}^+$-formula $\phi$, any  slanted $\mathcal{L}_{\mathrm{LE}}$-algebra $\bba$  and assignment $V$ on $\bbas$, we write $\phi(V)$ to denote the extension of $\phi$ in $\bbas$ under the assignment $V$. We remind the reader that, even when $\varphi$ is in the basic signature and $V$ is an admissible valuation, $\varphi(V)$ may fail to be an element of $\bba$ (cf.~Remark \ref{rem: counterexample} for a counterexample).

Let $p$ be a propositional variable occurring in $\phi$ and $V$ be any assignment. For any $x \in \bbas$, let $V[p:= x]$ be the assignment which is identical to $V$ except that it assigns $x$ to $p$. Then $x\mapsto \phi(V[p:= x])$  defines an operation on $\bbas$, which we will denote $\phi^{V}_{p}(x)$.

%Fix an LMA $\mathbb{A} = (A, \Diamond, \Box, \lhd, \rhd, \circ, \star)$, and  $\mathbb{A}^{\sigma}$ be its canonical extension.
\medskip

The proofs of the following two lemmas are more streamlined versions of those of Lemmas 10.6 and 10.7 of \cite{CoPa-nondist}. The modifications concern the differences between the notions of syntactically closed and open formulas (see Definition \ref{Syn:Opn:Clsd:Definition}) and ssc and sso (see Definition \ref{strictlySyn:Opn:Clsd:Definition}).

\begin{lemma}\label{Syn:Opn:Clsd:Appld:ClsdUp:Lemma}
Let $\phi$ be ssc and $\psi$ sso. Let $V$ be an admissible assignment, $c \in K(\bbas)$  and $o \in O(\bbas)$.
\begin{enumerate}
\item
    \begin{enumerate}
    \item If $\phi(p)$ is positive in $p$, then $\phi^{V}_{p}(c) \in K(\bbas)$, and
    \item if $\psi(p)$ is negative in $p$, then $\psi^{V}_{p}(c) \in O(\bbas)$.
    \end{enumerate}
\item
    \begin{enumerate}
    \item If $\phi(p)$ is negative in $p$, then $\phi^{V}_{p}(o) \in K(\bbas)$, and
    \item if $\psi(p)$ is positive in $p$, then $\psi^{V}_{p}(o) \in O(\bbas)$.
    \end{enumerate}
\end{enumerate}
\end{lemma}
\begin{proof}
We prove 1. by simultaneous induction on $\phi$ and $\psi$. Assume that $\phi(p)$ is positive in $p$ and  $\psi(p)$ is negative in $p$. The base cases of the induction are those when $\phi$ is of the form $\top$, $\bot$, $p$, $q$ (for propositional variables $q$ different from $p$) or $\nomi$, and $\psi$ is of the form $\top$, $\bot$, $q$ (for propositional variables $q$ different from $p$), or $\cnomm$ (note that $\phi$ cannot be a co-nominal $\cnomm$, since it is syntactically closed. Also, $\psi$ cannot be $p$ or a nominal $\nomi$, since $\psi$ is negative in $p$ and is syntactically open, respectively). These cases follow by noting (1) that $V[p:= c](\bot) = 0 \in \mathbb{A}$, $V[p:= c](\top) = 1 \in \mathbb{A}$, and $V[p:= c](q) = V(q) \in \mathbb{A}$, (2) that $V[p:= c](p) = c \in K(\bbas)$ and $V[p:= c](\nomi) \in J^{\infty}(\bbas) \subset K(\bbas)$, and (3) that $V[p:= c](\cnomm) \in M^{\infty}(\bbas) \subset O(\bbas)$ (see discussion on page \pageref{Page:JIr:Clsd:MIr:Opn}).

For the remainder of the proof we will not need to refer to the valuation $V$ and will hence omit reference to it. We will accordingly write $\phi$ and $\psi$ for $\phi^V_p$ and $\psi^V_p$, respectively.

In the cases $\phi(p) = f^\ast(\overline{\phi'(p)}, \overline{
\psi'(p)})$ for $f^\ast\in \mathcal{F}^\ast$, $\phi(p) = \phi_1(p) \wedge \phi_2(p)$, or $\phi(p) = \phi_1(p) \vee \phi_2(p)$ both $\phi_1(p)$ and $\phi_2(p)$ are ssc and positive in $p$, and each $\phi'_i(p)$ in $\overline{\phi'(p)}$ is ssc
and positive in $p$, and each $\psi'_i(p)$ in $\overline{\psi'(p)}$ is sso
and negative in $p$. Hence, the claim follows by the inductive hypothesis, and Lemma \ref{cor: open upset to open upset for white box}(2) if $f^\ast \in \mathcal{F}$, Lemma \ref{Blk:Diam:c:Clsd:Lemma} if $f^\ast \in \mathcal{F}^\ast \setminus \mathcal{F}$ and the fact that meets and finite joins of closed elements are closed, respectively.

Similarly, if $\psi(p) = g^\ast(\overline{\psi'(p)}, \overline{\phi'(p)})$ for $g^\ast\in \mathcal{G}^\ast$, $\psi(p) = \psi_1(p) \vee \psi_2(p)$ or $\psi(p) = \psi_1(p) \wedge \psi_2(p)$, then both $\psi_1(p)$ and $\psi_2(p)$ are sso and negative in $p$, and each $\phi'_i(p)$ in $\overline{\phi'(p)}$ is ssc
and positive in $p$, and each $\psi'_i(p)$ in $\overline{\psi'(p)}$ is sso
and negative in $p$. Hence,  the claim follows by the inductive hypothesis, and Lemma \ref{cor: open upset to open upset for white box}(1) if $g^\ast \in \mathcal{G}$, Lemma \ref{Blk:Diam:c:Clsd:Lemma} if $g^\ast \in \mathcal{G}^\ast \setminus \mathcal{G}$ and the fact that joins and finite meets of open elements are open, respectively.

Item (2) can similarly be proved by simultaneous induction on negative $\phi$ and positive $\psi$. %In fact, the induction is almost verbatim the same, except that
%
%\begin{itemize}
%\item the base cases are those when $\phi$ is of the form $\top$, $\bot$, $q$ (for propositional variables $q$ different from $p$) or $\nomi$, and when $\psi$ is of the form $\top$, $\bot$, $p$, $q$ (for propositional variables $q$ different from $p$), or $\cnomm$, and
%\item $c$ is uniformly replaced with $o$.
%\end{itemize}
\end{proof}
%\marginnote{modificare la dim: editer l'écriture et rassembler les cas $\mathcal{F}^\ast$}
\begin{lemma}\label{Esakia:Syn:Clsd:Opn:Lemma} Let $\phi(p)$ be ssc, $\psi(p)$ sso, $V$ an admissible assignment, $\mathcal{D} \subseteq K(\bbas)$ be down-directed, and $\mathcal{U} \subseteq O(\bbas)$ be up-directed.
\begin{enumerate}
\item
    \begin{enumerate}
    \item If $\phi(p)$ is positive in $p$, then $\phi^V_p(\bigwedge\mathcal{D}) = \bigwedge\{ \phi^V_p (d)\mid d\in \mathcal{D}\}$, and
    \item if $\psi(p)$ is negative in $p$, then $\psi^V_p(\bigwedge\mathcal{D}) = \bigvee\{\psi^V_p (d)\mid d\in \mathcal{D}\}$.
    \end{enumerate}
\item
    \begin{enumerate}
    \item If $\phi(p)$ is negative in $p$, then $\phi^V_p(\bigvee\mathcal{U}) = \bigwedge\{\phi^V_p (u)\mid u\in \mathcal{U}\}$, and
    \item if $\psi(p)$ is positive in $p$, then $\psi^V_p(\bigvee\mathcal{U}) = \bigvee\{ \psi^V_p (u)\mid u\in \mathcal{U}\}$.
    \end{enumerate}
\end{enumerate}
\end{lemma}

\begin{proof}
We prove (1) by simultaneous induction on $\phi$ and $\psi$. The base cases of the induction on $\phi$ are those when it is of the form $\top$, $\bot$, $p$, a propositional variable $q$ other than $p$, or $\nomi$, and for $\psi$ those when it is of the form $\top$, $\bot$, a propositional variable $q$ other than $p$ or $\cnomm$. In each of these cases the claim is trivial.

For the remainder of the proof we will omit reference to the assignment $V$, and simply write $\phi$ and $\psi$ for $\phi^V_p$ and $\psi^V_p$, respectively.

In the cases in which $\phi(p) = \phi_1(p) \vee \phi_2(p)$, $\phi(p) = \phi_1(p) \wedge \phi_2(p)$, $\phi(p) = f^\ast(\overline{\phi'(p)},\overline{\psi'(p)})$, $\psi(p) = \psi_1(p) \wedge \psi_2(p)$, $\psi(p) = \psi_1(p) \vee \psi_2(p)$, $\psi(p) = g^\ast(\overline{\psi'(p)},\overline{\phi'(p)})$,  we have that $\phi_1$ and $\phi_2$
are ssc and positive in $p$ and $\psi_1$ and $\psi_2$
are sso and negative in $p$, and moreover, each $\phi'_i(p)$ in $\overline{\phi'(p)}$ is ssc and positive in $p$, and each $\psi'_j(p)$ in $\overline{\psi'(p)}$ is sso and negative in $p$.

Hence, when $\phi(p) = \phi_1(p) \wedge \phi_2(p)$ and  $\psi(p) = \psi_1(p) \vee \psi_2(p)$, the claim follows by the inductive hypothesis and the associativity of, respectively, meet and join.

If $\phi(p) = \phi_1(p) \vee \phi_2(p)$, then
\begin{center}
\begin{tabular}{r c ll}
$\phi(\bigwedge \mathcal{D})$ &$ = $& $\phi_1(\bigwedge \mathcal{D})\vee \phi_2(\bigwedge \mathcal{D})$ & \\
&$ = $& $\bigwedge \{\phi_1(c_i)\mid c_i\in \mathcal{D}\}\vee \bigwedge \{\phi_2(c_i)\mid c_i\in \mathcal{D}\}$ & (induction hypothesis)\\
&$ = $& $\bigwedge \{\phi_1(c_i)\vee \phi_2(c_j)\mid c_i, c_j\in \mathcal{D}\}$ & ($\ast$)\\
&$ = $& $\bigwedge \{\phi_1(c)\vee \phi_2(c)\mid c\in \mathcal{D}\}$ & ($\phi$ monotone and $\mathcal{D}$ down-directed)\\
&$ = $& $\bigwedge \{\phi(c)\mid c\in \mathcal{D}\}$, & \\
\end{tabular}
\end{center}
where the equality marked with ($\ast$) follows from a restricted form of distributivity enjoyed by canonical extensions of general bounded lattices (cf.\ \cite[Lemma 3.2]{GH01}), applied to the family $\{A_1, A_2\}$ such that $A_i: = \{\phi_i(c_j)\mid c_j\in \mathcal{D}\}$ for $i\in \{1, 2\}$. Specifically, the monotonicity in $p$ of $\phi_i(p)$ and $\mathcal{D}$ being down-directed imply that $A_1$ and $A_2$ are down-directed subsets, which justifies the application of \cite[Lemma 3.2]{GH01}.

 If $\phi(p) = f^\ast(\overline{\phi'(p)},\overline{\psi'(p)})$, with $f^\ast \in \mathcal{F}^\ast$, then
 
 \begin{center} $\phi(\wedge \mathcal{D}) = f^\ast(\overline{\phi'(\wedge \mathcal{D})}, \overline{\psi'(\wedge \mathcal{D}})) = f^\ast(\bigwedge_{d \in \mathcal{D}}\phi'_1(d), \ldots,\bigwedge_{d \in \mathcal{D}}\phi'_k(d), \bigvee_{d \in \mathcal{D}} \psi'_{k+1}(d),\ldots,\bigvee_{d \in \mathcal{D}} \psi'_{n_{f^\ast}}(d) ). $
 \end{center}
%\begin{center}
%$\phi(\bigwedge\mathcal{D}) = f^\ast(\phi_1(\bigwedge\mathcal{D}),\ldots,\phi_{n_f}(\bigwedge\mathcal{D}))
%%
%= f^\ast(\bigwedge_{d \in \mathcal{D}}^{\epsilon_{f^\ast}(1)} \phi_1(d),\ldots, \bigwedge_{d \in \mathcal{D}}^{\epsilon_{f^\ast}(n_{f^\ast})} \phi_{n_{f^\ast}}(d))$.
%\end{center}
%
%&= &\bigwedge \{ \phi_1(c_{i_1}) \star \phi_2(c_{i_2}) \mid i_1, i_2  \in I \}\\
%
%&= &\bigwedge \{ \phi_1(c_{i}) \star \phi_2(c_{i}) \mid i \in I \}.
%
The second equality above holds by the inductive hypothesis. To finish the proof, we need to show that
\begin{center}
$f^\ast(\bigwedge_{d \in \mathcal{D}}\phi'_1(d), \ldots,\bigvee_{d \in \mathcal{D}} \psi'_{n_{f^\ast}}(d) ) = \bigwedge_{d \in \mathcal{D}} f^\ast( \overline{\phi'(d)},\overline{\psi'(d)}).$
\end{center}
The `$\leq$' direction immediately follows from the monotonicity of $f^\ast$. For the converse inequality, by denseness, it is enough to show that if $o\in O(\bbas)$ and $f^\ast(\bigwedge_{d \in \mathcal{D}}\phi'_1(d), \ldots,\bigvee_{d \in \mathcal{D}} \psi'_{f^\ast}(d) ) \leq o$, then $ \bigwedge_{d \in \mathcal{D}} f^\ast( \overline{\phi'(d)},\overline{\psi'(d)}) \leq o$. By Lemmas \ref{Syn:Opn:Clsd:Appld:ClsdUp:Lemma}(1) and \ref{lemma:uncongenial for the whites} or \ref{lemma: uncongenial for the blacks} according to whether $f^\ast \in \mathcal{F}$ or $f^\ast \in \mathcal{F}^\ast \setminus \mathcal{F}$, we have:

\begin{center}
$f^\ast(\bigwedge_{d \in \mathcal{D}}\phi'_1(d), \ldots,\bigvee_{d \in \mathcal{D}} \psi'_{n_{f^\ast}}(d) )= \bigwedge\{f^\ast( \phi'_1(d_1),\ldots,\psi'_{n_{f^\ast}}(d_{n_{f^\ast}}))\mid d_h\in \mathcal{D} \mbox{ for every } 1\leq h\leq n_{f^\ast}\}$.
\end{center}
By compactness (which can be applied by Lemmas \ref{cor: open upset to open upset for white box} or \ref{Blk:Diam:c:Clsd:Lemma}, again according to the nature of $f^\ast$,  and \ref{Syn:Opn:Clsd:Appld:ClsdUp:Lemma}(1)),
\begin{center}
$\bigwedge\{f^\ast( \phi'_1(d^{(i)}_1),\ldots,\psi'_{n_{f^\ast}}(d^{(i)}_{n_{f^\ast}}))\mid 1\leq i\leq n\}\leq o$.
\end{center}
Let $\mathcal{D}': = \{d^{(i)}_h\mid 1\leq i\leq n\mbox{ and } 1\leq h\leq n_{f^\ast}\}$. Since $\mathcal{D}$ is down-directed, $d^\ast\leq \bigwedge \mathcal{D}'$ for some $d^\ast\in \mathcal{D}$. Then, by monotonicity of $f^\ast$ (Recall the notations at the beginning of \ref{Subsec:PrelimDef}) and since each $\phi'_i(p)$ is positive in $p$ and each $\psi'_j(p)$ is negative in $p$, the following chain of inequalities holds
\begin{center}
\begin{tabular}{r c l}
$\bigwedge_{d \in \mathcal{D}} f^\ast( \overline{\phi'(d)},\overline{\psi'(d)})$ & $\leq$ &
$f^\ast(\overline{\phi'(d^\ast)},\overline{\psi'(d^\ast)})$\\
& $\leq$ & $ f^\ast(\phi'_1(\bigwedge_{1\leq i\leq n}d^{(i)}_1),\ldots, \psi'_{n_{f^\ast}}(\bigwedge_{1\leq i\leq n}d^{(i)}_{n_{f^\ast}}))$\\
& $\leq$ & $ \bigwedge\{f^\ast( \phi'_1(d^{(i)}_1)\ldots,\psi'_{n_{f^\ast}}(d^{(i)}_{n_{f^\ast}}))\mid 1\leq i\leq n\}$\\
& $\leq$ & $o.$
\end{tabular}
\end{center}

The remaining cases are similar, and left to the reader.

Thus the proof of item (1) is concluded. Item (2) can be proved similarly by simultaneous induction on $\phi$ negative in $p$ and $\psi$ positive in $p$.
\end{proof}

\paragraph{Proof of the Righthanded Ackermann lemma for admissible assignments (Lemma \ref{Ackermann:Dscrptv:Right:Lemma})}
To keep the notation uncluttered, we will simply write $\beta_{i}$ and $\gamma_{i}$ for ${\beta_i}^V_p$ and ${\gamma_i}^V_p$, respectively. The implication from bottom-to-top follows by the monotonicity of the $\beta_i$ and the antitonicity of the $\gamma_i$ in $p$. Indeed, if $\alpha(V) \leq u$, then, for each $1 \leq i \leq n$, ${\beta_i}(\alpha(V)) \leq {\beta_i}(u) \leq {\gamma_i}(u) \leq {\gamma_i}(\alpha(V))$.

For the sake of the converse implication assume that ${\beta_i}(\alpha(V)) \leq {\gamma_i}(\alpha(V))$ for all $1 \leq i \leq n$. By lemma \ref{Syn:Opn:Clsd:Appld:ClsdUp:Lemma}, $\alpha(V)\in \kbbas$. Hence $\alpha(V) = \bigwedge \{a \in \bba \mid \alpha(V) \leq a \}$, making it the meet of a down-directed subset of $\kbbas$.  Thus, for any $1 \leq i \leq n$, we have
\[
\beta_i(\bigwedge \{a \in \bba \mid \alpha(V) \leq a \}) \leq \gamma_i(\bigwedge \{a \in \bba \mid \alpha(V) \leq a \}).
\]
Since $\gamma_i$ is syntactically open and negative in $p$, and $\beta_i$ is syntactically closed and positive in $p$, we may apply lemma \ref{Esakia:Syn:Clsd:Opn:Lemma} and equivalently obtain
\[
\bigwedge \{\beta_i(a)  \mid a \in \bba, \: \alpha(V) \leq a \} \leq \bigvee \{\gamma_i(a) \mid a \in \bba, \: \alpha(V) \leq a \}.
\]
By lemma \ref{Syn:Opn:Clsd:Appld:ClsdUp:Lemma}, $ \beta_i(a)\in K(\bbas)$ and  $\gamma_i(a)\in O(\bbas)$ for each $a \in \bba$. Hence by compactness,
\[
\beta_i(b_1) \wedge \cdots \beta_i(b_k) \leq \gamma_i(a_1) \vee \cdots \vee \gamma_i(a_m).
\]
for some $a_1, \ldots, a_m, b_1, \ldots b_k \in \bba$ with $\alpha(V) \leq a_j$, $1 \leq j \leq m$, and $\alpha(V) \leq b_h$, $1 \leq h \leq k$.
Let $a_i = b_1 \wedge \cdots \wedge b_{k} \wedge a_1 \wedge \cdots \wedge a_m$. Then $\alpha(V) \leq a_i \in \bba$. By the monotonicity of $\beta_i$ and the antitonicity of $\gamma_i$ it follows that
\[
\beta_i(a_i) \leq \gamma_i(a_i).
\]
Now, letting $u = a_1 \wedge \cdots \wedge a_n$, we have $\alpha(V) \leq u \in \bba$, and by the  monotonicity of the $\beta_i$ and the antitonicity of the $\gamma_i$ we get that
\[
\beta_i(u) \leq \gamma_i(u) \textrm{ for all } 1 \leq i \leq n.
\]

\paragraph{Proof of the Lefthanded Ackermann lemma for admissible assignments (Lemma \ref{Ackermann:Dscrptv:Left:Lemma})}
As in the previous lemma we will write $\beta_{i}$ and $\gamma_{i}$ for ${\beta_i}^V_p$ and ${\gamma_i}^V_p$, respectively. The implication from bottom to top follows by the antitonicity  of the $\beta_i$ and the monotonicity of the $\gamma_i$.

For the sake of the converse implication assume that ${\beta_i}^V_p(\alpha(V)) \leq {\gamma_i}^V_p(\alpha(V))$ for all $1 \leq i \leq n$. But $\alpha$ is syntactically open and (trivially) negative in $p$, hence by lemma \ref{Syn:Opn:Clsd:Appld:ClsdUp:Lemma}(2), $\alpha(V)\in O(\bbas)$, i.e.\ $\alpha(V) = \bigvee\{a \in \bba \mid a \leq  \alpha(V)\}$. Thus, for any $1 \leq i \leq n$, it is the case that
\[
{\beta_i}(\bigvee\{a \in \bba \mid a \leq  \alpha(V)\}) \leq {\gamma_i}(\bigvee\{a \in \bba \mid a \leq  \alpha(V)\}).
\]
Hence by lemma \ref{Esakia:Syn:Clsd:Opn:Lemma} (3) and (4)
\[
\bigwedge \{ {\beta_i}(a) \mid a \in \bba, a \leq  \alpha(V) \}  \leq \bigvee \{ \gamma_i(a) \mid a \in \bba,  a \leq  \alpha(V) \}.
\]
The proof now proceeds like that of lemma \ref{Ackermann:Dscrptv:Right:Lemma}.

\end{document}